\documentclass[12pt]{amsart}
\usepackage{pdfsync}
\usepackage{graphicx}
\usepackage{xcolor}
\usepackage{amssymb,amsmath,latexsym,amsfonts,amsthm}
\usepackage{epstopdf}
\usepackage[mathscr]{eucal}
\usepackage{stmaryrd}
\usepackage{verbatim} 
\DeclareMathOperator{\dom}{dom}

\DeclareMathOperator{\last}{last}

\begin{document}

\begin{abstract}
A relational  structure is {\em indivisible} if for every partition of its set of elements into two parts there exists an embedding of the structure into one of the parts of the partition. A relational structure is  {\em homogeneous} if every isomorphism of a finite induced substructure to a finite induced substructure extends to an automorphism.  This article establishes a necessary and sufficient condition for  Henson type, see \cite{Henson}, homogeneous structures to be indivisible.    
 \end{abstract}

\title{Colouring homogeneous structures}      
\author [N. Sauer]{Norbert Sauer}
\address{N. Sauer: University of Calgary, Department of Mathematics and Statistics, Calgary, Alberta, Canada T2N 1N4}
\email{nsauer@ucalgary.ca}

\date{January 04 2010 }
\maketitle

\newcommand{\snl} {\\ \smallskip}
\newcommand{\mnl}{\\ \medskip} 
\newcommand{\Bnl} {\\ Bigskip}

\newcommand{\edge}[1]{\makebox[22pt]{$\circ\mspace{-6 mu}\stackrel{#1}{-}\mspace{-6 mu}\circ$}}
\newcommand{\nedge}[1]{\makebox[29pt]{$\circ\mspace{-6 mu}\stackrel{#1}{\cdots\cdots}\mspace{-6 mu}\circ$}}

\newcommand{\Edge}[2]{\, \circ{{\stackrel{#1}{\sim}}_#2}\circ\, }  
\newcommand{\nEdge}[2]{\, \circ\negmedspace\stackrel{#1}{\cdots}_{#2}\negmedspace\circ\, }

\newcommand{\und}[1]{\underline{#1}}

\newcommand{\tr}{\, |\, }

\newcommand{\bnl}{\bigskip\noindent}
\newcommand{\rro}{\boldsymbol{\rho}}
\newcommand{\rrro}{\overline{\boldsymbol{\rho}}}
\newcommand{\bet}{\boldsymbol{\beta}}
\newcommand{\fsub}{\underset{finite}{\subset}}

\newcommand{\restrict}[1]{\mspace{-3mu}\mathbin{\downarrow}\mspace{-3mu} #1}
\newcommand{\Kat}{Kat\u{e}tov }
\newcommand{\Fra}{Fra\"{\i}ss\'e}
\newcommand{\str}[1]{\stackrel{#1}{\sim}}
\newcommand{\spe}{\mathrm{spec}}
\newcommand{\X}{{\underline X}}
\newcommand{\U}{{\underline U}}
\newcommand{\V}{{\underline V}}
\newcommand{\W}{{\underline W}}
\newcommand{\Y}{{\underline Y}}
\newcommand{\Z}{{\underline Z}}
\newcommand{\R}{{\underline R}}
\newcommand{\T}{{\underline T}}
\renewcommand{\S}{{\underline S}}
\renewcommand{\P}{{\underline P}}
\newcommand{\Q}{{\underline Q}}
\renewcommand{\L}{{\underline L}} 
\newcommand{\M}{{\underline M}}
\newcommand{\N}{{\underline N}}
\newcommand{\K}{{\underline K}}
\newcommand{\J}{{\underline J}}
\newcommand{\I}{{\underline I}}
\newcommand{\D}{{\underline D}}
\newcommand{\C}{{\underline C}}
\newcommand{\E}{{\underline E}}
\newcommand{\B}{{\underline B}}

\newcommand{\rem}[1]{\textcolor{blue}{#1}}

\newcommand{\End}{{\hskip 330pt $\vartriangle$}}        

\newcommand{\Jb}{\boldsymbol{\mathscr{I}}}

\newcommand{\upl}[1]{\sideset{ }{^{c}}{\operatorname{\mathit{{#1}}}}}
\newcommand{\lpl}[1]{\sideset{ }{_{\negmedspace c}}{\operatorname{\mathit{{#1}}}}}
\newcommand{\ipl}[1]{\sideset{ }{_{\negmedspace c}^{c}}{\operatorname{\mathit{{#1}}}}}

\newcommand{\bde}{\overline{\de}}
\newcommand{\dw}[1]{\ulcorner\negthickspace#1\negthickspace\urcorner}

\newcommand\cat{{{\, }^\frown}}
\newcommand\catt{{{}^{\frown *}}}

\newcommand{\concat}{%
  \mathord{
    \mathchoice
    {\raisebox{1ex}{\scalebox{.7}{$\frown$}}}
    {\raisebox{1ex}{\scalebox{.7}{$\frown$}}}
    {\raisebox{.7ex}{\scalebox{.5}{$\frown$}}}
    {\raisebox{.7ex}{\scalebox{.5}{$\frown$}}}
  }
}

\newcommand{\iiota}{\overline{\iota}}

\newtheorem{thm}{Theorem}[section]
\newtheorem*{thm*}{Theorem}
\newtheorem{lem}{Lemma}[section]
\newtheorem{ass}{Assumption}[section]
\newtheorem{defin}{Definition}[section]
\newtheorem{example}{Example}[section] 
\newtheorem{fact}{Fact}[section] 
\newtheorem{cond}{Condition}

\newtheorem{prop}{Proposition}[section]
\newtheorem{obs}{Observation}[section]
\newtheorem{cor}{Corollary}[section]
\newtheorem{sublem}{Sublemma}[section]
\newtheorem{claim}{Claim}
\newtheorem{question}{Question}[section] 
\newtheorem{note}{Note}[section]

\newtheorem{problem}{Problem}[section]
\newtheorem{remark}{Remark}[section]

\newcommand\con{\char'136{}}

\section{Preliminaries}\label{sect:prelim}

The reader of this article will need to have had some previous exposure to \Fra\  theory or homogeneous structures. See \cite{Fra} or \cite{Fraisse} or \cite{Hodges}. Even so,  some of the  necessary information about \Fra\  theory will be provided in the form of definitions or just stated as facts.   This  section, Section \ref{sect:prelim}, presents the notation together with some basic information on homogeneous structures. It is followed by an introduction, in  Section \ref{sect:intro}. Section \ref{sect:prelim} contains two subsections providing some facts which hold for all countable homogeneous structures, facts which might not have been widely noticed. But helpful for some of the arguments in this article.  Subsection \ref{subsect:uslemneed} is about a characterization of the images of self embeddings. This has appeared as Theorem 5.1  in \cite{Pos-Cop}. Subsection~\ref{subsect:gropversio} deals with some aspects of the automorphism groups of homogeneous structures. Essentially already discussed in \cite{Pos-Cop} and \cite{Siblings}. For convenience collected in Subsection  \ref{subsect:gropversio}. 

For a function $f$, let $f[S]=\{f(s)\mid s\in \dom(f)\cap S\}$. 
For any natural  number $n$ and any   tuple $\vec{a}=(a_0, a_1,a_2, \dots, a_{n-1})$ let $f(\vec{a})=(f(a_0),f(a_1), f(a_2), \dots,f(a_{n-1}))$. For a tuple $\vec{a}$  let $x\in \vec{a}$ or $X\subseteq \vec{a}$ or $\vec{a}\subseteq X$ mean that $x$ is an entry of $\vec{a}$ or that the set $S$ is a subset of the set of entries of $\vec{a}$ or that the set of entries of $\vec{a}$ is a subset of the set $X$ respectively.     Relational structures will usually be denoted by roman letters and their domains by math italic letters.   In this paper, we only consider relational structures and often call them just structures.  

We do not allow 0-ary relations. Let $\boldsymbol{L}$ be a relational language. An {\em $\boldsymbol{L}$-structure} will be a realization $\mathrm{A}$ of $\boldsymbol{L}$ in which   for every $n$-ary relation symbol $R$, if  $R_A(x_0,x_1,\dots, x_{n-1})$ is true then $x_i\not=x_j$ for all $i\not=j$.

Let $\boldsymbol{L}$ be a relational language.  A {\em monomorphism} of an $\boldsymbol{L}$-structure $\mathrm{A}$ to an $\boldsymbol{L}$-structure $\mathrm{B}$ is an injection $f: A\to B$ so that  for all relation symbols $R\in \boldsymbol{L}$ the relation $R_\mathrm{A}(\vec{a})$ implies the relation $R_\mathrm{B}(f(\vec{a}))$. If  the identity map on $A$ is a monomorphism then $\mathrm{A}$ is a {\em substructure} of $\mathrm{B}$.   A monomorphism of $\mathrm{A}$ to $\mathrm{B}$ is an {\em embedding} of $\mathrm{A}$ to $\mathrm{B}$ if the function $f^{-1}: f[A]\to A$ is also a monomorphism.   If $f$ is an embedding and onto then $f$ is an {\em isomorphism}. The structure $\mathrm{A}$ is an {\em induced substructure of $\mathrm{B}$} if the identity map on $A$ is an embedding. If $S\subseteq B$ then the {\em restriction} of $\mathrm{B}$ to $S$, denoted $\mathrm{B}_{\downarrow S}$,  is the structure $\mathrm{S}$ with domain $S$ for which the identity map on $S$ is an embedding of $\mathrm{S}$ to $\mathrm{B}$.  The class of finite structures which can be embedded into a relational structure $\mathrm{U}$ is the {\em age of\/ $\mathrm{U}$}.   If $f$ is an embedding of $\mathrm{A}$ to $\mathrm{B}$ then the restriction of $\mathrm{B}$ to $f[S]$ is a {\em copy} of $\mathrm{A}$ in $\mathrm{B}$. A  copy of $\mathrm{B}$, without specifying a target structure, will usually mean a copy of $\mathrm{B}$ in $\mathrm{B}$. 

\begin{defin}\label{defin:indiv}
A structure $\mathrm{U}$ is {\em indivisible} if for every colouring function $\boldsymbol{c}: U\to 2$ there exists a copy $\mathrm{C}$ of $\mathrm{U}$ in $\mathrm{U}$ for which the colouring function $\boldsymbol{c}$ is constant on $C$.  
\end{defin}

Let $\mathrm{U}$ be a relational structure and $F\subseteq U$. For $\{a,b\}\subseteq F$ and $a\not= b$ let $a\edge{F} b$  if there exists a relation $R$ of $\mathrm{U}$ and a tuple $\vec{x}$ with  $\{a,b\}\subseteq \vec{x}\subseteq F$ and with $R(\vec{x})$. Otherwise let $a \, \, \, \, \nedge{F} \, \, \, \,  b$.  Note that if $a \, \, \, \, \nedge{F}  \, \, \, \, b$ then $a\not=b$. The graph with $F$ as the set of vertices and the edge relation $\edge{F}$ is the {\em Gaifman graph} or {\em 2-section} of the structure $\mathrm{U}_{\downarrow F}$.   A structure $\mathrm{K}$ is {\em irreducible} if its Gaifman graph is complete.

Let $\boldsymbol{L}$ be a relational language.  A pair  $\mathrm{M}$ and $\mathrm{M}'$ of $\boldsymbol{L}$-structures with $M=B\cup C$ and $M'=B\cup C'$ and with with $B\cap C=B\cap C'=C\cap C'=\emptyset$ is an {\em  amalgamation instance} if $\mathrm{M}_{\downarrow B}=\mathrm{M}'_{\downarrow B}$. The $\boldsymbol{L}$-structure $\mathrm{N}$ with $N=B\cup C\cup C'$ and with 
\begin{enumerate}
\item $\mathrm{N}_{\downarrow B\cup C}=\mathrm{M}_{\downarrow B\cup C}$ and $\mathrm{N}_{\downarrow B\cup C'}=\mathrm{M}'_{\downarrow B\cup C'}$,
\item $x \, \, \, \, \nedge{N} \, \, \, \, y$ for all $x\in C$ and all $y\in C'$,
\end{enumerate}
is the {\em free amalgam} of the structures $\mathrm{M}$ and $\mathrm{M}'$.  A class $\mathfrak{A}$ of $\boldsymbol{L}$-structures is a {\em  free amalgamation age} if it is closed under induced substructures and if the free amalgam of every amalgamation instance of structures in $\mathfrak{A}$ is again a structure in $\mathfrak{A}$. 

It is often convenient to specify a class of structures $\mathfrak{A}$ closed under induced substructures by providing a ``boundary" of the class $\mathfrak{A}$. For example the class of simple graphs is given by forbidding loops and edges $(a,b)$ for which $(b,a)$ is not an edge. Within the class of simple graphs forbidding then further triangles specifies the class of triangle free graphs. In general, let $\mathfrak{A}$ be a class of finite relational structures in language $\boldsymbol{L}$  which is closed under induced substructures. Note that there exists then a set $\mathfrak{K}$ of structures, called the {\em boundary} of $\mathfrak{A}$, so that:
\begin{enumerate}
\item No structure in $\mathfrak{K}$ can be embedded into any other structure in $\mathfrak{K}$.
\item If $\mathrm{C}$ is a finite $\boldsymbol{L}$-structure then $\mathrm{C}\in \mathfrak{A}$ if and only if there is no structure $\mathrm{K}\in \mathfrak{K}$ which has an embedding into $\mathrm{C}$.
\item Up to isomorphisms the boundary of $\mathfrak{A}$ is unique.    
\end{enumerate}   

\begin{note}\label{note:freboundcop}
Let $\mathfrak{A}$ be a free amalgamation age. Then every structure in the boundary of\/ $\mathfrak{A}$ is irreducible. Conversly, if every structure in the boundary of\/ $\mathfrak{A}$ is irreducible then $\mathfrak{A}$ is a free amalgamation age. 
\end{note}

A structure $\mathrm{U}$ is  {\em homogeneous} if for every finite subset $A$ of $U$ and every embedding $f$ of $\mathrm{U}_{\downarrow A}$ into $\mathrm{U}$ there exists an automorphism $g$ of $\mathrm{U}$ which agrees with $f$ on $A$. It follows from the general \Fra \ theory that there exists for every countable free amalgamation age $\mathfrak{A}$ a countable homogeneous structure $\mathrm{U}$ whose age is equal to $\mathfrak{A}$.  A homogeneous structure $\mathrm{U}$ is a {\em free amalgamation homogeneous structure} if its age $\mathfrak{A}$ is a free amalgamation age. Free amalgamation homogeneous structures have been discovered as a class of countable homogeneous structures by W. Henson, see \cite{Henson}.  For example,  the class of finite simple graphs is a free amalgamation age. The countable homogeneous structure whose age is the class of all finite simple graphs is the {\em Rado graph}.  

 For an important example note that the class of $\mathrm{K}_n$-free finite graphs, that is the age of finite graphs with boundary $\{\mathrm{K}_n\}$,  has free amalgamation.   The {\em $\mathrm{K}_n$-free homogeneous graph} is the countable  homogeneous graph whose age is the free amalgamation age consisting of all the finite simple graphs which do not embed the complete graph on $n$ vertices.  The {\em random $k$-uniform homogeneous hypergraph is the countable homogeneous structure whose age is the class of all finite $k$-uniform hypergraphs.}

The next  Facts \ref{fact:uniquhomage},  \ref{fact:contemb} follow from the general \Fra\ theory, see \cite{Fra} .  In particular Fact \ref{fact:contemb}, the  {\em extension property} of homogeneous structures,  will be used repeatedly:

\begin{fact}\label{fact:uniquhomage}
Two countable homogeneous structures having the same age are isomorphic.
\end{fact}

\begin{fact}\label{fact:contemb}
Let $\mathrm{U}$ be a homogeneous structure and $\mathrm{A}$ an element of the age of $\mathrm{U}$ with  $S\subseteq A\cap U$ for which $\mathrm{U}_{\downarrow S}=\mathrm{A}_{\downarrow S}$. Then there exists an embedding $f: A\to U$ with $f(s)=s$ for all $s\in S$.  
\end{fact}

Let $\mathrm{G}$ be the group of automorphisms of a homogeneous structure $\mathrm{U}$. Then for $F$ a finite subset of $U$ let:
\[
\mathrm{G}_{F}:=\{g\in \mathrm{G}\mid \forall a\in F\, (g(a)=a)\}  
\]
The group $\mathrm{G}_{F}$ is a {\em stabilizer subgroup of $\mathrm{G}$}. 

A {\em type} $\T$ of $\mathrm{U}$ is a pair of the form $\langle F\tr x\rangle$  with $F$ a finite subset of $U$ and $x\in U\setminus F$. The set $F$ is the {\em sockel}, denoted by $\iota(\T)$,  of the type $\T$.  Two types $\langle F\tr x\rangle$ and $\langle E\tr y\rangle$ are {\em equal} if $F=E$ and if there exists a function $g\in \mathrm{G}_F$ with $g(x)=y$. The {\em typeset}, denoted by $\sigma(\T)$,  of the type $\T=\langle F\tr x\rangle$ is the set:
\[
\sigma(\T):= \{y\in U\setminus F\mid \exists g\in \mathrm{G}_F\, (g(x)=y)\}. 
\]
Note that two types $\S$ and $\T$ are equal if and only if $\iota(\S)=\iota(\T)$ and $\sigma(\S)=\sigma(\T)$. Typesets of different types with the same sockel are disjoint. 

 For example, let now $\mathrm{U}$ be the, say $\mathrm{K}_7$-free homogeneous graph. Let $F$ be a finite subset of $U$ for which the graph $\mathrm{U}_{\downarrow F}$ has clique number, say 5. Let $\T=\langle F\tr x\rangle$ be a  type of $\mathrm{U}$. Every isomorphism $f$ of a finite subset of $\sigma(\T)$ to a finite subset of $\sigma(\T)$ extends to an isomorphism fixing the set $F$ pointwise and hence to an automorphism in $\mathrm{G}_F$ which fixes $\sigma(\T)$ setwise. It follows that the graph  $\mathrm{U}_{\downarrow \sigma(\T)}$ is homogeneous.   Let $S\subseteq F$ be the set of points adjacent to $x$. If the graph $\mathrm{U}_{\downarrow S}$ has clique number, say 3,   then the age of the graph $\mathrm{U}_{\downarrow \sigma(\T)}$ is the class of all finite $\mathrm{K}_4$-free graphs. It follows that the graph $\mathrm{U}_{\downarrow \sigma(\T)}$ is then isomorphic to the $\mathrm{K}_4$-free homogeneous graph. Note that $\mathrm{G}_F$ restricted to $\sigma(\T)$ is the group of automorphisms of the graph $\mathrm{U}_{\downarrow \sigma(\T)}$.   Let now $\mathrm{U}$ be the random homogeneous $3$-uniform hypergraph. Any two two-element subsets of $U$ are then isomorphic.  Let $F\not=\emptyset$ be a finite subset of $U$ and $\T=\langle F\tr x\rangle$ be a type of $\mathrm{U}$.  Note that in this case there are isomorphisms of finite subsets of $\mathrm{U}_{\downarrow \sigma(\T)}$, in particular of two element subsets of $\mathrm{U}_{\downarrow \sigma(\T)}$,   which do not extend to an automorphism in $\mathrm{G}_F$. Nevertheless the hypergraph $\mathrm{U}_{\downarrow \sigma(\T)}$ is, in this case,  isomorphic to $\mathrm{U}$. This follows from Lemma \ref{lem:charercop}, the Pouzet Lemma, just as for the typesets of the Rado graph. In a similar way as discussed in the paragraph after Lemma~\ref{lem:charercop}. Also: Let now $\mathrm{U}$ be the 3-uniform homogeneous hypergraph whose age is the class of all finite 3-uniform hypergraphs  which do not contain an irreducible induced sub hypergraph on five vertices and six or more  hyperedges. Let $\T=\langle \{a,b\}\tr x\rangle$ be a type of $\mathrm{U}$ for which $\{a,b,x\}$ forms a hyperedge.  Let $y\in \sigma(\T)$ for which the sets $\{a,x,y\}$ and $\{b,x,y\}$ are hyperedges of $\mathrm{U}$. Then there is no $z\in \sigma(\T)$ for which the set $\{x,y,z\}$ forms a hyperedge of $\mathrm{U}$. On the other hand the structure $\mathrm{U}_{\downarrow \sigma(\T)}$ does contain hyperedges. It follows that the structure $\mathrm{U}_{\downarrow \sigma(\T)}$ is not homogeneous.

A {\em bundle of types} or just {\em bundle} is a set of types all of which are having the same sockel.  Note that the typesets of different types having the same sockel are disjoint. 



\begin{defin}\label{defin:groologo}
Let $\mathrm{G}$ be a subgroup of the symmetric group of a set $M$. Two $n$-tuples with entries in $M$ are equivalent if there exists a $g\in \mathrm{G}$ mapping one to the other. The  group $\mathrm{G}$  is {\em oligomorphic} if for every $n\in \omega$ the number of equivalence classes of the $n$-tuples is finite. A structure $\mathrm{M}$ is {\em oligomorphic} if its group of automorphisms is oligomorphic. 
\end{defin}

\begin{note}\label{note:iligolig}

A homogeneous structure $\mathrm{U}$ with group $\mathrm{G}$ of automorphisms is  oligomorphic if and only if for every finite $F\subseteq U$ the set of types with sockel $F$ is finite. 
\end{note}

 For a structure $\mathrm{A}$ let the {\em age} of $\mathrm{A}$, denoted $\rro(\mathrm{A})$, be the class of all finite structures   which have an embedding  into $\mathrm{A}$. For $S$ a subset of $U$ put $\rro(S)=\rro(\mathrm{U}_{\downarrow S})$ and call then $\rro(S)$ the {\em age} of $S$. Then $\rro(U)$ is the age of the structure $\mathrm{U}$. For $\T$ a type let $\rro(\T):=\rro(\sigma(\T))$ be the {\em rank} of the type $\T$. 


If the group of automorphisms of $\mathrm{U}$ is transitive then $\U$ is the type $\langle \emptyset \tr x\rangle$ for any element $x\in U$.  Implying that   the age of $\mathrm{U}$, that is $\rro(U)$,  is equal to $\rro(\U)$ because then $\sigma(\U)=U$.

 Put
 \[
\mathfrak{R}(\mathrm{U}):=\{\mathfrak{r}\mid \text{there exists a type $\T$ with $\rro(\T)=\mathfrak{r}$}\}.
\]

\begin{defin}\label{defin:ranklinear}
A homogeneous structure $\mathrm{U}$ is {\em rank linear} if the partial order $(\mathfrak{R}(\mathrm{U});\subseteq)$ of ranks is a linear order. 
\end{defin}
For example, the homogeneous graph whose boundary is $\{\mathrm{K}_5\}$,    for $\mathrm{K}_5$ being the complete graph on five vertices, is rank linear. The linear order of ranks consists of the classes of  $\mathrm{K}_5$-free finite graphs, the  $\mathrm{K}_4$-free finite graphs, the  $\mathrm{K}_3$-free finite graphs and the class of graphs containing no edges.  Let $\boldsymbol{L}$ be the language with only relation symbol $R$ and this one has arity one. If $x$ is an element of a $\boldsymbol{L}$-structure then $R(x)$ or $\neg R(x)$. Let $\mathrm{U}$ be the free amalgamation homogeneous structure with just this one relation symbol. The structure $\mathrm{U}$ has two types, $\langle \emptyset \tr x\rangle$ and $\langle \emptyset \tr y\rangle $ with $R(x)$ and $\neg R(y)$ whose ranks are not linearly ordered under $\subseteq$. Of course $\mathrm{U}$ is not indivisible. For a less trivial example see Example~\ref{ex:twtriange}.

If a homogeneous structure is rank linear then its automorphism group acts transitively on $U$. For otherwise let $x$ and $y$ be two elements in different transitivity classes of $U$. Because $\mathrm{U}$ is homogeneous the structures $\mathrm{X}=\mathrm{U}_{\downarrow \{x\}}$ and and $\mathrm{Y}=\mathrm{U}_{\downarrow \{y\}}$ are not isomorphic. Then $\T=\langle \emptyset\tr x\rangle$ and $\S=\langle\emptyset\tr y\rangle$ are two types with $\mathrm{X}\in \rro(\T)\setminus \rro(\S)$ and  with $\mathrm{Y}\in \rro(\S)\setminus \rro(\T)$. Hence if a homogeneous structure $\mathrm{U}$ is rank linear then   the linear order of ranks has a maximum, namely the class $\rro(\U)$.


\noindent
\textbf{Important:}
For this article, unless explicitly otherwise stated, the group of automorphisms of the homogeneous structures $\mathrm{U}$ under consideration acts transitively on the set $U$ of elements of $\mathrm{U}$.

\begin{defin}\label{defin:ageindiv}
As defined in  \cite{Fra} a structure $\mathrm{M}$ is {\em age indivisible} if for every partition $(S,P)$ of $M$ the age of $\mathrm{M}_{\downarrow S}$ is equal to the age of $\mathrm{M}$ or   the age of $\mathrm{M}_{\downarrow S}$ is equal to the age of $\mathrm{M}$. 
\end{defin}

\subsection{ A useful Lemma}\label{subsect:uslemneed} A completely written out proof of the following Lemma can be found as Theorem 5.1 of \cite{Pos-Cop}. 

\begin{lem}\label{lem:charercop}  [Pouzet Lemma]
Let $\mathrm{U}$ be a countable homogeneous structure. A subset $C$ of $U$ induces a copy of\/ $\mathrm{U}$ if and only if $\sigma(\T)\cap C\not=\emptyset$ for every type $\T$ of $\mathrm{U}$ with $\iota(\T)\subseteq C$. 
\end{lem}
\begin{proof}
The condition is necessary because a copy is an isomorphic structure and a typeset with sockel in the copy is a typeset of the copy and of the homogeneous structure. The standard back and forth argument provides a proof that the condition of the Lemma is sufficient.
\end{proof}

Using this Lemma there is a simple argument for partitions of the Rado graph, which unfortunately does not seem to extend to the general case. The type of  argument provided in Section \ref{sect:guide} is the one which is extended in this article for the proof of the main Theorem. First observe that if $\T=\langle F\tr x\rangle $ is a type of the Rado graph then $\sigma(\T)$ induces a copy of the Rado graph. For let $\S=\langle E\tr y\rangle$ be a type with $E\subseteq \sigma(\T)$. Then $\sigma(\S)\cap \sigma(\T)\not=\emptyset$ beccause there exists a graph with set of vertices $F\cup E$ and an additional vertex $z$ which is attached to $F$ just like $x$ and attached to $E$ just like $y$. Using Fact  \ref{fact:contemb} this graph has an embedding $h$ into $U$ with $h(v)=v$ for all $v\in F\cup E$. Then $h(z)\in \sigma(\S)\cap \sigma(\T)$. It follows from Lemma \ref{lem:charercop} that $\sigma(\T)$ induces a copy of $\mathrm{U}$. 

The next theorem  is due to Peter Cameron, see \cite{Cameronrand}. 
\begin{thm}\label{thm:exactindRad}[P. Cameron]
 Given a partition of the Rado graph into red and blue vertices. If the set of red vertices does not induce a copy of the Rado graph then the set of blue vertices induces a copy of the Rado graph. 
 \end{thm}
\begin{proof}
If the set of red vertices does not induce a copy of the Rado graph then according to Lemma \ref{lem:charercop} there exists a type $\T$ of $\mathrm{U}$ with every vertex in $\iota(\T)$ being red and every vertex in $\sigma(\T)$ being blue.  If the set of blue vertices does not induce a copy of the Rado graph then according to Lemma \ref{lem:charercop} there exists a type $\S$ of $\mathrm{U}$ with every vertex in $\iota(\S)$ being blue and every vertex in $\sigma(\S)$ being red.   Leading to a contradiction because any two typesets of the Rado graph with disjoint sockels have a point in common. Which can easily be verified using the extension property.
\end{proof}

\subsection{Subgroups  of the symmetric group $\mathfrak{S}(U)$  for a  countable set $U$} \label{subsect:gropversio}

 Let  $\mathrm{G}$ be a subgroup of the symmetric group $\mathfrak{S}(U)$ of $U$. An {\em embedding for $\mathrm{G}$} is an injection $f$ of $U$ into $U$ so that for every finite subset $A$ of $U$ there exists a function $g\in \mathrm{G}$ with $f(a)=g(a)$ for all $a\in A$. The image of such an embedding is a {\em copy} for $\mathrm{G}$.  The composition of embeddings for $\mathrm{G}$ is an embedding for $\mathrm{G}$.  The group $\mathrm{G}$ is {\em closed} if every bijective embedding for $\mathrm{G}$ is an element of $\mathrm{G}$. The {\em closure} of $\mathrm{G}$ is $\mathrm{G}$ together with all the bijective embeddings of $\mathrm{G}$.  The following assertions are not difficult to derive. See \cite{Cameronoligo} and or \cite{Pos-Cop}  for additional details.  The closure of $\mathrm{G}$ is a closed subgroup of $\mathfrak{S}(U)$. A function $f: U\to U$ is an embedding for $\mathrm{G}$ if and only if $f$ is an embedding of the closure of $\mathrm{G}$.  There exists a homogeneous structure $\mathrm{U}$ with automorphism group $\mathrm{G}$ if and only if the group $\mathrm{G}$ is closed.  If $\mathrm{U}$ is a homogeneous structure and $\mathrm{G}$ is the group of automorphisms of $\mathrm{U}$ then $f$ is an embedding of $\mathrm{U}$ into $\mathrm{U}$ if and only if $f$ is an embedding for $\mathrm{G}$. 

The group $\mathrm{G}$ is {\em indivisible} if for every partition $(S_0,S_1)$ of $U$ there exists an $i\in 2$ and an embedding $f$ for the group $\mathrm{G}$ with $f[U]\subseteq S_i$. It is not difficult to verify that a group $\mathrm{G}$ is indivisible if and only if the closure of $\mathrm{G}$ is indivisible.   Also, if $\mathrm{U}$ is a homogeneous structure and $\mathrm{G}$ is the group of automorphisms of $\mathrm{U}$ then the homogeneous structure $\mathrm{U}$ is indivisible if and only if the group $\mathrm{G}$ is indivisible. A subset of $U$ is a copy of $U$ if and only if it is a copy for $\mathrm{G}$.  If $A$ is a finite subset of $U$ and $g$ a function in the closure of $\mathrm{G}$ then there exists a function $f\in \mathrm{G}$ with $f(a)=g(a)$ for all $a\in A$. This observation implies that for most of the definitions  and results in Section \ref{sect:necess}, except for Theorem \ref{thm:maincecess},   the group need not be closed. 

\begin{defin}\label{defin:equdefage}
For $S\subseteq U$ let $\rro_\mathrm{G}(S)$, the {\em $\mathrm{G}$-age} of $S$,  be the set of finite subsets $A$ of $U$ for which there exists a function $g\in \mathrm{G}$ with $g[A]\subseteq S$. 
\end{defin}

 Note that $\rro_\mathrm{G}(S)=\rro_{\overline{\mathrm{G}}}(S)$ for every subset $S$ of $U$. Let $\overline{\mathrm{G}}$ be the closure of $\mathrm{G}$ and let $\mathrm{U}$ be a homogeneous structure  which has  $\overline{\mathrm{G}}$ as group of automorphisms.    Let $S$ and $T$ be two subsets of $U$.  Then $\rro(S)\subseteq \rro(\T)$ if and only if $\rro_\mathrm{G}(S)\subseteq \rro_\mathrm{G}(\T)$. For let $\rro(S)\subseteq \rro(\T)$ and a set  $A\in \rro_\mathrm{G}(S)$. Then there exists  a function $g\in \mathrm{G}$ with $g[A]\subseteq S$ and hence $\mathrm{U}_{\downarrow(g[A])}\in \rro(S)\subseteq \rro(T)$. That is there exists a subset $B$ of $T$ and an isomorphism $h$ of the structure $\mathrm{U}_{\downarrow(g[A])}$ to the structure $\mathrm{U}_{\downarrow B}$. Hence there exists a function $f\in \mathrm{G}$ with $f\circ g[A]\subseteq T$. Implying $A\in \rro_\mathrm{G}(T)$. The argument in the other direction is similar. Hence:

 \begin{fact}\label{fact:equdefage}
 If $\mathrm{U}$ is a homogeneous structure on $U$ with $\overline{\mathrm{G}}$ as group of automorphisms then the partial orders $(\{\rro(S)\mid S\subseteq U\};\subseteq )$ and $(\{\rro_{\overline{\mathrm{G}}}(S)\mid S\subseteq U\};\subseteq )$ are isomorphic under a function which associates $\rro(S)$ with $\rro_{\overline{\mathrm{G}}}(S)$ for every subset $S$ of $U$. 
 
$A\in \rro_\mathrm{G}(S)$ if and only if\/ $\mathrm{U}_{\downarrow A}\in \rro(S)$. $\mathrm{A}'\in \rro(S)$ if and only if there exists a subset $\rro_\mathrm{G}(S)\ni A \subseteq U$ and an isomorphism $f$ of $\mathrm{A}'$ to $\mathrm{U}_{\downarrow A}$.  
\end{fact}
\noindent
This fact justifies the not quite correct but convenient identification of $\rro$ with $\rro_\mathrm{G}$ occurring within some of the arguments in this article. For example:

\begin{defin}\label{defin:agindtyp}
A subset $S$ of $U$ is {\em $\rro$-age indivisible} if $\rro(S_0)=\rro(S)$ or $\rro(S_1)=\rro(S)$  for every partition $(S_0,S_1)$ of $S$. A subset $S$ of $U$ is {\em $\rro_\mathrm{G}$-age indivisible} if $\rro_\mathrm{G}(S_0)=\rro_\mathrm{G}(S)$ or $\rro_\mathrm{G}(S_1)=\rro_\mathrm{G}(S)$  for every partition $(S_0,S_1)$ of $S$. 
\end{defin}

\begin{lem}\label{lem:ageintypp}
A subset $S$ of $U$ is $\rro$-age indivisible if and only if it is $\rro_\mathrm{G}$-age indivisible.
\end{lem}
\begin{proof}
Let $S$ be $\rro$-age indivisible and $(S_0,S_1)$ be a partition of $S$. Assume that $\rro(S_0)=\rro(S)$. It follows from Fact \ref{fact:equdefage} that then $\rro_\mathrm{G}(S_0)=\rro_\mathrm{G}(S)$. The converse can be seen similarly.
\end{proof}
On account of Lemma \ref{lem:ageintypp} we will just write age indivisible. The  group $\mathrm{G}$ is {\em age indivisible} if the set $U$ is age indivisible. 

Let  a  type  of the group $\mathrm{G}$ be defined as  in Section \ref{sect:prelim}. Note that $\langle F\tr x \rangle$ is a type of $\mathrm{G}$ if and only if it is a type of $\overline{\mathrm{G}}$.  Hence a type of any of the homogeneous structures whose group of automorphisms is $\overline{\mathrm{G}}$. The sockel $\iota(\T)$ of a type $\T$ and the typeset $\sigma(\T)$ of a type $\T$ and of rank $\rro(\T)$ of a type $\T$ are defined  as in Section \ref{sect:prelim}. A type $\T$ of $\mathrm{G}$ is {\em age indivisible} if the set $\sigma(\T)$ is age indivisible. It follows from Fact \ref{fact:equdefage} that a homogeneous structure $\mathrm{U}$ is rank linear if and only if it is rank linear using the $\rro_\mathrm{G}$ definition  of age instead of the $\rro$ definition. Making it possible to define rank-linear subgroups of the symmetric  group. Then, a homogeneous structure $\mathrm{U}$ is rank linear if and only its automorphism group is rank linear.   

Indivisibility questions do not arise in the context of group actions on finite sets. For other questions and notions arising in the context of group actions on infinite sets which do not arise in the context of group actions on finite sets, for example notions in connections with various topologies on the elements of the group and on the copies for the group,   see \cite{Pos-Cop} or \cite{Siblings}.

\section{Introduction}\label{sect:intro}

It is well known  and not very difficult to verify that the Rado graph is indivisible. See Theorem \ref{thm:exactindRad}.   An outline of a proof whose general structure serves as a template for the sufficient part of Theorem \ref{thm:main}, will be provided  in Section \ref{sect:guide}. Whether the triangle free homogeneous graph $\mathrm{H}_3$ is indivisible had been an open problem of Erd\H{o}s. It has then been shown by Komj\'{a}th and   R\"{o}dl, see \cite{KomRo},  that the homogeneous graph $\mathrm{H}_3$    is indeed indivisible. Section \ref{sect:guide} provides some   indication    why,  deciding whether the homogeneous graph $\mathrm{H}_3$ is indivisible, is not easy. This then is used to justify the elaborate notions and arguments needed to settle the general case.   Subsequently to \cite{KomRo}  it was shown in \cite{EZS1} that the $\mathrm{K}_n$-free simple homogeneous graphs $\mathrm{H}_n$ are indivisible and in \cite{EZS2} that a generalization, from complete graphs to tournaments, of this result  holds. It follows from a general result of \cite{SaCan}  for binary homogeneous structures that: A countable, oligomorphic, binary and free amalgamation  homogeneous relational structure  is indivisible if and only if it is rank linear.\\
\noindent
The main result of this article is:
\begin{thm}\label{thm:main}
A countable, oligomorphic and free amalgamation  homogeneous relational structure  is indivisible if and only if it is rank linear.
\end{thm}
\vskip -5pt
\noindent
The non binary case is  substantially more difficult than the binary one.

This article falls into the general area of structural partition theory. Which recently has seen two remarkable results. The solution of the very difficult open problem dealing with  partitions of  the set of finite substructures of the  $\mathrm{K}_n$-free  homogeneous graphs by Natasha Dobrinen, see \cite{Dobrinen}. A considerably more intricate situation  than dealt with in this article, in which only partitions of the set of elements are investigated. The article, see \cite{Honza},  by Jan Hubi\v{c}ka and Jaroslav Ne\v{s}et\v{r}il constituting, after a long history,  in some sense  a complete solution  concerning Ramsey classes. The introduction of \cite{Dobrinen},  provides an excellent exposition on the background and history of the partition theory of homogeneous structures. The work on partitions of sets of other structures than single elements has a long history. The topic of colouring other structures than single elemts is completely absent from this article.  I can,  in this introduction, not  do any better than Dobrinen's  introduction in \cite{Dobrinen} and just direct the reader to it. Age indivisibility and the obvious generalization to partitions of substructures are via compactness directly connected to structural Ramsey theory results as in \cite{Honza}. A third interesting development is the application of big Ramsey degrees to topological dynamics, see Andy Zucker's article \cite{Zucker}. The introductions of \cite{Dobrinen}, \cite{Honza} and \cite{Zucker} provide a good overview of  structural partition theory in connection with homogeneous structures.

To prove that the conditions of Theorem \ref{thm:main} are sufficient to guarantee that the homogeneous structure is indivisible is by far the longest and most difficult part of this article. Theorem \ref{thm:manaaaod} states that: If a countable, oligomorphic and free amalgamation  homogeneous relational structure is rank linear then it  is indivisible. A large part of the notions introduced for and the arguments of the proof of 
Theorem~\ref{thm:manaaaod} are within the notational framework of group actions. The main exception is the use of Fact \ref{fact:contemb}. It is easier to deal with group actions than considerations of structural properties. An example of this are the results of Subsection \ref{subsec:weakindiv2}. Proving Theorem \ref{thm:wiekstruncor} staying with structure considerations would require to use notions like $\mathrm{U}_{\downarrow \mathfrak{c}^{-1}(i+1)}$ while the same notion within group actions only requires to refer to the set $ \mathfrak{c}^{-1}(i+1)$. The necessary conditions are proven within the framework of group actions and then translated to structural statements. We obtained in Section \ref{sect:necess}:

\begin{thm}\label{them:necesssary}
Let $\mathrm{G}$ be a subgroup of the symmetric group of a countable infinite set $U$. If  there are two age indivisible types $\T$ and $\S$ having infinite typesets such that $\rro(\T)\setminus \rro(\S)\not=\emptyset$ and $\rro(\S)\setminus \rro(\T)\not=\emptyset$ then the group $\mathrm{G}$ is divisible.  
\end{thm}
\vskip -3pt
\noindent
Which implied: 
\vskip -3pt
\noindent
\begin{thm}\label{thm:maincecess}
Let the age $\mathfrak{A}$ of the homogeneous structure $\mathrm{U}$ be a  free amalgamation class. If $\mathrm{U}$ is not rank linear then $\mathrm{U}$ is divisible.
\end{thm}
\vskip -3pt
\noindent
Giving together with  Theorem \ref{thm:manaaaod} the main theorem, Theorem  \ref{thm:main}. 

\vskip 5pt

To  obtain Theorem \ref{thm:manaaaod} just for  the binary case,  Section \ref{section: formedtyp}     and Section~ \ref{sect:aingind} can be omitted. Using the additional references to the binary case provided in the body of this article it should be fairly easy to extract an article dealing with  the binary case only. 

\vskip 2pt
\noindent
In Subsection \ref{subsec:weakindiv2} we established the following weaker version of indivisibility:
\begin{thm}\label{thm:wiekstruncor}
Let $\mathrm{U}$ be a free amalgamation homogeneous relational structure with a transitive automorphism group $\mathrm{G}$. Let $\mathfrak{A}$ be the age of $\mathrm{U}$.  Let  $\mathfrak{c}: U\to 2$ be a two colouring of $U$. If  $\rro(\mathrm{U}_{\downarrow (\mathfrak{c}^{-1}(i))})\not=\mathfrak{A}$ then the induced substructure    $\mathrm{U}_{\downarrow \mathfrak{c}^{-1}(i+1)}$   embeds a copy of $\mathrm{U}$. (Hence $\mathrm{U}$ is age indivisible.)
\end{thm}
That is if a subset of $U$ does not induce  some element of the age then the complement induces a copy of $\mathrm{U}$. Note that rank linear and oligomorphic are not required by Theorem \ref{thm:wiekstruncor}. Given a two colouring of $U$ it is only then difficult to find a monochromatic copy of $\mathrm{U}$ if  each of the colour classes induces all of the age of $\mathrm{U}$.

The countable homogeneous $k$-uniform hypergraph $\mathrm{U}$,  whose boundary is a complete $k$-uniform hypergraph on $n$ vertices, that is one for which every $k$-element subset is an edge, could of course  be considered to be a natural generalization of the $K_n$-free homogeneous graphs. It follows from Theorem \ref{thm:main} that those homogeneous structures $\mathrm{U}$ are indivisible. But  according to Lemma \ref{lem:singlrant} item (1) those homogenous structures $\mathrm{U}$ have the property that $|\mathfrak{R}(\mathrm{U})|=1$. Hence we do not have that interesting chain of ranks as for the $K_n$-free homogeneous structures.

The following Theorem \ref{thm:Knfrreegen}, proven in Subsection \ref{subsect:allirred},  describes  indivisible, homogeneous  hypergraphs which form, in some way,  another natural generalization of the $K_n$-free homogeneous graphs $\mathrm{H}_n$ to $k$-uniform hypergraphs. If $n<m$ then every $K_n$-free graph is also $K_m$-free. Implying that the homogeneous graphs $\mathrm{H}_n$ are rank linear.   Example \ref{ex:nelinenine} shows that forbidding only all irreducible 3-uniform hypergraphs on 9 vertices is not enough to characterize an age for which the corresponding homogeneous structure is rank linear. But the next theorem states that forbidding all irreducible 3-uniform hypergraphs on 9 or more vertices will produce a rank linear homogeneous hypergraph. 

\begin{thm}\label{thm:Knfrreegen}
Let $2\leq k\in \omega$ and let $n> k$. Let $\mathfrak{B}$ be the class of all irreducible k-uniform hypergraphs having at least $n$ vertices. Let $\mathfrak{A}$ be the age of all finite $k$-uniform hypergraphs which do not embed any one of the hypergraphs in $\mathfrak{B}$.  Then $\mathfrak{A}$ is a free amalgamation age. The  countable homogeneous structure $\mathrm{U}$ whose age is $\mathfrak{A}$  is rank linear and hence indivisible.  The linear order of the ranks of the types of $\mathrm{U}$ consists of $n-k+1$ elements. 
\end{thm}

Theorem \ref{thm:singlrant}, verified in Subsection \ref{subsect:onerank},  shows that if $|\mathfrak{R}(\mathrm{U})|=1$ then free amalgamation is not needed for indivisibility. 
\begin{thm}\label{thm:singlrant}
Let $\mathrm{U}$ be a countable, oligomorphic,  homogeneous relational  structure. If $|\mathfrak{R}(\mathrm{U})|=1$ then $\mathrm{U}$ is indivisible.
\end{thm}

\section{A guide towards the proof of Theorem \ref{thm:manaaaod}}\label{sect:guide}

Let $\mathrm{R}$ be the Rado graph. Let $\T=\langle F\tr x\rangle$ be a type of $\mathrm{R}$ with typeset $\sigma(\T)$. Then the structure $\mathrm{R}_{\downarrow \sigma(\T)}$ is isomorphic to the Rado graph. For if $\mathrm{A}$ is an element in the age of $\mathrm{R}$, that is $\mathrm{A}$ is a finite graph, we may assume that $A\cap R=\emptyset$. Otherwise we take a graph in the age disjoint from $R$ and isomorphic to $\mathrm{A}$. Let then  $\mathrm{M}$ be the graph with set of vertices $F\cup A$ for which $\mathrm{M}_{\downarrow F}=\mathrm{R}_{\downarrow F}$ and $\mathrm{M}_{\downarrow A}=\mathrm{A}$.  In addition for every element $a\in A$ exists a function  $f: F\cup \{x\}\to F\cup \{a\}$ with $f(x)=a$ and $f(v)=v$ for all $v\in F$, which is an isomorphism of $\mathrm{U}_{\downarrow F\cup \{x\}}$ to $\mathrm{M}_{\downarrow F\cup \{a\}}$. Then $\mathrm{M}$ is a graph, hence in the age of $\mathrm{R}$. According to Fact \ref{fact:contemb} there exists an embedding $h$ of $\mathrm{M}$ into $\mathrm{R}$ with $f(v)=v$ for all $v\in F$. Because the vertices in $A$ are attached to $F$ just as $x$ is attached to $F$, the embedding $h$ maps $A$ into $\sigma(\T)$. Hence the age of $\sigma(\T)$ is equal to the age of $\mathrm{R}$, namely the class of all finite simple graphs. Let $\mathrm{G}$ be the group of automorphisms of the Rado graph $\mathrm{R}$. Then the group $\mathrm{G}_F$ acts on $\sigma(\T)$ and is the group of automorphisms of $\mathrm{R}_{\downarrow \sigma(\T)}$. It follows from Fact \ref{fact:uniquhomage} that $\mathrm{R}_{\downarrow \sigma(\T)}$ is isomorphic to the Rado graph. Actually, if $\mathrm{\T}$ is a type of a binary homogeneous structure $\mathrm{U}$ with group of automorphisms $\mathrm{G}$ then $\mathrm{G}_{\iota(\T)}$ is the automorphism group of the structure $\mathrm{U}_{\downarrow \sigma(\T)}$. Which is not necessarily true in the non binary case. The argument above demonstrates how Fact \ref{fact:contemb} will be used throughout this article. 
 
 The Rado graph is indivisible, because:  Let $\mathfrak{c}$ be  a colouring, that is a partition, of its set of vertices $R$ into red and blue vertices. Let $u_0, u_1, u_2, u_3, \dots$ be an $\omega$-enumeration of $R$. We will try to construct a red copy of $\mathrm{R}$ in $\mathrm{R}$ in such a way that if we fail, then there exists a blue copy of $\mathrm{R}$ in $\mathrm{R}$. If there is no red vertex then clearly $R$ is a blue copy. Otherwise let $x_0$ be a red vertex. Assume we already found a sequence $x_0, x_1, \dots, x_{n-1}$  of red vertices so that the  the function $f_n: \{u_i\mid i\in n\} \to \{x_i\mid i\in n\}$ with $f(u_i)=x_i$ is an isomorphism of $\mathrm{U}_{\downarrow \{u_i\mid i\in n\}}$ to  $\mathrm{U}_{\downarrow \{x_i\mid i\in n\}}$. The Rado graph is homogeneous. Hence there exists a function $g\in \mathrm{G}$ which extends $h$ to $R$. If we can find a red vertex in the typeset of the type $\langle \{x_i\mid i\in n\}\tr g(u_n)\rangle$ the process continues. If not, all of the vertices in $\sigma(\langle \{x_i\mid i\in n\}\tr g(u_n)\rangle)$ are blue. Then we obtained our blue copy. If the process of constructing a sequence of red points never stops, we obtained a red copy. 
 
 We will try to extend this simple construction. But even in the next, seemingly just a bit more intricate case of the triangle free homogeneous graph $\mathrm{H}_3$,  we run into difficulties. Let us proceed as above. Enumerate the vertices of $H_3$ as $u_0, u_1, u_2, u_3, \dots$. Start to construct the isomorphic sequence $x_0, x_1, x_2, \dots$ of red vertices. If we are never stopped, then indeed we have a red copy. But assume we are stopped at step $n$. That is we have the  sequence $\{x_i\mid i\in n\}$ of red vertices.  Let the function $g$ be defined as in the case of the Rado graph and every vertex in the typeset $\sigma\langle \{x_i\mid i\in n\}\tr g(u_n)\rangle$ is blue. If the age of $\sigma\langle \{x_i\mid i\in n\}\tr g(u_n)\rangle$ is the class of all finite triangle free graphs, we obtained a blue copy. But assume there is a vertex, say $u_j$, which is adjacent to $u_n$. Then $x_j$ would be adjacent to $g(u_n)$. Implying that the set $\sigma\langle \{x_i\mid i\in n\}\tr g(u_n)\rangle$ does not contain two adjacent vertices. For otherwise there is a triangle. In which case we did not obtain a blue copy. We might try to rectify the construction of the red sequence by choosing the red points such that looking ahead, there is no such typeset in the way,  obstructing the red sequence. But then we should have a proof that if we can not do this, there exists a blue copy. 
 
 What does it mean, to look ahead? Well, if we are at stage $\{x_i\mid i\in n\}$ we have to investigate the bundle, that is the set, of all types, say $\mathcal{T}$,  with sockel $\{x_i\mid i\in n\}$. Because, for every one  of those types  $\T\in \mathcal{T}$ at   some later stage,  we will have to choose a vertex in the typeset of some successor $\S$ of $\T$.  This successor depends on the choices of red vertices until a vertex in $\T$ has to be chosen.  The typesets of those types form a partition of the set $H_3\setminus\{x_i\mid i\in n\}$. If we have then to conclude that the colouring of those typesets is such that we can not find a red copy going on from this stage of the construction, we have to go to a previous step and select another red vertex which leads to a more promising bundle of types. The process starts with the first vertex we choose. How do we know that there is such a good start for red?  If not there should then be a good first vertex for blue.  The fact that there exists such a good first choice for red or blue is established in Section~\ref{sect:basvertthm}. 
 
Let $\mathrm{U}$ be a countable homogeneous structure. The example of the triangle free graph should indicate that in order to extend the argument for the indivisibility of the Rado graph to $\mathrm{U}$ we should have the following tools available: A notion of a ``good" bundle of types,  Section~\ref{sect:basvertthm}. An understanding of how bundles and in particular  good bundles grow by extending the sockel by one additional element, Section \ref{sect:bundles}.  Bundles being sets of types. Hence understanding them necessitates to understand and develop properties of types, Section \ref{sect:types}. Which,  for binary homogeneous structures would be a good plan. 

Unfortunately if the homogeneous structure is not binary we need to do some additional work. There are two reasons for this. Let $\T$ and $\S$ be two types of $\mathrm{U}$ for which $\sigma(\T)\cap \sigma(\S)\not=\emptyset$. Then if $\mathrm{U}$ is binary there exists a single type $\X$ with $\iota(\X)=\iota(\T)\cup \iota(\S)$ and with $\sigma(\X)=\sigma(\T)\cap \sigma(\S)$. If $\mathrm{U}$ is not binary then there exist potentially several different types, all with the same sockel, $\iota(\T)\cup \iota(\S)$, so that the union of their typesets is equal to $\sigma(\T)\cap \sigma(\S)$. The age of every copy is the same as the age of $\mathrm{U}$. Hence we can not afford to decrease the age during the construction process.  For the construction we need then to be able to pick out of those types one whose typeset has the same age as the set $\sigma(\T)\cap \sigma(\S)$. Leading to age partition problems. The situation is actually more complicated because it occurs in the course of the construction in the context of bundles of types. Leading to age partition problems in the case of non transitive automorphism groups. This situation is addressed in Section \ref{sect:aingind}. The other difficulty in the non binary case is caused by the action of the groups $\mathrm{G}_{\iota(\T)}$ for the automorphism group $\mathrm{G}$ and the types of $\mathrm{U}$. Let $\T$ be a type of $\mathrm{U}$. Let $A$ and $B$  be two finite subsets  of $\sigma(\T)$ for which $\mathrm{U}_{\downarrow A}$ is isomorphic to $\mathrm{U}_{\downarrow B}$. Because $\mathrm{U}$ is homogeneous there exists function $g\in \mathrm{G}$, the group of automorphisms of $\mathrm{U}$, with $g[A]=B$. Even better, if $\mathrm{U}$ is binary there exists then a function $f\in \mathrm{G}_{\iota(\T)}$ with $f[A]=B$. But this, for the construction important function, does not necessarily exist if $\mathrm{U}$ is not binary. For more details about this see Lemma \ref{thm:exactindRad} and Example \ref{ex:strmelmeld}.   In order to rectify this we needed to single out special easy to deal with types, the so called ``formed types" of Section \ref{section: formedtyp}. There we developed the important properties of those formed types. 

Assume we  arrived during the course of the construction at a ``good" bundle, say $\mathcal{C}$. That is the sockel of the bundle has the correct properties as in the construction for the Rado graph and in addition the bundle of types with this sockel is ``good".   Extending the sockel by an additional element creates a new bundle, $\mathcal{B}$. But in general this new bundle will not be ``good". The types of the  bundle $\mathcal{C}$ have to be refined in such a way that we arrive at a good bundle $\mathcal{D}$ having the same ranks as the types in $\mathcal{B}$. Then the two bundles $\mathcal{B}$ and $\mathcal{D}$ have to be ``meshed correctly" to arrive at an extension of the bundle with the extended sockel and which is ``good".  

After doing all of this preparatory work the construction of a red or blue copy along the lines of the example for the Rado graph happens in Section \ref{sect:constru1}. There we obtain Theorem \ref{thm:manaaaod} : Every oligomorphic, rank linear, free amalgamation homogeneous structure $\mathrm{U}$ is indivisible. The condition  oligomorphic, except for Lemma \ref{lem:erefx234},   is used the first time in Section \ref{sect:constru1} in the prove of Theorem  \ref{thm:manaaaod}. For this proof Lemma \ref{lem:existastolog} is used and the assumption of Lemma \ref{lem:existastolog} is that the set of types of the bundles dealt with is finite. Lemma \ref{lem:erefx234} is not used again until in the proof of Theorem  \ref{thm:manaaaod}. The proof of this Lemma requires, just in the non binary case, Corollary  \ref{cor:agindfreambund}. All of the bundles dealt with in Section~ \ref{sect:aingind} are finite. Lemma \ref{lem:existastolog}, Lemma \ref{lem:jonxtC} and the proof of Theorem \ref{thm:manaaaod} make use of Theorem \ref{thm:rmsfreeh} which depends on Theorem  \ref{thm:fin game} which in turn depends on Lemma \ref{lem:gammaplI}. Lemma \ref{lem:gammaplI} uses the condition of Definition \ref{defin:mho-game} that the order $(\mathfrak{R};\subseteq)$ is linear. The condition  rank linear is not used before Section \ref{sect:basvertthm}.

\section{Types in  homogenous  structures}\label{sect:types}

Let $\mathrm{U}$, with a countable  domain $U$, be a free amalgamation  relational  homogeneous  structure in a relational language $\boldsymbol{L}$. The    age of $\mathrm{U}$ will be   denoted by $\mathfrak{A}$ and the group of  automorphisms of $\mathrm{U}$ will be denoted by $\mathrm{G}$. The group $\mathrm{G}$ acts transitively on $U$.

A type $\S$ is a {\em successor} of a type $\T$ if $\iota(\T)\subseteq \iota(\S)$ and if $\sigma(\S)\subseteq \sigma(\T)$. If $\S$ is a successor of $\T$ then $\T$ is a {\em predecessor of $\S$}. Note that if $\S$ is  a successor of $\T$ then there exists for every $x\in \sigma(\S)$ and every $y\in \sigma(\T)$ a function $g\in \mathrm{G}_{\iota(\T)}$ with $g(x)=y$. The type $\S$ is an {\em $E$-successor of $\T$} if $\S$ is a successor of $\T$ and $\iota(\S)=\iota(\T)\cup E$ and if $E\cap \iota(\T)=\emptyset$. A type $\S$ is a {\em refinement} of a type $\T$ if it is a successor of $\T$ and if $\rro(\S)=\rro(\T)$. A type $\S$ is an {\em $\mathfrak{r}$-restriction} of a type $\T$ if it is a successor of $\T$ and if $\rro(\S)=\mathfrak{r}$.

\begin{defin}
Let $\T$ be a type and $E\subseteq \iota(\T)$. The predecessor $\R$ of $\T$ with $\iota(\R)=E$ will be denoted by $\T_{\downarrow E}$. 
\end{defin}

Two types $\T=\langle T\tr x\rangle$ and $\S=\langle S\tr y\rangle$ of $\mathrm{U}$ are  {\em compatible} if there exists a type $\R=\langle T\cap S\tr z\rangle$, the {\em common predecessor of $\T$ and $\S$},  with $\sigma(\T)\cup \sigma(\S)\subseteq \sigma(\R)$. That is the types $\T$ and $S$ are compatible if for every $x\in \sigma(\T)$ and every $y\in \sigma(\S)$ there exists a function $g\in \mathrm{G}_{T\cap S}$ with $g(x)=y$. 
The  types $\T=\langle T\tr x\rangle$ and $\S=\langle S\tr y\rangle$  are in {\em free position} if they are compatible and if $t\, \, \, \, \nedge{S\cup T} \, \, \, \, s$  for all $t\in T\setminus S$ and all $s\in S\setminus T$.

Let $\T=\langle T\tr x\rangle$ and $\S=\langle S\tr y\rangle$ be two types in free position. A type  $\R=\langle T\cup S\tr z\rangle$ for which $z\in \sigma(\T)\cap \sigma(\S)$ and for which  $t\, \, \, \, \, \nedge{\{z\}\cup S\cup T}\, \, \, \, \, s$ for all $t\in T\setminus S$ and all $s\in S\setminus T$ is a {\em join} of $\T$ and $\S$. Note that any two joins of $\T$ and $\S$ are equal. The join of $\T$ and $\S$, if it exists, will be denoted by $\T\sqcap \S$. Note that $\sigma(\T)\supseteq \sigma(\T\sqcap \S)\subseteq \sigma(\S)$. The equality $\R=\T\sqcap \S$ will always imply that the types $\T$ and $\S$ are in free position and that $\R$ is the join of $\T$ and $\S$.

\begin{lem}\label{lem:tyinfA}
Let $\mathrm{U}$ be a free amalgamation homogeneous structure. Let $\T=\langle T\tr x\rangle$ be a type,  $F$ be a finite subset of $U$  and $\mathrm{A}\subseteq \sigma(\T)$ be finite. Then there exists a function $f\in \mathrm{G}_T$  with $f[A]\subseteq \sigma(\T)\setminus F$. 
\end{lem}
\begin{proof}
Let $\mathrm{M}$ be the $\boldsymbol{L}$-structure with $M=F\cup T\cup A'$ and with $\mathrm{M}_{\downarrow F\cup T} =\mathrm{U}_{\downarrow F\cup T}$ and with $s\, \, \, \, \nedge{M}\, \, \, \,  a$ for all $s\in F$ and all $a\in A'$ and   for which there exists an  embedding $h$ of $\mathrm{M}_{\downarrow T\cup A'}$ into $\mathrm{U}$  with $h(v)=v$ for all $v\in T$ and with $h[A']=A$. The structure $\mathrm{M}$ is the free amalgam of $M_{\downarrow T\cup A'}$ and $\mathrm{M}_{\downarrow F\cup T}$, hence   $\mathrm{M}\in \mathfrak{A}$. Then there exists  an embedding $k$ of $\mathrm{M}$ into $\mathrm{U}$ with $k(v)=v$ for all $v\in F\cup T$. Note that $k[A']\subseteq \sigma(\T)$ and $k[A']\cap F=\emptyset$. We put   $f$ to be an extension of the function $k\circ h^{-1}$ to an automorphism of $\mathrm{G}$. 
\end{proof}
\begin{cor}\label{cor:tyinfA}
Let the age $\mathfrak{A}$ of\/ $\mathrm{U}$ be a free amalgamation age. The set $\sigma(\T)$ is infinite for every type $\T$ of $\mathrm{U}$. 
\end{cor}

\begin{lem}\label{lem:commeldstrm}
Let $\mathrm{U}$ be a  free amalgamation homogeneous structure. Let $\B=\langle B\tr b\rangle$ be a type of $\mathrm{U}$ with a $C$-successor $\T=\langle B\cup C\tr x\rangle$  and an  $S$-successor $\S=\langle B\cup S\tr y\rangle$  such  that $C\cap S=\emptyset$ and  $\T$ and $\S$ are in free position. 

If for some finite set $\emptyset \not=A\subseteq \sigma(\S)$ there exists a function $h\in \mathrm{G}_B$ with $h[A]\subseteq \sigma(\T)$ then the type $\T\sqcap \S$ exists and there exists a function $f\in \mathrm{G}_{B\cup S}$ with $f[A]\subseteq \sigma(\T\sqcap \S)$. 
\end{lem}
\begin{proof}
Let $A$ be finite and $\emptyset \not=A\subseteq \sigma(\S)$ and  $h\in \mathrm{G}_B$ with $h[A]\subseteq \sigma(\T)$. Let $\mathrm{M}=\mathrm{U}_{\downarrow B\cup C\cup h[A]}$. Let $l$ with $\dom(l) = B\cup S\cup A$ be the function with $l(v)=v$ for all $v\in B\cup S$ and with $l(a)=h(a)$ for all $a\in A$.   Let $\mathrm{N}$ be the $\boldsymbol{L}$-structure for which $N=B\cup S\cup h[A]$ and  so that   the function $l$ is an isomorphism of $\mathrm{U}_{\downarrow  B\cup S\cup A}$ to $\mathrm{N}$. The structures $\mathrm{M}$ and $\mathrm{N}$ form an amalgamation instance with $M\cap N=B\cup h[A]$. Let $\mathrm{R}$ be the free amalgam of $\mathrm{M}$ and $\mathrm{N}$.  It follows from Fact \ref{fact:contemb} that there exists an embedding $k$ of the structure $\mathrm{R}$ into $\mathrm{U}$ with $k(v)=v$ for all elements $v\in B\cup C\cup S$. It follows from $k{v}=v$ for all $v\in \iota(\T)$ that $k[h[A]]\subseteq \sigma(\T)$. The function $k\circ l$ is an isomorphism of $\U_{B\cup S\cup A}$ to $\mathrm{U}_{\downarrow B\cup S\cup k\circ l[A]}$. Let $f\in \mathrm{G}$ be an extension of the isomorphism $k\circ l$ to a function in $\mathrm{G}$.
\end{proof}

\begin{cor}\label{cor:commeldstrm}
Let $\mathrm{U}$ be a  free amalgamation homogeneous structure. Let $\B=\langle B\tr b\rangle$ be a type of $\mathrm{U}$ with a $C$-successor $\T=\langle B\cup C\tr x\rangle$  and an $S$-successor $\S=\langle B\cup S\tr y\rangle$  and  so that $\T$ and $\S$ are in free position. Then the type $\T\sqcap \S$ exists. In particular,  $\sigma(\T)\cap \sigma(\S)\not=\emptyset$. 
\end{cor}
\begin{proof}
Follows from the fact that for every $a\in \sigma(\B)$ there exists a function $h\in \mathrm{G}_B$ with $h(a)\subseteq \sigma(\T)$. 
\end{proof}

\begin{lem}\label{lem:agejoinST}
Let $\mathrm{U}$ be a  free amalgamation homogeneous structure. Let $\B=\langle B\tr b\rangle$ be a type of $\mathrm{U}$ with a $C$-successor $\T=\langle B\cup C\tr x\rangle$  and an  $S$-successor $\S=\langle B\cup S\tr y\rangle$  such  that $C\cap S=\emptyset$ and  $\T$ and $\S$ are in free position. 

Then $\rro(\T\sqcap \S)=\rro(\T)\cap \rro(\S)$. 
\end{lem}
\begin{proof}
The age $\rro(\T\sqcap \S)\subseteq \rro(\T)\cap \rro(\S)$ because $\sigma(\T\sqcap \S)\subseteq \sigma(\T)\cap \sigma(\S)$. 

Let the structure $\mathrm{A}\in \rro(\T)\cap \rro(\S)$ and $A$ its set of elements. Let $f_\T$ be an embedding of $\mathrm{A}$ into $\sigma(\T)$ and $f_\S$ be an embedding of $\mathrm{A}$ into $\sigma(\S)$. Let $\mathrm{id}_\T$ be the identity map on $B\cup C$ and let  $\mathrm{id}_\S$ be the identity map on $B\cup S$.  Let $\mathrm{M}$ be the $\boldsymbol{L}$-structure for which $M=S\cup B\cup C\cup A$ so that $\mathrm{M}_{\downarrow (S\cup B\cup C)}=\mathrm{U}_{\downarrow (S\cup B\cup C)}$ and $f_T\cup \mathrm{id}_\T$ is an embedding of $\mathrm{M}_{\downarrow (A\cup B\cup C)}$ into $\mathrm{U}$ and  $f_\S\cup \mathrm{id}_\S$ is an embedding of $\mathrm{M}_{\downarrow (A\cup B\cup S)}$ into $\mathrm{U}$ and $t    \, \, \, \, \nedge{M} \,  \, \, \, \,   s$ for all $t\in C$ and all $s\in S$. The structure $\mathrm{M}$ is an element of the age of $\mathrm{U}$ because every structure in the boundary of $\mathrm{U}$ is irreducible. It follows from Fact \ref{fact:contemb} that there exists an embedding $h$ of $\mathrm{M}$ into $\mathrm{U}$ with $h(x)=x$ for all elements $s\in B\cup C\cup s$. Implying that $h[A]\subseteq \sigma(\T\sqcap \S)$.  

\end{proof}

\begin{defin}\label{defin:additiontyp}
The equation $\R=\T+\S$ will always imply that the types $\T$ and $\S$ are in free position, that $\iota(\T)\cap \iota(\S)=\emptyset$ and that $\R=\T\sqcap \S$. 
\end{defin}

\begin{lem}\label{lem:trmeldcan}
Let $\T=\langle T\tr y\rangle$ and $\S=\langle S\tr x\rangle$ be two types and $\R=\T+\S$.   Then $\rro(\R)=\rro(\T)\cap \rro(\S)$ and there exists for every finite $A\subseteq \sigma(\T)$ with $\mathrm{U}_{\downarrow A}\in \rro(\S)$ a function $g\in \mathrm{G}_T$ with $g[A]\subseteq \sigma(\R)$. 
\end{lem}
\begin{proof}
$\rro(\T\sqcap \S)=\rro(\T)\cap \rro(\S)$ follows from Lemma \ref{lem:agejoinST}. Let $A\subseteq \sigma(\T)$ be finite. Because $\mathrm{U}_{\downarrow A}\in \rro(\S)$ there exists a function $h\in \mathrm{G}$ with $h[A]\subseteq \sigma(\S)$. Because $\mathrm{G}=\mathrm{G}_\emptyset$ there exists a function $h\in \mathrm{G}_\emptyset$ with $h[A]\subseteq \sigma(\S)$. It follows from Lemma \ref{lem:commeldstrm} that there exists a function $g\in \mathrm{G}_T$ with $g[A]\subseteq \sigma(\T\sqcap \S)$. 
\end{proof}

Remember here that if  $\Y=\langle E\cup X\tr y\rangle$ is an $E$-successor of $\X=\langle X\tr x\rangle$ then $E\cap X=\emptyset$. 

\begin{defin}\label{defin:freetyd}
The type $\X=\langle X\tr x\rangle$ is a {\em free type} if $x\, \, \, \, \nedge{\{x\}\cup X}\, \, \, \,  a$  for all elements $a\in X$. An $E$-successor $\Y=\langle E\cup X\tr y\rangle$ of a type $\X=\langle X\tr x\rangle$ is a {\em free $E$-successor} of $\X$  if  $y\, \, \, \, \, \, \nedge{\{y\}\cup E\cup X}\, \, \, \, \, \,  e$ for all elements $e\in E$. (A free type $\X=\langle X\tr x\rangle$ is then a free $X$-successor of the type $\U=\langle \emptyset \tr b\rangle $ for $b\in U$.)  
\end{defin}

\begin{lem}\label{lem:freeduch}
Let $\mathrm{U}$ be a  free amalgamation homogeneous structure. For every type $\X=\langle X\tr x\rangle$ and finite set $E\subseteq U$ with $E\cap X=\emptyset$, there exists a unique free $E$-successor $\Y=\langle E\cup X\tr y\rangle$ of the type $\X$. If $\Y$ is a free $E$-successor of $\X$  and $A$ a finite subset of $\sigma(\X)$ then there exists a function $f\in \mathrm{G}_X$ with $f[A]\subseteq \sigma(\Y)$. Hence  
$\rro(\Y)=\rro(\X)$. (In particular then if $\X$ is a free type then $\rro(\X)=\rro(\U)=\mathfrak{A}$.)   
\end{lem}
\begin{proof}
If $\langle E\cup X\tr y\rangle$ and $\langle E\cup X\tr z\rangle$ are free $E$ successors of $\X$ then the function $k: E\cup X\cup \{y\}\to E\cup X\cup \{z\}$ with $k(v)=v$ for all $v\in E\cup X$ and $k(y)=z$ is an embedding of $\mathrm{U}_{(\downarrow E\cup X\cup \{y\})}$ to $\mathrm{U}_{(\downarrow E\cup X\cup \{y\})}$. Implying that the types $\langle E\cup X\tr y\rangle$ and $\langle E\cup X\tr z\rangle$ are equal.  Hence if there exist a free $E$-successor of $\X$ it is unique. 

Let $A\not=\emptyset$ be a finite subset of $\sigma(\X)\setminus E$ and $\mathrm{M}$ the structure with $M=A\cup X\cup E$ and so that: 
\begin{enumerate}
\item $\mathrm{M}_{\downarrow(A\cup X)}=\mathrm{U}_{\downarrow(A\cup X)}$ and $\mathrm{M}_{\downarrow(E\cup X)}=\mathrm{U}_{\downarrow(E\cup X)}$.
\item  $a \, \, \, \, \,  \nedge{\mathrm{M}} \, \, \, \, \, e$ for all elements $a\in A$ and  $e\in E$. 
\end{enumerate}
The structure $\mathrm{M}$ is an element of  $\mathfrak{A}$.  It follows from Fact \ref{fact:contemb} that there exists an embedding $h$ of $\mathrm{M}$ into $\U$ with $h(v)=v$ for all elements $v\in X\cup E$. The function $h$ is an embedding of $\mathrm{U}_{\downarrow X\cup A}$ into $\mathrm{U}$.  Let $f\in \mathrm{G}$ with $f(v)=h(v)$ for all $v\in X\cup A$. Then $f\in \mathrm{G}_X$. Because $h$ is an embedding of $\mathrm{M}$ into $\mathrm{U}$ we have $h(a)\, \, \,  \, \, \, \, \,  \nedge{E\cup X\cup h[A]} \, \, \, \, \, \, \, \,  e$ for all elements $a\in A$ and  $e\in E$. Hence for $a\in A$ is the type  $\Y=\langle E\cup X\tr h(a)\rangle$ a free $E$-successor of $\X$ and $f[A]\in \sigma(\Y)$. 
\end{proof}

\begin{lem}\label{lem:freecoppr}
Let $\mathrm{U}$ be a  free amalgamation homogeneous structure and $\X=\langle X\tr x\rangle$ the free type with sockel $X$. Then $\sigma(\X)$ induces a copy of $\mathrm{U}$. 
\end{lem}
\begin{proof}
According to Lemma \ref{lem:charercop} it suffices  to prove that if $\T$ is a type of $\mathrm{U}$ with $\iota(\T)\subseteq \sigma(\X)$ then $\sigma(\T)\cap \sigma(\X)\not=\emptyset$. Let $y$ be some element not in $U$ and $\mathrm{M}$ the structure with $M=X\cup\iota(\T)\cup \{y\}$ for which exists an embedding $f$ of $\mathrm{M}_{\downarrow \iota(\T)\cup \{y\}}$ into $\mathrm{U}$ with $f(y)=x$ and  $f(v)=v$ for all $v\in \iota(\T)$.  Also $u \, \, \, \, \, \, \, \, \, \nedge{X\cup \iota(\T)\cup \{y\}}\, \, \, \, \, \, \,\, \,   y$ for all $u\in X$. This structure $\mathrm{M}$ is an element of the age of $\mathrm{U}$. Hence it follows from Fact \ref{fact:contemb} that there exists an embedding $h$ of $\mathrm{M}$ into $\mathrm{U}$ with $h(v)=v$ for all $v\in X\cup \iota(\T)$. Then $h(y)\in \sigma(\T)$ and in $\sigma(\X)$. 
\end{proof}

\begin{lem}\label{lem:binstrmeld}
Let $\mathrm{U}$ be a countable binary homogeneous structure and $\T$ be a type of $\mathrm{U}$.   Let $A$ be a finite subsets of  $\sigma(\T)$ and let $g\in \mathrm{G}$ with $g[A]\subseteq \sigma(\T)$.  Then there exists a function $f\in \mathrm{G}_{\iota(\T)}$ with $f(a)=g(a)$ for all $a\in A$.  

Let $\S$ be a successor of $\T$ and $A$ a finite subset of $\sigma(\T)$ with $\mathrm{U}_{\downarrow A}\in \rro(\S)$. Then there exists a function $f\in \mathrm{G}_{\iota(\T)}$ with $f[A] \subseteq \sigma[S]$. 
\end{lem}
\begin{proof}
For every $a\in A$ and every $x\in \iota(\T)$ are the pairs $(x,a)$ and $(x,g(a))$ in the same binary relation. Implying that the function $h: \iota(\T)\cup A\to \iota(\T)\cup g[A]$ with $h(a)=g(a)$ for all $a\in A$ and $h(x)=x$ for all $x\in \iota(\T)$ is an isomorphism. Let $f$ be the extension of $h$ to a function in $\mathrm{G}$. 

$\mathrm{U}_{\downarrow A}\in \rro(\S)$, hence there exists a function $g\in \mathrm{G}$ with $g[A]\subseteq \sigma(\S)$. Implying that there exists a function $f\in \mathrm{G}_{\iota(\T)}$ with $f[A] \subseteq \sigma[S]$. 
\end{proof}

\begin{remark}\label{remark:meld}
This nice ``melding" property proven in Lemma \ref{lem:binstrmeld} for the binary case does not necessarily hold in general. See Example \ref{ex:strmelmeld}. 
\end{remark}

\begin{cor}\label{cor:freambinmeld}
Let $\mathrm{U}$ be a countable binary homogeneous structure and $\T$ be a type of $\mathrm{U}$. Then the structure $\mathrm{U}_{\downarrow \sigma(\T)}$ is a free amalgamation homogeneous structure with  $\mathrm{G}_{\iota(\T)}$ as group of automorphisms.
\end{cor}

Let $\T=\langle F|x\rangle$ be a type and $f\in \mathrm{G}$. If $\langle F|x\rangle=\langle F|y\rangle$, (that is if there is a function $g\in \mathrm{G}_F$ with $g(x)=y$,) then $\langle f[F]|f(x)\rangle =\langle f[F]|f(y)\rangle$ because $f\circ g\circ f^{-1}\in \mathrm{G}_{f[F]}$ and $f\circ g\circ f^{-1}(f(x)=f(y))$. It follows that $\sigma\langle f[F]|f(x)\rangle=f[\sigma\langle F|x\rangle]$ and we can define:

\begin{defin}\label{defin:typeim}
Let $\T=\langle F|x\rangle$ be a type and $f\in \mathrm{G}$. Then $f[\T]:=\langle f[F]|f(x)\rangle$. 
\end{defin}

\begin{note}\label{note:typeim}
Let $\T$ be a type and $f\in \mathrm{G}$. Then $\sigma(f[\T])=f[\sigma(\T)]$ and hence $\rro(f[\T])=\rro(\T)$. Let $g\in \mathrm{G}$ with $g(x)=f(x)$ for all $x\in \iota(\T)$ then $g[\T]=f[\T]$. Let $\T$ and $\S$ be two types and $f\in \mathrm{G}$ then $f[\sigma(\T)\cap \sigma(\S)]=f[\sigma(\T)]\cap f[\sigma(\S)]$. If $A\subseteq \sigma(\T)$ then $f[A]\subseteq \sigma(f[\T])$. 
\end{note}

\begin{lem}\label{lem:restriction}
Let $\mathrm{U}$ be a  free amalgamation homogeneous structure.  Let $\X=\langle X\tr x\rangle$ be a type with $\rro(\X)\supseteq \mathfrak{r}\in \mathfrak{R}$. Then there exists a type $\Z=\langle Z\tr z\rangle$ for which $\rro(\X+\Z)=\mathfrak{r}$.    
\end{lem}
\begin{proof}
Let $\X'$ be the free type with sockel $X$. Because $\mathfrak{r}\in \mathfrak{R}$ there exists a type $\Z'=\langle Z'\tr z'\rangle$ with $\rro(\Z')=\mathfrak{r}$.  There exists an automorphism  $f\in \mathrm{G}$ with $f[Z']\subseteq \sigma(\X')$ because the age of a free type is equal to the age of $\mathrm{U}$ according to Lemma \ref{lem:freeduch}. Then $\rro(f[\Z'])=\mathfrak{r}$ and and the types $\X$ and $f[\Z']$ with $X\cap f[Z']=\emptyset$ are in free position. It follows from Lemma \ref{lem:trmeldcan} that $\rro(\X+f[\Z'])=\rro(\X\sqcap f[\Z'])=\rro(\X)\cap \rro(f[\Z'])=\rro(\X)\cap \mathfrak{r}=\mathfrak{r}$. Put $\Z=f[\Z']$. 
\end{proof}

\section{Formed types} \label{section: formedtyp}

Let $\mathrm{U}$, with a countable  domain $U$, be a free amalgamation relational  homogeneous  structure in a relational language $\boldsymbol{L}$. The    age of $\mathrm{U}$ will be   denoted by $\mathfrak{A}$ and the group of  automorphisms of $\mathrm{U}$ will be denoted by $\mathrm{G}$. The group $\mathrm{G}$ acts transitively on $U$. The set of types of $\mathrm{U}$ will be denoted by $\mathscr{T}(\mathrm{U})$ and $\U=\langle \emptyset\tr u\rangle$ for some element $u\in U$.

\begin{defin}\label{defin:plusopre}
 A successor $\R$ of $\T$ is of the {\em  form $\T+$} if there exists a set $E$ for which $\R=\T+ \R_{\downarrow E}$.  (Then $E=\iota(\R)\setminus\iota(\T)$.)

A successor $\R$ of $\T$ is of the {\em  form $\T\dasharrow$}  if there exists a set $E$ for which $\R$ is the free $E$-successor of $\T$. (Then $E=\iota(\R)\setminus\iota(\T)$.)
\end{defin}

\begin{defin}\label{defin:floorbr}
A successor $\R$ of a type  $\T$ is of the form $\T[F]$ if $F\subseteq \iota(\R)\setminus\iota(\T)$ and if  there exists a set $E$ for which   $\R=\R_{\downarrow \iota(\T)\cup F}\sqcap \R_{\downarrow \iota(\T)\cup E}$.  The set $F$ is the {\em interior} between the matched pair of brackets $[$ and $]$.
\end{defin}
That is if an $F$-successor $\S$ of a type $\T$ and an $E$-successor $\R$ of $\T$ are in free position  then the type $\S\sqcap \R$ is of the form $\T[F]$ and also of the form $\T[E]$.

\begin{defin}\label{defin:bracksequ}
A finite sequence $\mathscr{F}=\{x_i\mid i\in n\in \omega\}$  of symbols in the set $\{+, \dasharrow, [, ]\}$  is a {\em form sequence} if the brackets $[$ and $]$ are matched in the usual sense and if every matched pair of brackets has a non empty interior, that is a left bracket is not immediately followed by a right bracket. The empty sequence is a form sequence. 

 An {\em entry} of a sequence is a pair consisting of a symbol appearing in the sequence together with an enumeration index. For $\mathscr{F}\not=\emptyset$ the expression $\last(\mathscr{F})$ denotes the last symbol  of the sequence $\mathscr{F}$, that is the symbol of the entry having the largest index. 
\end{defin}

Let $\mathscr{F}=\{x_i\mid i\in n\in \omega\}\not=\emptyset$ be a form sequence. If $x_{n-1}\in \{+, \dasharrow\}$ let $\mathscr{F}^{(-)}$ be the form sequence obtained from $\mathscr{F}$ by removing  $x_{n-1}$ from the sequence. The  symbol $x_{n-1}$  can not be a left bracket. If  $x_{n-1}$ is a right bracket  then the form sequence $\mathscr{F}$ has the shape $\mathscr{F}=\mathscr{E}\concat [ \mathscr{H}]$ where $\mathscr{E}$ and $\mathscr{H}$ are form sequences and the left bracket $[$ of $\mathscr{E}\concat [ \mathscr{H}]$ is the matching bracket to the last right bracket of $\mathscr{F}$. The form sequence $\mathscr{E}$ might be empty but $\mathscr{H}$ is not empty. The form sequence $\mathscr{F}^{(-)}$ is obtained from $\mathscr{F}$ by removing the last right bracket as well as its matching left bracket. That is $\mathscr{F}^{(-)}=\mathscr{E}\concat\mathscr{H}$. 

If $\phi$ is a function with domain the entries in $\mathscr{F}$  then $\phi^{(-)}$ denotes the restriction of $\phi$ to the entries of $\mathscr{F}^{(-)}$.   Let $\phi$ be a function with domain the entries of $\mathscr{F}$ for which $\phi(x)$ is  a set for every entry $x$ of $\mathscr{F}$. Then $\phi$ is  {\em form preserving} if it maps matching brackets to the same set and otherwise different entries to disjoint sets.   Note that if $\phi$ is form preserving then $\phi^{(-)}$ is a form preserving function with domain the entries of $\mathscr{F}^{(-)}$.  Let then also $\bigcup_{\mathscr{F}}\phi$ denote the union of the sets $\phi(x)$ with $x$ an entry of $\mathscr{F}$. 

Next we define inductively on the length of form sequences the notion that a successor  of a type $\T$ has the ``form" $\T\mathscr{F}$. 

Let $\T$ be a type with successor $\R$ and $\mathscr{F}=\{x_i\mid i\in n\in \omega\}$ a form sequence indexed by $n\in \omega$. The type  $\R$ as a successor  of $\T$ has the  form $\T\mathscr{F}$ if there exists  a form preserving function $\phi$ which associates with every entry in the sequence $\mathscr{F}$ a subset of $\iota(\T)\setminus\iota(\R)$ so that $\bigcup_{\mathscr{F}}\phi=\iota(\R)\setminus \iota(\T)$ and  the following   conditions hold:
\begin{enumerate}
\item The type $\T$ as a successor of $\T$ has the form $\T\emptyset$. 

\item If $x_{n-1}=+$ then the type $\R':=\R_{\downarrow \iota(\R)\setminus\phi(x_{n-1})}$ has the  form $\T\mathscr{F}^{(-)}$ with a form preserving function $\phi^{(-)}$. The type $\R$ as a successor of $\R'$ has the form $\R'+$ with $\R=\R'+\R_{\downarrow \phi(x_{n-1})}$ and  the form preserving function which maps    $x_{n-1}$ to $\phi(x_{n-1})$. 

\item If $x_{n-1}=\dasharrow$ then the type $\R':=\R_{\downarrow \iota(\R)\setminus\phi(x_{n-1})}$ has the  form $\T\mathscr{F}^{(-)}$ with a form preserving function $\phi^{(-)}$. The type $\R$ as a successor of $\R'$ has the form $\R'\dasharrow$ with $\R$ being the free $\phi(x_{n-1})$-successor of $\R'$  and  the form preserving function which maps $x_{n-1}$ to $\phi(x_{n-1})$. 

\item If $x_{n-1}=]$ and $\mathscr{F}=\mathscr{E}\concat[\mathscr{H}]$ so that $\mathscr{F}^{(-)}=\mathscr{E}\concat \mathscr{H}$ then the type $\R':=\R_{\downarrow \iota(\R)\setminus\phi(x_{n-1})}$ has the  form $\T\mathscr{F}^{(-)}$ with a form preserving function $\phi^{(-)}$. For $F=\iota(\R)\setminus\phi(x_{n-1})=\bigcup_{\mathscr{H}}\phi$ and $E=\phi(x_{n-1})$ has the type $\R$ as a successor of $\R'$  the form $\R'[F]$ with $\R=\R_{\downarrow \iota(\T)\cup F}\sqcap \R_{\downarrow \iota(\T)\cup E}$.  (Note that $\R_{\downarrow \iota(\T)\cup F}=\R'$ because $\iota(\T)\cup F=\iota(\R)\setminus\phi(x_{n-1})$.)
\end{enumerate}

Hence: If a successor  $\R$ of a type $\T$  has form  $\T\mathscr{F}$ with a form preserving function $\phi$ and $\P=\R+\S$ then the successor $\P$ of $\R$ as  a successor of $\T$ has form $\T\mathscr{F}\concat+$. The extension of $\phi$ which maps this last $+$ to $\iota(\S)$ will be a form preserving function for the successor $\P$ of $\T$. If a successor  $\R$ of a type $\T$  has form  $\T\mathscr{F}$ with a form preserving function $\phi$ and $\P$ is a free $E$-successor of $\R$  then the successor $\P$ of $\R$ as  a successor of $\T$ has form $\T\mathscr{F}\concat\dasharrow$. The extension of $\phi$ which maps this last $\dasharrow$ to $\iota(\S)$ will be a form preserving function for the successor $\P$ of $\T$. If any successors $\S$ of a type $\T$ and a successor  $\R$ of a type $\T$ having form  $\T\mathscr{F}$ with a form preserving function $\phi$ are in free position then the type $\P=\R\sqcap S$ has form $\T[\mathscr{H}]$ as a successor of $\T$. The extension of $\phi$ which maps this last $]$ as well as its matching $[$ to $\iota(\S\setminus \iota(\T))$ will be a form preserving function for the successor $\P$ of $\T$.

\begin{lem}\label{lem:succsuccform}
Let $\R$ having form $\T\mathscr{F}$  be a successor of $\T$ with form preserving function $\phi$.  Let $\S$ having form $\R\mathscr{H}$ be a successor of $\R$ having form preserving function $\psi$.  Then $\S$ has form $\T\mathscr{F}\concat\mathscr{H}$ as a successor of $\T$ for which $\phi\cup \psi$ is a form preserving function.
\end{lem}
\begin{proof}
By induction on the number of entries of $\mathscr{H}$.
\end{proof}

\begin{lem}\label{lem:movestrsucc}
Let $\S$ having form $\T\mathscr{F}$  be a successor of $\T$ with form preserving function $\phi$.    Then for every finite subset $A$ of $\sigma(\T)$ with $\mathrm{A}:=\mathrm{U}_{\downarrow A}\in \rro(\S)$ there exists a function $g\in \mathrm{G}_{\iota(\T)}$ with $g[A]\subseteq \sigma(\S)$. 
\end{lem}
\begin{proof}
 Induction on   the number of entries of $\mathscr{F}$. The Lemma clearly holds for $\mathscr{F}=\emptyset$.  
 
 The type $\R:=\S_{\downarrow \iota(\S)\setminus\phi(x_{n-1})}$ has the  form $\T\mathscr{F}^{(-)}$. Because  $\mathrm{U}_{\downarrow A}\in \rro(\S)\subseteq \rro(\R)$ and $\mathscr{F}^{(-)}$ has fewer entries than $\mathscr{F}$ there exists a function $f\in \mathrm{G}_{\iota(\T)}$ with $f[A]\subseteq \sigma(\R)$. 
 
 Let the last symbol of $\mathscr{F}$ be a $+$. The type $\S$ is the disjoint join of $\R$ and $\S_{\downarrow \phi(x_{n-1})}$ because  $\S=\R+\S_{\downarrow \phi(x_{n-1})}$ according to Item (2) above. Both types $\R$ and $\S_{\downarrow \phi(x_{n-1})}$ are predecessors of the type $\S$ implying that the structure $\mathrm{A}\in \rro(\R)$ and in $\rro(\S_{\downarrow \phi(x_{n-1})})$. Hence it follows from Lemma \ref{lem:trmeldcan} that there exists a function $h\in \mathrm{G}_{\iota(\R)}$ with $h[f[A]]\subseteq \sigma(\S)$. Note that $h\circ f\in\mathrm{G}_{\iota(\T)}$. 
 
Let the last symbol of $\mathscr{F}$ be a $\dasharrow$. The type $\S$ is the  free $\phi(x_{n-1})$-successor of $\R$  according to Item (3) above. It follows from Lemma \ref{lem:freeduch} that there exists a function $h\in \mathrm{G}_{\iota(\R)}$ with $h[f[A]]\subseteq \sigma(\S)$.

Let the last symbol of $\mathscr{F}$ be a $]$ and let $\mathscr{F}=\mathscr{E}\concat[\mathscr{H}]$ so that $\mathscr{F}^{(-)}=\mathscr{E}\concat \mathscr{H}$. For $F:=\iota(\S)\setminus(\iota(\T)\cup\phi(x_{n-1}))=\bigcup_{\mathscr{H}}\phi$ and $E=\phi(x_{n-1})$ has the type $\S$  as a successor of $\R$ according to Item (4) above  the form $\T[F]$ with $\S=\S_{\downarrow \iota(\T)\cup F}\sqcap \S_{\downarrow \iota(\T)\cup E}$. Note that $\R=\S_{\downarrow \iota(\T)\cup F}$ because $\iota(\T)\cup F=\iota(\S)\setminus\phi(x_{n-1})$. Both types $\R$ and $\S_{\downarrow \iota(\T)\cup E}$ are predecessors of the type $\S$ implying that  $\mathrm{A}\in \rro(\R)$ and  $\mathrm{A}\in\rro(\S_{\downarrow \iota(\T)\cup E})$. Hence it follows from Lemma \ref{lem:commeldstrm} that there exists a function $h\in \mathrm{G}_{\iota(\R)}$ with $h[f[A]]\subseteq \sigma(\S)$.  

\end{proof}

\begin{defin}\label{defin:formedtypes}
A type $\T$ of $\U$ is a {\em formed type} if there exist a form sequence $\mathscr{F}$ for which the type $\T$ as a successor of $\U$ has the form $\U\mathscr{F}$. Let $\boldsymbol{\mathscr{S}}$ denote the set of formed types of $\mathrm{U}$. 
\end{defin}

\section{Bundles of types}\label{sect:bundles}

Let $\mathrm{U}$, with a countable  domain $U$, be a free amalgamation  relational, homogeneous  structure in a relational language $\boldsymbol{L}$.  For this section the    age of $\mathrm{U}$ will be   denoted by $\mathfrak{A}$ and the group of  automorphisms of $\mathrm{U}$ will be denoted by $\mathrm{G}$. Except for Lemma \ref{lem:existastolog} the Lemmata in this section do not require the structure $\mathrm{U}$ to be oligomorphic.

\begin{defin}\label{defin:bundle}
A {\em bundle} is a set $\mathcal{B}$  of types of\/ $\mathrm{U}$ so that $\iota(\S)=\iota(\T)$ for all types  $\{\S,\T\}\subseteq \mathcal{B}$. We put $\iota(\mathcal{B})=\iota(\S)$ for some type $\S\in\mathcal{B}$ and put $\sigma(\mathcal{B})=\bigcup_{\S\in \mathcal{B}}\sigma(\S)$ and $\rro(\mathcal{B})=\rro(\sigma(\mathcal{B}))$.   For $g\in \mathrm{G}$ let $g[\mathcal{B}]$ be the bundle $\{g[\B]\mid \B\in \mathcal{B}\}$.   
\end{defin}

\begin{defin}\label{defin:succbundle}
A bundle $\mathcal{B}$ is a {\em successor} of a bundle $\mathcal{C}$ if:
\begin{enumerate}
\item $\iota(\mathcal{B})\supseteq \iota(\mathcal{C})$.
\item $\B_{\downarrow\iota(\mathcal{C})}\in \mathcal{C}$ for every type $\B\in \mathcal{B}$.
\end{enumerate}  
\end{defin}

For a function $\alpha$ with domain $A$ and $b\in \alpha[A]$  let:\\
 $\alpha^{-1}(b)=\{a\in A\mid \alpha(a)=b\}$. 
\begin{defin}\label{defin:wellplsubs}
Let  $\mathcal{B}$ be a successor of a bundle $\mathcal{C}$. A function $\alpha$ for which $\dom(\alpha):=A$ is a finite subset of $\sigma(\mathcal{C})$ and with $\alpha[A]\subseteq \mathcal{B}$  is  {\em $\mathcal{B}$ conform}  if  for every type $\B\in \alpha[A]$:
\begin{enumerate}
\item $\alpha^{-1}(\B)\subseteq \sigma(\B_{\downarrow \iota(\mathcal{C})})$. 
\item $\mathrm{U}_{\downarrow \alpha^{-1}(\B)}\in  \rro(\B)$.
\end{enumerate}
\end{defin}

\begin{defin}\label{defin:wellplsucc}
A successor $\mathcal{B}$ of a bundle $\mathcal{C}$ is a {\em melding successor} of $\mathcal{C}$ if for every $\mathcal{B}$ conform function $\alpha$  there exists a function $g\in \mathrm{G}_{\iota(\mathcal{C})}$ with $g(a)\in \sigma(\alpha(a))$ for every element $a\in A$.  A successor $\S$ of a type $\T$ is a {\em melding successor} of $\T$ if for every finite set  $A\subseteq \sigma(\T)$ with $\U_{A}\in \rro(\S)$ there exists a function $g\in \mathrm{G}_{\iota(\T)}$ with $g[A]\subseteq \sigma(\S)$. 
\end{defin}
 Let $\S$ be a successor of a type $\T$ and let $A$ be a finite subset of $\sigma(\T)$. Then the only $\{\S\}$ conform function of $A$ to $\{\S\}$ is the function mapping every element $a\in A$ to $\S$. Then we will just say that the type $\S$ is a {\em melding successor} of the type $\T$ if the bundle $\{\S\}$ is a melding successor of the bundle $\{\T\}$. It follows from Lemma  \ref{lem:freeduch} that every free successor of a type $\T$, that is every successor of the form $\T\dasharrow$, is a melding successor of the type $\T$. In the binary case every successor of a type is a melding successor of that type according to Lemma \ref{lem:binstrmeld} and Remark \ref{remark:meld}. Clearly every melding successor of a melding successor of a type $\T$ is a melding successor of $\T$. More generally, we  will not need it but it might still be noteworthy to observe that a melding successor of a melding successor of a bundle $\mathcal{B}$ is a melding successor of that bundle $\mathcal{B}$.

\begin{lem}\label{lem:movestrsucc2}
Let $\S$ having  form $\T\mathscr{F}$ be a successor of a type $\T$.  Then $\S$ is a melding successor of $\T$. 
\end{lem}
\begin{proof}
Follows from Lemma \ref{lem:movestrsucc} and the definition of melding successor of a type. 
\end{proof}

Note that if $\mathcal{B}$ is a melding successor of the bundle $\mathcal{C}$ then every type $\B\in \mathcal{B}$ is a melding successor of the type $\B_{\downarrow \iota(\mathcal{C})}$. But for the converse,  we do not claim and it is actually not the case, that a successor $\mathcal{B}$ of $\mathcal{C}$  for which every type $\B\in \mathcal{B}$ is a melding successor of the type $\B_{\downarrow \iota(\mathcal{C})}$, is necessarily a melding successor of $\mathcal{B}$. We will  need the fact that the converse does hold under the following additional conditions. 

\begin{defin}\label{defin:neutralsucc}
If  $\mathcal{B}$, equipped with  a linear order $\preceq$,  is  a  successor of a bundle $\mathcal{C}$    and  $^\ast$ is a function of $\mathcal{B}$ to the set of finite subsets of $E:=\iota(\mathcal{B})\setminus \iota(\mathcal{C})$, then for  $\B\in \mathcal{\B}$ let $\B\uparrow$ be the set of types $\D\in \mathcal{B}$ with $\B\prec_{\not=} \D$ and let $(\B\uparrow^\ast)=\bigcup_{\D\in \B\uparrow}\D^\ast$. Let $\B\downarrow$ be the set of types $\D\in \mathcal{B}$ with $\B\succ_{\not=} \D$ and let $(\B\downarrow^\ast)=\bigcup_{\D\in \B\downarrow}\D^\ast$.

A successor $\mathcal{B}$ of a bundle $\mathcal{C}$ with $E:=\iota(\mathcal{B})\setminus \iota(\mathcal{C})$ is a $^\ast$- successor of $\mathcal{C}$ if ther exists a linear order $(\mathfrak{B};\preceq)$, the {\em $^\ast$ linear order,} and if   there exists a function $^*$, the {\em $^\ast$- successor function},  of the bundle $\mathcal{B}$ to the set of subsets of $E$ so that for all $\{\B,\D\}\subseteq \mathcal{B}$ with $\B\not=\D$:
\begin{enumerate}
\item $\B^*\cap \D^*=\emptyset$. 
\item $\B$ is the free $(\B\uparrow^\ast)$-successor  of  $\B_{(\downarrow \iota(\mathcal{B})\setminus (\B\uparrow^\ast))}$.  
\item  $\B_{(\downarrow \iota(\mathcal{B})\setminus (\B\uparrow^\ast))}$ is a melding successor of the type $\B_{\downarrow \iota(\mathcal{C})}$.
\end{enumerate}
 (It follows then, using Lemma  \ref{lem:freeduch} that $\B$ is a melding successor of the type $\B_{\downarrow \iota(\mathcal{C})}$.)
\end{defin}
The bundle $\mathcal{C}$ is a successor of $\mathcal{C}$ with  $\iota(\mathcal{C})\setminus \iota(\mathcal{C})=\emptyset$.  Let $\preceq$ be any linear order on $\mathcal{C}$. For $\C\in \mathcal{C}$ let $\C^\ast=\emptyset$. Then $\mathcal{C}$ with this $^\ast$ successor function is a $^\ast$-successor of $\mathcal{C}$,    the {\em trivial} $^\ast$-successor of $\mathcal{C}$..

\begin{lem}\label{lem:frmeldduu}
Let $\mathcal{B}$ be a $^\ast$- successor of the bundle $\mathcal{C}$. Then $\mathcal{B}$ is a melding successor of $\mathcal{C}$.
\end{lem}
\begin{proof}
 Let $\alpha$ be a $\mathcal{B}$ conform function   with $A:=\dom(\alpha)\subseteq \sigma(\mathcal{C})$. For $\B\in \alpha[A]$ let $f_{\B}\in \mathrm{G}_{\iota(\mathcal{C})}$ be a function with $f_{\B}[\alpha^{-1}(\B)]\subseteq \sigma(\B_{\downarrow E\setminus (\B\uparrow^\ast)})$. Such a function $f_{\B}$ exists because $\B_{\downarrow E\setminus (\B\uparrow^\ast)}$ is a melding successor of the type $\B_{\downarrow \iota(\mathcal{C})}$.  Because $\B$ is the free $(\B\uparrow^\ast)$-successor  of $\B_{\downarrow E\setminus (\B\uparrow^\ast)}$ there exists according to Lemma  \ref{lem:freeduch}        a function $h_\B\in\mathrm{G}_{E\setminus (\B\uparrow^\ast)}$ with $h_\B[f_{\B}[\alpha^{-1}(\B)]]\subseteq \sigma(\B)$. Let $l_\B$ be the restriction of the function $h_\B\circ f_\B$ to the set $\iota(\mathcal{C})\cup \alpha^{-1}(\B)$ and $l=\bigcup_{\B\in \mathcal{B}}l_\B$. Then $l(v)=v$ for all elements $v\in \iota(\mathcal{C})$ and  $l[A]\subseteq \sigma(\mathcal{B})$. Note that for $a\in A$ and $\B\in \mathcal{B}$ the element  $l(a)\in \sigma(\B)$ if and only if $a\in \alpha^{-1}(\B)$. 
 
 Let $\mathrm{M}$ be the structure with set of elements $M=\iota(\mathcal{B})\cup l[A]$ and  for which:
 \begin{enumerate}
\item[i.] $\mathrm{M}_{\downarrow \iota(\mathcal{B})}=\mathrm{U}_{\downarrow \iota(\mathcal{B})}$. 
\item[ii.] $\mathrm{M}_{\downarrow \iota(\mathcal{C})\cup l[A]}$ is such that the function $l$ is an isomorphism from $\mathrm{U}_{\downarrow \iota(\mathcal{C})\cup A}$ to $\mathrm{M}_{\downarrow \iota(\mathcal{C})\cup l[A]}$. 
\end{enumerate}
We will use the fact that $\mathrm{M}_{\downarrow \iota(\mathcal{B})\cup l_\B[\alpha^{-1}(\B)]}=\mathrm{U}_{\downarrow \iota(\mathcal{B})\cup l_\B[\alpha^{-1}(\B)]}$. 

\vskip 2pt
\noindent
Claim: The structure $\mathrm{M}$ is an element of the age $\mathfrak{A}$ of $\mathrm{U}$. Assume for a contradiction that there is a structure $\mathrm{P}$ in the boundary of $\mathrm{U}$ for which there exists an embedding $e$  into $\mathrm{M}$. Because of the isomorphism  $l$  it is not possible that $e[P]\subseteq \iota(\mathcal{C})\cup l[A]$. Because $\mathrm{M}_{\downarrow \iota(\mathcal{B})}=\mathrm{U}_{\downarrow \iota(\mathcal{B})}$ it is not possible that $e[P]\subseteq \iota(\mathcal{B})$. Hence $X:=e[P]\cap E\not=\emptyset$   and $Y:=e[P] \cap A\not=\emptyset$. Implying in turn that the pair of sets $(X,Y)$ form a complete bipartite subgraph of the Gaifman graph of the structure $\mathrm{M}$. It follows  from Item (2) in Definition \ref{defin:neutralsucc} using the linear order $\prec$, that  this is only possible if there exists a type $\B\in \alpha[A]$ for which $Y\subseteq \mathrm{M}_{\downarrow \iota(\mathcal{B})\cup l_\B[\alpha^{-1}(\B)]}$. Leading to the contradiction that $\mathrm{P}$ embeds into $ \mathrm{M}_{\downarrow \iota(\mathcal{B})\cup l_\B[\alpha^{-1}(\B)]}=\mathrm{U}_{\downarrow \iota(\mathcal{B})\cup l_\B[\alpha^{-1}(\B)]}$. 
 
Because $\mathrm{M}\in \mathfrak{A}$ it follows from Fact \ref{fact:contemb} that there exists an embedding $l'$ of $\mathrm{M}$ into $\mathrm{U}$ with $l'(v)=v$ for all $v\in \iota(\mathcal{B})$. Note that for $a\in A$ and $\B\in \mathcal{B}$ the element  $l'\circ l(a)\in \sigma(\B)$ if and only if $a\in \alpha^{-1}(\B)$. The function $l'\circ l$ is an embedding of $\mathrm{U}_{\downarrow \iota(\mathcal{C})\cup A}$ to $\mathrm{U}_{\downarrow \iota(\mathcal{C})\cup l'\circ l[A]}$ with $l'\circ l(v)=v$ for all $v\in \iota(\mathcal{C})$. Let $g$ be an extension of the function $l'\circ l$ to an element in $\mathrm{G}$. Then $g\in \mathrm{G}_{\iota(\mathcal{C})}$ with $g(a)\in \sigma(\alpha(a))$ for every element $a\in A$.  Hence $\mathcal{B}$ is a melding successor of the bundle $\mathcal{C}$. 
 
 \end{proof}

\begin{defin}\label{defin:refinbundle}
A bundle $\mathcal{B}$ is a {\em refinement} of a bundle $\mathcal{C}$ if: 
\begin{enumerate}
\item The bundle $\mathcal{B}$ is a melding successor of the bundle $\mathcal{C}$.  
\item The function $_{\downarrow \iota(\mathcal{C})}:  \mathcal{B}\to \mathcal{C}$ which maps every type $\B\in \mathcal{B}$ to the type $\B_{\downarrow \iota(\mathcal{C})}\in \mathcal{C}$ is a bijection. The {\em refinement bijection}.  (That is, for every type $\C\in \mathcal{C}$ exists exactly one type $\B\in \mathcal{B}$ which is a successor of $\C$.)
\end{enumerate}
\end{defin}
Observe that if a bundle $\mathcal{B}$ is a refinement of a bundle $\mathcal{C}$ then for every finite subset $A$ of $\sigma(\mathcal{C})$ there exists only one $\mathcal{B}$ conform function $\alpha$ with domain $A$. Namely $\alpha(a)=\B\in \mathcal{B}$ if and only if $a\in \sigma(\B_{\downarrow \iota(\mathcal{C})})$. Hence we obtain:

\begin{lem}\label{lem:charrefv3}
A bundle $\mathcal{B}$ is a {\em refinement} of a bundle $\mathcal{C}$ if and only if: 
\begin{enumerate}
\item The bundle $\mathcal{B}$ is a successor of the bundle $\mathcal{C}$. 
\item For every finite subset $A$ of $\sigma(\mathcal{C})$ exists a function $g\in \mathrm{G}_{\iota(\mathcal{C})}$ with $f[A]\subseteq \sigma(\mathcal{B})$.   
\item The function $_{\downarrow \iota(\mathcal{C})}:  \mathcal{B}\to \mathcal{C}$ which maps every type $\B\in \mathcal{B}$ to the type $\B_{\downarrow \iota(\mathcal{C})}\in \mathcal{C}$ is a bijection. 
\end{enumerate}
\end{lem}

Be aware of the following: A successor $\S$  of a type $\T$ has been defined early in   Section \ref{sect:types} to be a refinement of $\T$ if $\S$ and $\T$ have the same rank, that is if $\rro(\S)=\rro(\T)$. On the other hand the bundle $\{S\}$ is a refinement  of the bundle $\T$ if the type $\S$ is a successor of the type $\T$ and if for every finite $A\subseteq \sigma(\T)$ there exists a function $g\in \mathrm{G}_{\iota(\T)}$ with $g[A]\subseteq \sigma(\S)$. It follows that if the bundle $\{S\}$ is a refinement of the bundle $\{T\}$ then the type $\S$ is a refinement of the type $\T$. But the converse need not be true.

\begin{lem}\label{lem:begconst}
Let $\U$ be a countable and rank linear homogeneous structure.   Let $\V=\langle V\tr v\rangle$ be a refinement of the type $\U=\langle \emptyset\tr u\rangle$ for $u\in U$. Then the bundle $\{\V\}$ is a  refinement of the bundle $\{\U\}$. 
\end{lem}
\begin{proof}
We have $\rro(\V)=\rro(\U)$ because $V$ is a refinement of $\U$. It follows then from Lemma \ref{lem:trmeldcan} that the type $\V$ is a melding refinement of the type $\U$. Implying that the bundle $\{ \V\}$ is a refinement of the bundle $\{\U\}$. 
\end{proof}

\begin{note}\label{note:compref}
If a bundle $\mathcal{C}$ is a  refinement of a bundle $\mathcal{B}$ which in turn is a  refinement of a bundle $\mathcal{A}$ then the bundle $\mathcal{C}$ is a refinement of the bundle $\mathcal{A}$ with a refinement bijection which is the composition of the refinement bijections. 
\end{note}

\begin{defin}\label{defin:xsuccb}
For a type $\T$ and an element $x\in U\setminus \iota(\T)$ let $\T_{\uparrow x}$ denote the bundle of all $\{x\}$-successors of the type $\T$. For a bundle $\mathcal{B}$ let $\mathcal{B}_{\uparrow x}:=\bigcup_{\T\in \mathcal{B}}\T_{\uparrow x}$. We say that $\mathcal{B}_{\uparrow x}$ is the {\em $x$-successor of the bundle $\mathcal{B}$.}  For $\X\in \mathcal{B}_{\uparrow x}$ let $\X_{\downarrow -x}$ be the type $\T\in \mathcal{B}$ for which $\X\in \T_{\uparrow x}$. (Then $\X_{\downarrow -x}=\X_{\downarrow \iota(\X)\setminus\{x\}}$.)
\end{defin} 

The following Note follows then from Note \ref{note:typeim}. 

\begin{note}\label{note:bundlerefxind}
Let  $\mathcal{C}$ be a  bundle and let $\T\in \mathcal{C}$ and $\{x,y\}\subseteq \sigma(\T)$.   Then  every automorphism $f\in \mathrm{G}_{\iota(\mathcal{C})}$ with $f(x)=y$ induces a bijection of the bundle $\mathcal{C}_{\uparrow x}$ to the bundle $\mathcal{C}_{\uparrow y}$
by mapping  $\X\in \mathcal{C}_{\uparrow x}$ to $f[\X]\in \mathcal{C}_{\uparrow y}$. Moreover, then $\rro(\X)= \rro(f[\X])$.  
\end{note}

\begin{defin}\label{defin:conformsux}
Let  $\mathcal{C}$ be a bundle  with $\Z\in \mathcal{C}$ and  $x\in \sigma(\Z)$.  A  melding successor $\mathcal{B}$ of the bundle $\mathcal{C}$ with $x\not\in \iota(\B)$ {\em agrees with} the bundle $\mathcal{C}_{\uparrow x}$ if there exists a function $\beta: \mathcal{C}_{\uparrow x}\to \mathcal{B}$, the {\em agreement function} of $\mathcal{C}_{\uparrow x}$ to $\mathcal{B}$,  so that  for all types $\D\in \mathcal{C}_{\uparrow x}$
\begin{enumerate}
         \item $\D_{\downarrow \iota(\mathcal{C})}=(\beta(\D))_{\downarrow \iota(\mathcal{C})}$. 
         \item $\rro(\D)=\rro(\beta(\D))$. 
\end{enumerate}
\end{defin}
Item (1) above implies that both typesets $\sigma(\D)$ and $\sigma(\beta(\D))$ are subsets of the typeset of the same type $\C\in \mathcal{C}$.

\begin{lem}\label{lem:chanxz}
Let  $\mathcal{C}$ be a bundle  with $\Z\in \mathcal{C}$ and  $\{x,y\}\subseteq  \sigma(\Z)$. Let $g\in \mathrm{G}_{\iota(\mathcal{C})}$ with $g(x)=y$.  Let $\mathcal{B}$ be a melding successor of $\mathcal{C}$ with $\{x,y\}\cap \iota(\mathcal{B})=\emptyset$. Then a function $\beta: \mathcal{C}_{\uparrow x}\to \mathcal{B}$ is an agreement function of $\mathcal{C}_{\uparrow x}$ to $\mathcal{B}$ if and only if $\beta\circ g^{-1}: \mathcal{C}_{\uparrow y}\to \mathcal{B}$ is an agreement function of  $\mathcal{C}_{\uparrow y}$ to $\mathcal{B}$.  The  melding successor $\mathcal{B}$ of the bundle $\mathcal{C}$ agrees with the bundle $\mathcal{C}_{\uparrow x}$ if and only if $\mathcal{B}$  agrees with the bundle $\mathcal{C}_{\uparrow y}$. 

\end{lem}
\begin{proof}
Let $\D\in \mathcal{C}_{\uparrow y}$.\\
 Then $(g^{-1}(\D))_{\iota(\mathcal{C})}=\D_{\iota(\mathcal{C})}$ because $g\in \mathrm{G}_{\iota(\mathcal{C})}$. It follows from Item (1) of Definition \ref{defin:conformsux} that
 $(g^{-1}(\D))_{\iota(\mathcal{C})}=(\beta(g^{-1}(\D)))_{\iota(\mathcal{C})}$. Using  Note \ref{note:bundlerefxind}  we obtain $\rro(\D)=\rro(g^{-1}(\D))=\rro(\beta(g^{-1}(\D)))$. 
\end{proof}

We will need conditions under which finite subsets of $\sigma(\D)$ can be ``moved" while ``fixing"  $\iota(\mathcal{C})$ into $\sigma(\beta(\D))$. Of course if a subset of $\sigma(\D)$ induces a structure which is not in the rank of $\beta(\D)$ then such a move is not possible. Item (2) of Definition \ref{defin:conformsux}  removes this obstacle.

\begin{defin}\label{defin:jionbundlbet}
Let $\mathcal{B}$ be a successor of a bundle   $\mathcal{C}$  and $x\in\sigma(\mathcal{C})\setminus\iota(\mathcal{B})$.  Let $\beta$ be a function from  $\mathcal{C}_{\uparrow x}$ to $\mathcal{B}$ for which  every type $\D\in \mathcal{C}_{\uparrow x}$ is in free position with the type $\beta(\D)$. Then $\mathcal{B}\stackrel{\beta}{\sqcap}\mathcal{C}_{\uparrow x}$ is the bundle:
\[
\mathcal{B}\stackrel{\beta}{\sqcap}\mathcal{C}_{\uparrow x}:=\{\beta(\D)\sqcap \D\mid \D\in\mathcal{C}_{\uparrow x}\}
\]
\end{defin}
It follows from Lemma \ref{lem:commeldstrm} that the type $\beta(\D)\sqcap \D$ exists for every type $\D\in \mathcal{C}_{\uparrow x}$.

\begin{lem}\label{lem:conformsuxk}
Let  $\mathcal{C}$ be a bundle  with $\Z\in \mathcal{C}$ and  $x\in \sigma(\Z)$.   Let $\mathcal{B}$ be a  melding successor  of the bundle $\mathcal{C}$  which agrees with the bundle $\mathcal{C}_{\uparrow x}$. Let $\beta$ be the agreement function of $\mathcal{C}_{\uparrow x}$ to $\mathcal{B}$.    Let $x\in \sigma(\Z')$ for $\Z'$ being the free $(\iota(\mathcal{B})\setminus \iota(\mathcal{C}))$-successor of $\Z$.

Then the types $\D$ and $\beta(\D)$ are in free position for every type $\B\in \mathcal{B}$ and  the bundle $\mathcal{B}\stackrel{\beta}{\sqcap}\mathcal{C}_{\uparrow x}$ is a refinement of the bundle $\mathcal{C}_{\uparrow x}$. 
\end{lem}
\begin{proof}
Every type $\D\in \mathcal{C}_{\uparrow x}$ is in free position with the type $\beta(\D)\in \mathcal{B}$.  Because  
$\iota(\D)\cap \iota(\beta(\D))=\iota(\mathcal{C})$,  $\iota(\D)\setminus \iota(\mathcal{C})=\{x\}$ and the element $x$ is in $\Z'$  which is the free $(\iota(\beta(\D))\setminus \iota(\mathcal{C}))$-successor of $\Z\in \mathcal{C}$. 

Let $A$ be a finite subset of $\sigma(\mathcal{C}_{\uparrow x})$. Note that $\sigma(\mathcal{C}_{\uparrow x})=\sigma(\mathcal{C})$. For every $a\in A$ there exists a unique type, denoted $\delta(a)\in \mathcal{C}_{\uparrow x}$, with $a\in \sigma(\delta(a))$. Let  $\beta\circ \delta:=\alpha: A\to \mathcal{B}$. Observe that $\alpha$ is $\mathcal{B}$ conform. The bundle $\mathcal{B}$ is melding. Hence there exists following Definition \ref{defin:wellplsucc}  a function $h\in \mathrm{G}_{\iota(\mathcal{C})}$ with $h(a)\in \sigma(\alpha(a))$ for every element $a\in A$.

 Let $\mathrm{M}=\mathrm{U}_{\downarrow \iota(\mathcal{B})\cup h[A]}$. Let $l$ with $\dom(l)=\iota(\mathcal{C})\cup \{x\}\cup A$ be the function with $l(v)=v$ for all $v\in \iota(\mathcal{C})\cup \{x\}$ and with $l(a)=h(a)$ for all $a\in A$.  Let $\mathrm{N}$ be the $\boldsymbol{L}$-structure for which $N=\iota(\mathcal{C})\cup \{x\}\cup A$ and  so that   the function $l$   is an isomorphism  of $\mathrm{U}_{\downarrow  \iota(\mathcal{C})\cup \{x\}\cup A}$ to $\mathrm{N}$. The structures $\mathrm{M}$ and $\mathrm{N}$ form an amalgamation instance with $M\cap N=\iota(\mathcal{C})\cup h[A]$. Let $\mathrm{R}$ be the free amalgam of $\mathrm{M}$ and $\mathrm{N}$.  It follows from Fact \ref{fact:contemb} that there exists an embedding $k$ of the structure $\mathrm{R}$ into $\mathrm{U}$ with $k(v)=v$ for all elements $v\in \iota(\mathcal{B})\cup \{x\}$. It follows from $k{v}=v$ for all $v\in \iota(\mathcal{B})$ that $k\circ h(a)\in \sigma(\alpha(a))$ for every element $a\in A$. Implying that $k\circ h[A\cap \sigma(\D)]\subseteq \sigma(\beta(\D))$ for every type $\D\in \mathcal{C}_{\uparrow x}$.   The function $k\circ l$ is an isomorphism of $\U_{\iota(\mathcal{C})\cup \{x\}\cup A}$ to $\mathrm{U}_{\downarrow \iota(\mathcal{C})\cup \{x\}\cup k\circ l[A]}$. Let $f\in \mathrm{G}$ be an extension of the isomorphism $k\circ l$ to a function in $\mathrm{G}$. Then $f\in \mathrm{G}_{\iota(\mathcal{C})\cup \{x\}}$ and $f[A\cap \sigma(\D)]=k\circ h[A\cap \sigma(\D)]\subseteq \sigma(\beta(\D))$ for every type $\D\in \mathcal{C}_{\uparrow x}$. Hence $f[A\cap \sigma(\D)]\subseteq \sigma(\beta(\D)\sqcap \D)$ for every type $\D\in \mathcal{C}_{\uparrow x}$.

\end{proof}

\begin{defin}\label{defin:tyepexppp}
Let a bundle  $\mathcal{B}$ be a refinement of a bundle $\mathcal{C}$  and let $\Z\in\mathcal{B}$  and $x\in \sigma(\Z)$. Let $F=\iota(\mathcal{C})$ and $E=\iota(\mathcal{B})\setminus F$.  For $\B\in \mathcal{B}$ and $\Y\in (\B_{\downarrow F})_{\uparrow x}$ let $\Y^{\langle -1\rangle}$ denote the bundle of all types  $\X\in \B_{\uparrow x}$ which are  $E$-successors of $\Y$. That is for which $\sigma(\X)\subseteq \sigma(\Y)\cap \sigma(\B)$.

An {\em $x$-continuation} of a refinement $\mathcal{B}$  of a bundle  $\mathcal{C}$ is a refinement $\mathcal{X}\subseteq \mathcal{B}_{\uparrow x}$ of the bundle $\mathcal{C}_{\uparrow x}$ so that $\X_{\downarrow F\cup \{x\}}\in \mathcal{C}_{\uparrow x}$ for every type $\X\in \mathcal{X}$.   (The bundle $\mathcal{X}$ selects one type out of every bundle $\Y^{\langle -1\rangle}$. Hence $\X\in (\X_{\downarrow E})^{\langle-1\rangle}$.)
\end{defin}
Note that if $\mathrm{U}$ is a binary relational structure then $|\Y^{\langle -1\rangle}|\leq 1$. It will follow  from the next Lemma \ref{lem:erefx234} that $|\Y^{\langle -1\rangle}|\geq 1$.

\begin{lem}\label{lem:erefx234}
Let $\mathrm{U}$ be an oligomorphic,   free amalgamation homogeneous structure.  Let a bundle  $\mathcal{B}$ of types of\/  $\mathrm{U}$ be a refinement of a bundle $\mathcal{C}$  and let $\Z\in\mathcal{B}$  and $z\in \sigma(\Z)$.   Then there exists a bundle $\overline{\mathcal{Y}}\subseteq \mathcal{B}_{\uparrow z}$ which is a  $z$-continuation of the bundle $(\mathcal{C})_{\uparrow z}$. 

For every $\C\in \mathcal{C}$ is the set of ranks of the types in $\C_{\uparrow z}$ equal to the set of ranks of the types in $\mathcal{Y}$ which are successors of the type $\C$. 
\end{lem}
\begin{proof}
Claim: There exists for every finite subset $A\subseteq \sigma(\mathcal{C}_{\uparrow z})$ a function $h\in \mathrm{G}_{\iota(\mathcal{C})\cup \{z\}}$ with $h[A]\subseteq \sigma(\mathcal{B}_{\uparrow z})$. 
\vskip 5pt
\noindent
Proof of Claim: Let $A$ be a finite subset of $\sigma(\mathcal{C}_{\uparrow z})= \sigma(\mathcal{C})\setminus\{z\}$. Then $A\cup \{z\}$ is a finite subset of $\sigma(\mathcal{C})$. Hence there exists a function $g\in \mathrm{G}_{\iota(\mathcal{C})}$ with $g[A\cup \{z\}]\subseteq \sigma(\mathcal{B})$.  Note  $g(z)\in \sigma(\Z)$. There exists a function $f\in \mathrm{G}_{\iota(\Z)}=\mathrm{G}_{\iota(\mathcal{B})}$ with $f(g(z))=z$.  The function $f$ is an element of $\mathrm{G}_{\iota(\mathcal{C})}$ because $\iota(\mathcal{C})\subseteq \iota(\mathcal{B})$. Hence $f\circ g[A]\subseteq \sigma(\mathcal{B}_{\uparrow z})$. Let $h=f\circ g$, verifying the Claim. 

According to Corollary \ref{cor:agindfreambund} there exists then a bundle $\overline{\mathcal{Y}}\subseteq \mathcal{B}_{\uparrow z}$ which is a refinement of the bundle $(\mathcal{C})_{\uparrow z}$.
\end{proof}

\section{A basic vertex partition theorem}\label{sect:basvertthm}

\begin{defin}\label{defin:mho-game}
Let $U$ be an infinite set, $(\boldsymbol{\mathscr{S}}, \sqsubseteq)$ a partially ordered set, $(\mathfrak{R};\subseteq)$ a linearly ordered set, $\sigma$ a map from $\mathscr{S}$ to the set of infinite subsets of $U$ and $\rro$ a map from  $\boldsymbol{\mathscr{S}}$ to $\mathfrak{R}$. We say that the property
\[
\mho\big(U, (\boldsymbol{\mathscr{S}};\sqsubseteq), (\mathfrak{R};\subseteq), \sigma, \rro\big)
\] 
holds if for all $\X,\Y\in \boldsymbol{\mathscr{S}}$:

\begin{enumerate}
\item[i.] $\X\sqsubseteq \Y$ implies $\sigma(\X)\subseteq \sigma(\Y)$.
\item[ii.] $\X\sqsubseteq \Y$ implies $\rro(\X)\subseteq \rro(\Y)$. 
\item[iii.] For all  $\X\in \boldsymbol{\mathscr{S}}$ and all $\mathfrak{r}\in\mathfrak{R}$ with $\mathfrak{r}\subseteq \rro(\X)$ there exists an $\R\in \boldsymbol{\mathscr{S}}$ with $\R\sqsubseteq \X$ and $\rro(\R)=\mathfrak{r}$.   
\item[iv.] $\boldsymbol{\mathscr{S}}$ has a maximum $\U$ with $\sigma(\U)=U$. 
\end{enumerate}

Let
$\mho\big(U, (\boldsymbol{\mathscr{S}};\sqsubseteq), (\mathfrak{R};\subseteq), \sigma, \rro\big)$. The elements of $\boldsymbol{\mathscr{S}}$ will be called {\em sorts}. The elements of $\mathfrak{R}$ will be called {\em ranks} and $\rro(\X)$ the {\em rank of the sort $\X$}.   A sort $\X$ is an {\em $\mathfrak{r}$-restriction} of a sort $\Y$ if $\X\sqsubseteq \Y$ and if $\rro(\X)=\mathfrak{r}$. A sort $\X$ is a {\em refinement} of a sort $\Y$ if $\X\sqsubseteq \Y$ and if $\rro(\X)=\rro(\Y)$. Observe that $\X$ is a refinement of $\X$ and that a refinement of a refinement is a refinement. 
\end{defin}

\begin{lem}\label{lem:realsortyp}
Let $\mathrm{U}$ be a countable homogeneous structure whose age has free amalgamation. Let $\boldsymbol{\mathscr{S}}$ denote the set of formed types of $\mathrm{U}$. For two types $\X$ and $\Y$ in $\mathscr{S}$ let $\X\sqsubseteq \Y$ if $\X$ is a successor of $\Y$.   Let $\mathfrak{R}$ denote the set of ranks of the set of formed types $\boldsymbol{\mathscr{S}}$. If the structure $\mathrm{U}$ is rank linear then:

\[
\mho\big(U, (\boldsymbol{\mathscr{S}};\sqsubseteq), (\mathfrak{R};\subseteq), \sigma, \rro\big).
\]
\end{lem}
\begin{proof}
It follows from Corollary \ref{cor:tyinfA} that $\sigma$ is a map from $\boldsymbol{\mathscr{S}}$ to the set of infinite subsets of $U$. The structure $\mathrm{U}$ is rank linear implying that the automorphism group of $\mathrm{U}$ acts transitively on $U$. Implying in turn that $\U:=\langle \emptyset\tr x\rangle$ is a type of $\mathrm{U}$ with $\sigma(\U)=U$. Hence,   the set $U$ of elements of the structure $\mathrm{U}$ is infinite. Item (1) in Section \ref{section: formedtyp} together with Definition \ref{defin:formedtypes} imply that $\U\in \boldsymbol{\mathscr{S}}$. It follows then from Lemma \ref{lem:succsuccform}  that
$(\boldsymbol{\mathscr{S}}, \sqsubseteq)$ is a partially ordered set. The set of ranks of the formed types  is a subset of the set of ranks all types and hence $(\mathfrak{R};\subseteq)$ is a linearly ordered set.  By definition $\rro$ is map from  $\boldsymbol{\mathscr{S}}$ to $\mathfrak{R}$. 

Let $\X,Y\in \boldsymbol{\mathscr{S}}$. Item i. is part of the definition of successor. Item ii. follows from Item i. using the Definition of the function $\rro$. Let $\mathfrak{r}\in \mathfrak{R}$ with $\mathfrak{r}\subseteq \rro(\X)$. It follows from Lemma \ref{lem:restriction} that there exists a type $\Z$ with $\rro(\X+\Z)=\mathfrak{r}$. Put $\R=\X+\Z$. Then $\rro(\R)=\mathfrak{r}$ and $\R$ is a successor of $\X$ having form $\X+$. Then $R\in \boldsymbol{\mathscr{S}}$ because $\X\in \boldsymbol{\mathscr{S}}$.  The type $\U$ is the maximum of the partial order $(\boldsymbol{\mathscr{S}};\sqsubseteq)$. 
\end{proof}

\begin{lem}\label{lem:realsortypt}
Let $\mathrm{U}$ be a countable homogeneous structure whose age has free amalgamation. Let $\boldsymbol{\mathscr{T}}$ denote the set of  types of $\mathrm{U}$. For two types $\X$ and $\Y$  let $\X\sqsubseteq \Y$ if $\X$ is a successor of $\Y$.   Let $\mathfrak{R}$ denote the set of ranks of the set of  types $\boldsymbol{\mathscr{T}}$. If the structure $\mathrm{U}$ is rank linear then:

\[
\mho\big(U, (\boldsymbol{\mathscr{T}};\sqsubseteq), (\mathfrak{R};\subseteq), \sigma, \rro\big).
\]
\end{lem}
\begin{proof}
Analogues to the proof of Lemma \ref{lem:realsortyp}. 
\end{proof}

A subset $S$ of $U$ is {\em large}  if there exists a set   $\mathscr{W}\subseteq \boldsymbol{\mathscr{S}}$ with
 $\U\in \mathscr{W}$ and  so that  $\X\in \mathscr{W}$ implies that $S\cap \sigma(\X)$ is infinite and that the following formula $\phi_{\mathscr{W}}(\X)$ holds:

\[  
\phi_{\mathscr{W}}(\X):=
\begin{cases}
\text{For all $\mathfrak{r}\in \mathfrak{R}$ with $\mathfrak{r}\subseteq \boldsymbol{\rho}(\X)$}   &(1)\\  
\text{there exists a  refinement  $\Y$ of $\X$ so that}                                                                       &(2)\\     
\text{for all  refinements $\Z\ $ of $\Y$}                                                                                              &(3)\\
\text{there exists an $\mathfrak{r}$-restriction  $\R\in \mathscr{W}$ of $\Z$}.                    &(4)
\end{cases}
\]

\noindent
The set   $\mathscr{W}$ is then a {\em witness} for  $S$ being large.    For a given $\mathfrak{r}\in \mathfrak{R}$ with $\mathfrak{r}\subseteq \boldsymbol{\rho}(\X)$ formula $\phi_{\mathscr{W}}(\X,\mathfrak{r})$ is the formula consisting of lines (2), (3) and (4) of formula $\phi_{\mathscr{W}}(\X,S)$. It is important to observe that the set $S$ does not appear in any of the four lines of formula $\phi_{\mathscr{W}}(\X)$.  Intuitively the formula says that the set $S$ is large enough so that going through the four alternating quantifiers along the partial order $(\boldsymbol{\mathscr{S}}, \sqsubseteq)$ will always produce a sort $\R$ of desired rank for which $S\cap \sigma(\R)$ is infinite.

\begin{defin}\label{defin:gameS_1S_2}
Let  $\mho\big(U, (\boldsymbol{\mathscr{S}};\sqsubseteq), (\mathfrak{R};\subseteq), \sigma, \rro\big)$ and let $S\subseteq U$.  Then $\lambda_S$ is a function which assigns to some of the sorts in $\boldsymbol{\mathscr{S}}$ an ordinal number as a {\em label}. This labelling $\lambda_S$ is defined recursively as follows:

If $\sigma(\X)\setminus S$ is finite then $\lambda_S(\X):=0$. If a sort $\X$ is not yet labelled but if there exists a rank $\mathfrak{R}\ni \mathfrak{r}\subseteq \rro(\X)$ so that  for every  refinement $\Y$ of $\X$ there exists a  refinement $\Z$ of $\Y$  for which all $\mathfrak{r}$-restrictions  are labelled,  let $\lambda_S(\X)$ be the smallest ordinal larger than all labels $\lambda_S(\Y)$ with $\Y\sqsubset \X$ and $\Y\not=\X$. 
\end{defin}

\begin{lem}\label{lem:sortinfS}
Let  $\mho\big(U, (\boldsymbol{\mathscr{S}};\sqsubseteq), (\mathfrak{R};\subseteq), \sigma, \rro\big)$ and let $S\subseteq U$.   If a sort $\X$ is labelled by $\lambda_S$ then every refinement of $\X$ is labelled and  $\sigma(\X)\cap S$ is infinite.
\end{lem}
\begin{proof}
Let $\X$ be labelled and $\mathfrak{r}$ the rank provided by Definition \ref{defin:gameS_1S_2} for which $\X$ became labelled.   Let $\Y$ be a refinement of $\X$ and $\Y'$ a refinement of $\Y$. Then $\Y'$ is a refinement of $\X$ and hence there exists a  refinement $\Z$ of $\Y$  for which all $\mathfrak{r}$-restrictions  are labelled. 

Let $\mathscr{T}$ be the set of labelled sorts $\T$ for which  $\sigma(\T)\cap S$ is finite. If $\mathscr{T}$ is not empty let $\mathfrak{q}=\min\{\lambda_S(\T)\mid \T\in \mathscr{T}\}$ and $\X\in \mathscr{T}$ with $\lambda_S(\X)=\mathfrak{q}$. Let $\mathfrak{r}$ be the rank provided by Definition \ref{defin:gameS_1S_2} for which $\X$ became labelled.  The sort $\X$ is a refinement of $\X$. Hence there exists a sort $\Z$ of $\X$ for which all $\mathfrak{r}$-restrictions are labelled. This set of $\mathfrak{r}$-restrictions is not empty according to Item $iii.$ of Definition \ref{defin:mho-game}. The set $\sigma(\R)$ of each of those $\mathfrak{r}$-restrictions $\R$ contains infinitely many elements of $S$ and is a subset of $\sigma(\X)$. 
\end{proof} 

Hence if a sort $\X$ is labelled by $\lambda_S$ then $\sigma(\X)\cap S$ is infinite and the following formula \\

$\Phi(\X,S):=$
\[  
\begin{cases}
\text{there exists an $\mathfrak{r}\in \mathfrak{R}$ with $\mathfrak{r}\subseteq \boldsymbol{\rho}(\X)$}   &(1)\\  
\text{so that for all   refinements  $\Y$ of $\X$ }                                                                       &(2)\\     
\text{there exists a  refinement $\Z\ $ of $\Y$ for which all}                                                                                              &(3)\\
\text{$\mathfrak{r}$-restrictions $\R$ of $\Z$ are are labelled by $\lambda_S$}                    &(4)
\end{cases}
\]
holds. Note here that if a sort $\X$ is not labelled then  $\neg\Phi(\X)$. Let $\Phi(\X,\mathfrak{r},S)$ if the sub-formula of $\Phi(\X,S)$ consisting of lines (2), (3) and (4) holds and if $\mathfrak{r}\in \mathfrak{R}$ with $\mathfrak{r}\subseteq \rro(\X)$. 

\begin{defin}\label{defin:witness}
For $S\subseteq U$  let $\mathscr{W}(S)$ be the set of sorts which are labelled by $\lambda_S$. Let  $\overline{\mathscr{W}(S)}$ be the set of sorts which are not  labelled by $\lambda_S$ and let $\overline{S}=U\setminus S$.
\end{defin}

\begin{lem}\label{lem:formphion} 
Let  $\mho\big(U, (\boldsymbol{\mathscr{S}};\sqsubseteq), (\mathfrak{R};\subseteq), \sigma, \rro\big)$ and let $S\subseteq U$.   If a sort $\X$ is not labelled by $\lambda_S$, that is if $\X\in \overline{\mathscr{W}(S)}$,    then $\phi_{\overline{\mathscr{W}}(S)}(\X)$ and $\sigma(\X)\cap \overline{S}$ is infinite.
\end{lem}
\begin{proof}
If $\sigma(\X)\cap \overline{S}$ is finite then $\sigma(\X)\setminus S$ is finite and hence $\X$ is labelled having label 0. If a sort $\X$ is not labelled then  $\neg\Phi(\X,S)$ which is the formula  $\phi_{\overline{\mathscr{W}(S)}}(\X)$. 
\end{proof}

\begin{lem}\label{lem:implofform}
Let  $\mho\big(U, (\boldsymbol{\mathscr{S}};\sqsubseteq), (\mathfrak{R};\subseteq), \sigma, \rro\big)$ and let $S\subseteq U$.  If $\Phi(\X,\mathfrak{r}, S)$ then $\phi_{\mathscr{W}(S)}(\X,\mathfrak{r})$.  If for a sort $\X$ the label $\lambda_S(\X)=0$ then  $\phi_{\mathscr{W}(S)}(\X)$. 
\end{lem}
\begin{proof}
Let $\Phi(\X,\mathfrak{r}, S)$. To establish line (2) of formula $\phi_{\mathscr{W}}(\X)$ put $\Y=\X$. Then for all refinements $\Z$ of $\X=\Y$  there exists a refinement $\Z'$ of $Z$  for  which all  $\mathfrak{r}$-restrictions $\R$ are sorts in $\mathscr{W}$. Every one of those $\mathfrak{r}$-restrictions is a $\mathfrak{r}$-restriction of $\Z$.  It follows from  Item ~iii.  of Definition~\ref{defin:mho-game} that this set of $\mathfrak{r}$-restrictions is not empty. 

Let $\X$ be a sort with $\lambda_S(\X)=0$ and $\mathfrak{r}\subseteq \rro(\X)$. Then $\R$ is labelled and  $\lambda_S(\R)=0$ for every sort $\R\sqsubseteq \X$.  Hence putting $\X$ for $\Y$ in formula $\phi_{\mathscr{W}(S)}(\X)$ let $\Z$ be a refinement of $\X$. According to Item iii. of Definition~\ref{defin:mho-game} there exists an $\mathfrak{r}$-restriction $\R$ of $\Z$.
\end{proof}

\begin{lem}\label{lem:smtobig}
Let  $\mho\big(U, (\boldsymbol{\mathscr{S}};\sqsubseteq), (\mathfrak{R};\subseteq), \sigma, \rro\big)$ and let $S\subseteq U$.  Let $\X$ be a $\lambda_S$-labelled sort for which $\Phi(\X,\mathfrak{r}_0)$. Then $\Phi(\X,\mathfrak{r})$ for all $\mathfrak{r}\in \mathfrak{R}$ with $\mathfrak{r}_0\subseteq \mathfrak{r}\subseteq \rro(\X)$. 
\end{lem}
\begin{proof}
If $\neg \Phi(\X,\mathfrak{r})$ there exists a refinement $\Y$ of $\X$ so that for all refinements $\Z$ of $\Y$  there exists an $\mathfrak{r}$-restriction $\R$  of $\Z$   which is not labelled. Because  $\Phi(\X,\mathfrak{r}_0,S)$ and every refinement $\Z$ of $\Y$ is a refinement of $\X$ there exists for all of the refinements $\Z$ of $ \Y$ a refinement $\Z'$ of $\Z$ for which all $\mathfrak{r}_0$-restrictions $\R'$ of $\Z'$ are are labelled. If $\R$ is an $\mathfrak{r}$-restriction of $\Z'$ then each of its $\mathfrak{r}_0$-restrictions is an $\mathfrak{r}_0$-restriction of $\Z'$ and hence all $\mathfrak{r}_0$-restrictions $\R'$ of $\R$ are are labelled. Implying $\Phi(\R,\mathfrak{r}_0)$. 

All of those refinements $\Z'$ of $\Z$ which in turn is a refinement of $\Y$ are themselves refinements of $\Y$.  Leading to the contradiction that there exists a refinement $\Z'$ of $\Y$ for which  there exists an $\mathfrak{r}$-restriction $\R$  of $\Z'$ with $\rro(\R)=\mathfrak{r}$  which is not labelled.
\end{proof} 

The proof of the next Lemma uses the assumption that $(\mathfrak{R};\subseteq)$ is  a linearly ordered set. That is  for every rank $\mathfrak{r}_0$ and every rank $\mathfrak{r}$ either $\mathfrak{r}\subseteq \mathfrak{r}_0$ or $\mathfrak{r}_0\subseteq \mathfrak{r}$. 

\begin{lem}\label{lem:gammaplI}
Let  $\mho\big(U, (\boldsymbol{\mathscr{S}};\sqsubseteq), (\mathfrak{R};\subseteq), \sigma, \rro\big)$  and let $S\subseteq U$.  Let $\X$ be a $\lambda_S$-labelled sort. Then $\phi_{\mathscr{W}(S)}(\X)$ and $\sigma(\X)\cap S$ is infinite.
\end{lem}
\begin{proof}
Let $\X$ be a labelled sort for which $\lambda_S(\X)$ is minimal amongst all of the labelled sorts for which it is not the case that $\phi_{\mathscr{W}(S)}(\X)$. Hence, according to Lemma \ref{lem:implofform}, there exists an $\mathfrak{r}\in \mathfrak{R}$ for which it is not the case that $\phi_{\mathscr{W}(S)}(\X,\mathfrak{r})$. If $\lambda_S(\X)=0$, that is if $\sigma(\X)\setminus S$ is finite, then $\phi_{\mathscr{W}(S)}(\X)$ according to Lemma~\ref{lem:implofform}. Because $\X$ is labelled there exists an $\mathfrak{r}_0\in \mathfrak{R}$ with $\Phi(\X,\mathfrak{r}_0,S)$. Implying, because $\neg\phi_{\mathscr{W}(S)}(\X)$,  according to Lemma \ref{lem:smtobig} and Lemma \ref{lem:implofform}  that $\mathfrak{r}\subsetneq \mathfrak{r}_0$. 

Every refinement $\Z$ of a sort  $\Y$ which is a refinement of $\X$  is a refinement of $\X$. Because $\Phi(\X,\mathfrak{r}_0,S)$ there exists a refinement $\Z$ of $\X$ for which all $\mathfrak{r}_0$-restrictions $\R$ of $\Z$ are labelled and $\lambda_S(\R)<\lambda_S(\X)$. (This sort $\Z$ will be used for line (2) of formula $\phi_{\mathscr{W}(S)(S)}(\X,S)$.) Implying that for every refinement $\Z'$ of $\Z$ all $\mathfrak{r}_0$-restrictions $\R$ of $\Z'$ are labelled and $\lambda_S(\R)<\lambda_S(\X)$. Then $\phi_{\mathscr{W}(S)}(\R)$, for every one of those $\mathfrak{r}_0$-restrictions $\R$ of $\Z'$,   because of the minimality of $\lambda_S(\X)$. Hence and because $\mathfrak{r}\subsetneq \rro(\R)$,  there exists a labelled $\mathfrak{r}$-restriction $\R'$ of $\R$. This labelled $\mathfrak{r}$-restriction $\R'$ of $\R$ is a labelled $\mathfrak{r}$-restriction of $\Z'$.
\end{proof}

\begin{thm}\label{thm:ramsW}
Let $\mho\big(U, (\boldsymbol{\mathscr{S}};\sqsubseteq), (\mathfrak{R};\subseteq), \sigma, \rro\big)$ and let $S\subseteq U$ and let $(S,\overline{S}=U\setminus S)$ be a partition of $U$. Then  $S$ is large with witness $\mathscr{W}(S)$ or $\overline{S}$ is large with witness $\overline{\mathscr{W}(S)}$. 
\end{thm}
\begin{proof}
If the type $\U$ is $\lambda_S$-labelled then $S$ is large with witness $\mathscr{W}(S)$ according to Lemma \ref{lem:gammaplI}. If the type $\U$ is not $\lambda_S$-labelled then $\overline{S}$ is large with witness $\overline{\mathscr{W}(S)}$ according to Lemma \ref{lem:gammaplI}. 

\end{proof}

\begin{defin}\label{defin:gammaV}
Let $\mho\big(U, (\boldsymbol{\mathscr{S}};\sqsubseteq), (\mathfrak{R};\subseteq), \sigma, \rro\big)$. For $\mathscr{Z}\subseteq \boldsymbol{\mathscr{S}}$ let:
\begin{align*}
& \gamma_\mathscr{Z}(\X):=
\begin{cases}
\text{For all  refinements $\Y$ of $\X$}       &(1)\\
\text{For all $\mathfrak{r}\in \mathfrak{R}$ with $\mathfrak{r}\subseteq \boldsymbol{\rho}(\X)=\boldsymbol{\rho}(\Y)$}   &(2)\\  
\text{there exists an $\mathfrak{r}$-restriction  $\R$ of $\Y$ with $\R\in \mathscr{Z}$}.                          &(3)
\end{cases}
\end{align*}
\end{defin}

Let $\mho\big(U, (\boldsymbol{\mathscr{S}};\sqsubseteq), (\mathfrak{R};\subseteq), \sigma, \rro\big)$.  Let $S$ be a large subset of  \/ $U$ with witness $\mathscr{W}$. For $\X\in \mathscr{W}$  let $\X'$ be the sort supplied by formula $\phi_{\mathscr{W}}(\X,\rro(\X),S)$ in line (2) of formula $\phi_{\mathscr{W}}(\X,S)$. Let $\mathscr{V}=\{\X'\mid \X\in \mathscr{W}\}$. It follows that if $\X\in \mathscr{V}$ then $\psi_{\mathscr{V}}(\X)$ holds, which is: 
\[  
\psi_{\mathscr{V}}(\X):=
\begin{cases}
\text{for all   refinements $\Y\ $ of $\X$}                                                                                              &(3)\\
\text{there exists a $\rro(X)$-restriction  $\R\in \mathscr{W}$ of $\Y$}.                    &(4)
\end{cases}
\]
If $\X\in \mathscr{V}$ then the set $\sigma(\X)\cap S$ is infinite because  $\X\sqsupseteq \Y\sqsupseteq  \R$ and $\sigma(\R)\cap S$ is infinite.

\begin{lem}\label{lem:phi2psi} 
Let $\mho\big(U, (\boldsymbol{\mathscr{S}};\sqsubseteq), (\mathfrak{R};\subseteq), \sigma, \rro\big)$ and  let $S$ be a large subset of\/ $U$ with witness $\mathscr{W}$. Then for every element  $\X\in \mathscr{V}$ the set $S\cap \sigma(\X)$ is infinite and formula $\gamma_{\mathscr{V}}(\X)$ holds.
\end{lem}
\begin{proof} 
Let $\X\in \mathscr{V}$ then $\psi_{\mathscr{V}}(\X)$.  Let $\Y$ be a  refinement of $\X$ and let $\mathfrak{r}\in \mathfrak{R}$ with $\mathfrak{r}\subseteq \rro(\X)$. Then according to formula $\psi_{\mathscr{V}}(\X)$ there exists a $\rro(\X)$-restriction  $\T\in \mathscr{W}$ of $\Y$.  Because $\T\in \mathscr{W}$ formula $\phi_{\mathscr{W}}(\T)$ holds.  Implying that there exists an $\mathfrak{r}$-restriction $\R\in \mathscr{W}$ of $\T$ and in turn of $\Y$ because a restriction of a refinement of $\Y$ is a restriction of $\Y$. 
 \end{proof}
Note that if $\X\in \mathscr{V}$ and if $\Y$ is a refinement of $\X$ then $\gamma(\Y)$ and $\sigma(\Y)\cap S$ is infinite because a refinement of a refinement is a refinement.  Let $\boldsymbol{\mathscr{V}}:=\{\Y\in \boldsymbol{\mathscr{S}}\mid \exists \X\in \mathscr{V}\, (\text{$\Y$ is a refinement of $\X$)}\}$. If $\Y\in \boldsymbol{\mathscr{V}}$ then  $\Z\in \boldsymbol{\mathscr{V}}$ if $\Z$ is a refinement of $\Y$.  The set $\boldsymbol{\mathscr{V}}$ is the {\em refinement closure} of $\mathscr{V}$.

\begin{thm}\label{thm:fin game}
Let $\mho\big(U, (\boldsymbol{\mathscr{S}};\sqsubseteq), (\mathfrak{R};\subseteq), \sigma, \rro\big)$ and $(S,P)$ be a partition of $U$. Then $S$ or $P$ is large, say $S$. There exists then a  refinement $\V$ of $\U$  and a subset $\boldsymbol{\mathscr{V}}\subseteq \boldsymbol{\mathscr{S}}$ so that $\V\in \boldsymbol{\mathscr{V}}$  and for all $\X\in \boldsymbol{\mathscr{V}}$:
\begin{enumerate}
\item $S\cap \sigma(\X)$ is infinite. 
\item If $\Y$ is a  refinement of $\X$ then $\Y\in \boldsymbol{\mathscr{V}}$. 
\item Formula $\gamma_{\boldsymbol{\mathscr{V}}}(\X)$ holds. 
\end{enumerate}
\end{thm} 
\begin{proof} 
Let $(S,P)$ be a partition of $U$. It follows from Theorem \ref{thm:ramsW} that $S$ or $P$ is large, say $S$ is large, with witness, say $\mathscr{W}$. Let $\mathscr{V}=\{X'\mid \text{$\X\in \mathscr{W}\}$ and $\rro(\X)\supseteq \mathfrak{r}\in \mathfrak{R}$} \}$ and let $\V=\U'$. Let $\boldsymbol{\mathscr{V}}$ be the refinement closure of $\mathscr{V}$. 
\end{proof}

\subsection{$\mho$ for  homogeneous structures with free amalgamation}

For this subsection, let $\mathrm{U}$ be a  countable homogeneous structure whose age has free amalgamation. Let $U$ be the set of elements of $\mathrm{U}$ and let $\boldsymbol{\mathscr{S}}$ denote the set of formed types of $\mathrm{U}$. For two types $\X$ and $\Y$ in $\mathscr{S}$ let $\X\sqsubseteq \Y$ if $\X$ is a successor of $\Y$.   Let $\mathfrak{R}$ denote the set of ranks of the formed  types $\boldsymbol{\mathscr{S}}$. For $\boldsymbol{\mathscr{V}}\subseteq \mathscr{S}$ and $\X\in \mathscr{S}$ the formula $\gamma_{\boldsymbol{\mathscr{V}},\mathscr{S}}(\X)$ is:
\begin{align*}
& \gamma_{\boldsymbol{\mathscr{V}},\mathscr{S}}(\X):=
\begin{cases}
\text{For all  refinements $\Y\in \mathscr{S}$ of $\X$}       &(1)\\
\text{For all $\mathfrak{r}\in \mathfrak{R}$ with $\mathfrak{r}\subseteq \boldsymbol{\rho}(\X)=\boldsymbol{\rho}(\Y)$}   &(2)\\  
\text{there exists an $\mathfrak{r}$-restriction  $\R\in \mathscr{S}$ of $\Y$ with $\R\in \boldsymbol{\mathscr{V}}$}.                          &(3)
\end{cases}
\end{align*}

\begin{defin}\label{defin:constrnewb}
Let $\mathrm{U}$ be a countable homogeneous structure which is rank linear and whose age has free amalgamation and let $U$ be the set of elements of $\mathrm{U}$. Let $\U=\langle \emptyset\tr x\rangle$ and let $S\subseteq U$.  A subset  $\boldsymbol{\mathscr{V}}$ of the set $\boldsymbol{\mathscr{S}}$ of formed types is {\em constructive for $S$}  if there exists a type $\V\in  \boldsymbol{\mathscr{V}}$ with $\rro(\V)=\rro(\U)$ and so that for all  types $\X\in  \boldsymbol{\mathscr{V}}$
\begin{enumerate}
\item $S\cap \sigma(\X)$ is infinite. 
\item If $\Y$ is a  refinement of $\X$ then $\Y\in \boldsymbol{\mathscr{V}}$. 
\item Formula $\gamma_{\boldsymbol{\mathscr{V}},\mathscr{S}}(\X)$ holds. 
\end{enumerate}
 
\end{defin}

The following Theorem \ref{thm:rmsfreeh} is  a direct consequence of Theorem~\ref{thm:fin game} and Lemma \ref{lem:realsortyp}. 

\begin{thm}\label{thm:rmsfreeh}
Let $\mathrm{U}$ be a countable homogeneous structure which is rank linear and whose age has free amalgamation. Let $U$ be the set of elements of $\mathrm{U}$. If $(S,P)$ is a partition of $U$ then there exists a constructive set $\boldsymbol{\mathscr{V}}\subseteq \mathscr{S}$ for $S$ or there exists  a constructive set $\boldsymbol{\mathscr{V}}\subseteq \mathscr{S}$ for $P$. 
\end{thm}

\section{Constructing a monochromatic copy of $\U$}\label{sect:constru1}

Let $\mathrm{U}$ be a  countable, oligomorphic,    free amalgamation  homogeneous relational structure. Let $\mathscr{S}$ be the set of formed types of $\mathrm{U}$.  

Let $(S,P)$ be a partition of $U$. It follows from Theorem \ref{thm:rmsfreeh}   that for one of the parts $S$ or $P$, say $S$, there exists a refinement $\V$ of $\U=\langle \emptyset\tr x\rangle$ and a constructive set  $\boldsymbol{\mathscr{V}}$ of types of $\mathrm{U}$ so that $\V\in \boldsymbol{\mathscr{V}}$  and $S\cap \sigma(\X)$ is infinite for all $\X\in \boldsymbol{\mathscr{V}}$.

 \begin{lem}\label{lem:existastolog}
 Let the structure $\mathrm{U}$ be oligomorphic and let $\mathfrak{R}$ be the set of ranks of the types of $\mathrm{U}$.  Let $\mathcal{C}\subseteq \boldsymbol{\mathscr{V}}$ be a bundle. Associated with every type $\C\in \mathcal{C}$ is a number $n_\C\in\omega$ and an $n_\C$ tuple $\tau(\C)=(\mathfrak{r}_{0,\C}, \mathfrak{r}_{1,\C}, \mathfrak{r}_{2,\C}, \dots, \mathfrak{r}_{n_\C-1,\C})$ with $\rro(\C)\supseteq \mathfrak{r}_i\in \mathfrak{R}$ for all $i\in n_\C$. Then:

   There exists a $^\ast$-successor $\boldsymbol{\mathscr{V}}\supseteq\mathcal{B}=\{\B(\C,i)\mid \text{ $\C\in \mathcal{C}$ and $i\in n_\C$}\}$ of $\mathcal{C}$ so that $\rro(\B(\C,i))= \mathfrak{r}_{i,\C}$ and $\B(\C,i)$ is a successor of $\C$ for every $\C\in \mathcal{C}$ and $i\in n_\C$. 
 \end{lem}
 \begin{proof}
Note that if such a $^\ast$-successor $\mathcal{B}$ exists and $n_\C=0$ for some $\C\in \mathcal{C}$ then the free $(\iota(\mathcal{B})\setminus \iota(\mathcal{C}))$-successor, say $\B$,  is the only successor of  $\mathcal{C}$ in $\mathcal{B}$ and that then $\B^\ast=\emptyset$. 

By induction on the number  $n=\sum_{\C\in\mathcal{C}}n_\C$. Because $\U$ is oligomorphic the bundle $\mathcal{C}$ consists of finitely many types, implying that $n\in \omega$.  If  $n=0$ let $\mathcal{B}$ be the trivial $^\ast$ successor of $\mathcal{C}$.   Let $\mathcal{C}\in \boldsymbol{\mathscr{V}}$ be a bundle, $\tau$ a function which associates as above an $n_\C$ tuple of ranks  with every type $\C\in \mathcal{C}$ and  $\mathcal{B}=\{\B(\C,i)\mid \text{ $\C\in \mathcal{C}$ and $i\in n_\C$}\}$ a $^\ast$-successor of the bundle $\mathcal{C}$ with $\preceq$ the $^\ast$ linear order of $\mathcal{B}$. Let $\D\in \mathcal{C}$ with $\tau(\D)=(\mathfrak{r}_{0,\D}, \mathfrak{r}_{1,\D}, \mathfrak{r}_{2,\D}, \dots, \mathfrak{r}_{n_\D-1,\D})$. Assume  $\rro(\D)\supseteq \mathfrak{r}_{n_\D,\D}\in \mathfrak{R}$. Let $\tau'$ be an association of tuples of ranks  which agrees with $\tau$ on every type $\C\not=\D$ and for which $\tau'(\D)=(\mathfrak{r}_{0,\D}, \mathfrak{r}_{1,\D}, \mathfrak{r}_{2,\D}, \dots, \mathfrak{r}_{n_\D-1,\D}, \mathfrak{r}_{n_\D,\D})$. We have to show that it is possible to extend the Lemma from the tuple assignment $\tau$ to the tuple assignment $\tau'$.

Let $E=(\iota(\mathcal{B})\setminus \iota(\mathcal{C}))$ and let $\D'$ be the free $E$-successor of  $\D$. That is  $\D'$ as a successor of $\D$ has the form $\D\dasharrow$. Hence $\D'\in \mathscr{S}$ according to Lemma \ref{lem:succsuccform} and $\rro(\D')=\rro(\D)$ according to Lemma \ref{lem:freeduch}.    It follows from Definition \ref{defin:constrnewb} and the fact that $\boldsymbol{\mathscr{V}}$ is constructive that   $\D'\in \boldsymbol{\mathscr{V}}$. Hence there exists an $\mathfrak{r}_{n_\D,\D}$-restriction $\R\in \boldsymbol{\mathscr{V}}$ of $\D'$.   Let $E'=\iota(\R)\setminus \iota(\mathcal{C})$ and $F=\iota(\R)\setminus \iota(\D')=E'\setminus E$. Recall Definition \ref{defin:neutralsucc}.    For $\B\in \mathcal{B}$ let $\hat{\B}$ be  the free $((\B\uparrow^\ast)\cup F)$-successor  of $\B_{\downarrow E\setminus (\B\uparrow^\ast)}$. 

Then $\hat{\mathcal{B}}=\{\hat{\B(\C,i)}\mid \text{ $\C\in \mathcal{C}$ and $i\in n_\C$}\}$.  If $\B\in \mathcal{B}$ then $\rro(\hat{\B})=\rro(B_{\downarrow E\setminus (\B\uparrow^\ast}))=\rro(\B)$. Hence $\rro(\hat{\B(\C,i)})=\rro(\B(\C,i))$. For $\{\B,\P\}\subseteq \mathcal{B}$ let $\hat{\B}\preceq \hat{\P}$ if $\B\preceq \P$.  Put $\hat{\B}^\ast=\B^\ast$. Extend $\hat{\mathcal{B}}$ to $\hat{\mathcal{D}}\cup \{R\}$ and $\preceq$ to $\preceq'$  by letting $\R$ be the new $\preceq'$-largest type with $\R^\ast=F$. 

\end{proof}

\begin{lem}\label{lem:jonxtC}
Let $\mathcal{C}\subseteq  \boldsymbol{\mathscr{V}}$ be a bundle and $\Z\in \mathcal{C}$. Then there exists a refinement $\bar{\Z}\in\boldsymbol{\mathscr{V}}$ of $\Z$ so that for every element $x\in \sigma(\bar{\Z})$ exists a refinement $\mathcal{C}'$of $\mathcal{C}_{\uparrow x}$ with $\mathcal{C'}\in \boldsymbol{\mathscr{V}}$.
\end{lem}
\begin{proof}
Let $\mathcal{C}\in \boldsymbol{\mathscr{V}}$ be a bundle and $\Z\in \mathcal{C}$. For $x\in \sigma(\Z)$ and $\C\in \mathcal{C}$  let $\mathrm{ranks}(\C,x)$ be the set of ranks of the types in $\C_{\uparrow x}$. It follows from Note \ref{note:bundlerefxind} that if $y\in \sigma(\Z)$ then $\mathrm{ranks}(\C,y)=\mathrm{ranks}(\C,x)$. Let  $\mathrm{ranks}(\C,\Z)=\mathrm{ranks}(\C,x)$ for any $x\in \sigma(\Z)$. The ranks in the set $\mathrm{ranks}(\C,\Z)$ are subsets of $\rro(\C)$ because they are ranks of successors of $\C$. Let $n_\C=|\mathrm{ranks}(\C,Z)|$ and $\tau(\C)=(\mathfrak{r}_{0,\C}, \mathfrak{r}_{1,\C}, \mathfrak{r}_{2,\C}, \dots, \mathfrak{r}_{n_\C-1,\C})$ a list of those ranks. It follows from Lemma \ref{lem:existastolog} that there exists a $^\ast$-successor $\boldsymbol{\mathscr{V}}\supseteq\mathcal{B}=\{\B(\C,i)\mid \text{ $\C\in \mathcal{C}$ and $i\in n_\C$}\}$ of $\mathcal{C}$ so that  $\rro(\B(\C,i))= \mathfrak{r}_{i,\C}$ and the type $\B(\C,i)$ is a successor of $\C$ for every $\C\in \mathcal{C}$ and $i\in n_\C$. It follows from Lemma \ref{lem:frmeldduu} that the bundle $\mathcal{B}$ is a melding successor of the bundle $\mathcal{C}$. Let $\bar{\Z}$ be the free $(\iota(\mathcal{B})\setminus \iota(\mathcal{C}))$-successor of $\Z$. The type $\bar{\Z}$ is a successor of $\Z$ of the form $\Z\dasharrow$. Hence $\bar{\Z}$ is a formed type. According to Lemma \ref{lem:freeduch} the type $\bar{\Z}$ is a refinement of the type $\Z$. The set $\boldsymbol{\mathscr{V}}$ of formed types is constructive implying $\bar{\Z}\in \boldsymbol{\mathscr{V}}$ according to Definition \ref{defin:constrnewb}.

Let $x\in \sigma(\bar{\Z})$. Then $x$ is in free position with every type in $\mathcal{B}$.  For $\D\in \mathcal{C}_{\uparrow x}$ and $\C=\D_{\downarrow \iota(\mathcal{C})}$ there exists a unique $i\in n_\C$ with $\rro(\D)=\mathfrak{r}_{i,\C}$. Let $\beta(\D)=\B(\C,i)\in \mathcal{B}$. It follows that $\rro(\D)=\rro(\beta(\D))$ and that the bundle $\mathcal{B}$ agrees with the bundle $\mathcal{C}_{\uparrow x}$ with agreement function $\beta$. Hence according to Lemma \ref{lem:conformsuxk} the bundle $\mathcal{B}\stackrel{\beta}{\sqcap}\mathcal{C}_{\uparrow x}$ is a refinement of the bundle $\mathcal{C}_{\uparrow x}$. 

For $\D\in \mathcal{C}_{\uparrow x}$ the types $\beta(\D)\in \mathcal{B}$ and $\D$ are in  free position.  The join $\beta(\D)\sqcap \D$ of $\beta(\D)$ and $\D$ is a successor of $\beta(D)$ having the form $\beta(\D)[\{x\}]$ and $\beta(\D)$ is a formed type. Hence $\beta(\D)\sqcap \D\in \mathscr{S}$ is a formed type. Because $\rro(\bet(\D))=\rro(\D)=\beta(\D)\sqcap \D$ the type $\beta(\D)\sqcap \D$ is a refinement of the type $\beta(\D)$. The set $\boldsymbol{\mathscr{V}}$ of formed types is constructive implying $\beta(\D)\sqcap \D\in \boldsymbol{\mathscr{V}}$ according to Definition \ref{defin:constrnewb}. Put $\mathcal{C}'=\mathcal{B}\stackrel{\beta}{\sqcap}\mathcal{C}_{\uparrow x}$. 

\end{proof}

\begin{thm}\label{thm:manaaaod}
Every oligomorphic, rank linear, free amalgamation homogeneous structure $\mathrm{U}$ is indivisible. 
\end{thm}
\begin{proof}

Let $(u_i\mid i\in \omega)$ be an $\omega$ enumeration of $U$. Let $U_n=\{u_i\mid i\in n\}$. We will construct step by step an embedding $f$ of $\mathrm{U}$ with  $f(u_i)=x_i\in S$ for every $i\in \omega$.  The construction will procede such that for every $1\leq n\in \omega$ and for $\mathcal{C}(n)$ being the bundle of all types with sockel $A(n):=\{x_i\mid i\in n\}$:
\begin{enumerate}
\item The function $f_n$ with  $f_n(u_i)=x_i\in S$ for all  $i\in n$ is an embedding of $\mathrm{U}_{\downarrow \{u_i\mid i\in n\}}$ into $\mathrm{U}$. 
\item There exists a refinement $\mathcal{B}(n)\subseteq \boldsymbol{\mathscr{V}}$ of the bundle $\mathcal{C}(n)$.    
\end{enumerate}
For $n=0$ let $f_0$ be the empty function,  $\mathcal{C}(0)$ be the singleton bundle $\{\langle \emptyset\tr x\rangle\}$ of types and let $\mathcal{B}(0)$ be the singleton bundle $\{\V\}$ of types. It follows from Lemma \ref{lem:begconst}  that the bundle $\{\V\}$ is a refinement of the bundle $\{\langle \emptyset\tr x\rangle\}$. 

Let $f_n$ and $A(n)=\{x_i\in S\mid i\in n\}\subseteq S$ and $\mathcal{C}(n)$ and $\mathcal{B}(n)$ be established.   Let $g_n\in\mathrm{G}$ such that $g_n(u_i)=f_n(u_i)=x_i$ for all $i\in n$.  Let $k$ be the index for which $g_n(u_n)\in \sigma(\C_k(n))$ and $\Z= \B_k(n)$.  Then for every element $x\in \sigma(\Z)$ every extension $h$ of $f_n$ to $u_n$ with $h(u_n)=x$ is an embedding of $\mathrm{U}_{\downarrow \{u_i\mid i\in n+1\}}$ into $\mathrm{U}$.

According to Lemma \ref{lem:jonxtC}  there exists a refinement $\bar{\Z}\in\boldsymbol{\mathscr{V}}$ of $\Z$ so that for every element $x\in \sigma(\bar{\Z})$ exists a refinement $(\mathcal{B}(n))'$of $(\mathcal{B}(n))_{\uparrow x}$ with $(\mathcal{B}(n))'\in \boldsymbol{\mathscr{V}}$. Because $\bar{\Z}\in \boldsymbol{\mathscr{V}}$ the set $S\cap \sigma(\bar{Z})$ is infinite. Pick $x_n\in S\cap \sigma(\bar{Z})$ and let $f_{n+1}$ be the extension of $f_n$ with $f_{n+1}(u_n)=x_n$. Let $(\mathcal{B}(n))'$ be a refinement of  $(\mathcal{B}(n))_{\uparrow x_n}$ with $(\mathcal{B}(n))'\in \boldsymbol{\mathscr{V}}$. There exists according to Lemma \ref{lem:erefx234}  a bundle $\overline{\mathcal{X}}\subseteq (\mathcal{B}(n))_{\uparrow x_n}$ which is an  $x_n$-continuation  of the bundle $\mathcal{C}_{\uparrow z}$. Then $\overline{\mathcal{X}}$ is a refinement of $\mathcal{C}_{\uparrow x_n}$.  Put $\mathcal{B}(n+1)=\{\B\in (\mathcal{B}(n))' \mid \B_{\downarrow \iota(\mathcal{B}(n))}\in \overline{\mathcal{X}}\}$, which is a refinement of $\mathcal{C}_{\uparrow x_n}$. The bundle $\mathcal{C}_{\uparrow x_n}:=\mathcal{C}_{n_1}$ is the bundle of all types with sockel $\{x_i\mid i\in n+1\}$. 

At the end of this process we constructed a sequence of embeddings $f_0\subseteq f_1\subseteq f_2\subseteq f_3\subseteq \dots$ of $U_0\subseteq U_1\subseteq U_2\subseteq U_3\subseteq \dots$ into $U$. Hence  $f:=\bigcup_{1\leq n\in \omega}f_n$ is an embedding of $\U$ into $\U$. Finally $f[U]\subseteq S$ because  $f_n{U_n}\subseteq S$ for all $n\in \omega$. 

\end{proof}

\subsection{Proof of Theorem \ref{thm:Knfrreegen}}\label{subsect:allirred}   Theorem \ref{thm:Knfrreegen}:

\noindent
Let $2\leq k\in \omega$ and let $n> k$. Let $\mathfrak{B}$ be the class of all irreducible k-uniform hypergraphs having at least $n$ vertices. Let $\mathfrak{A}$ be the age of all finite $k$-uniform hypergraphs which do not embed any one of the hypergraphs in $\mathfrak{B}$.  Then $\mathfrak{A}$ is a free amalgamation age. The  countable homogeneous structure $\mathrm{U}$ whose age is $\mathfrak{A}$  is rank linear and hence indivisible.  The linear order of the ranks of the types of $\mathrm{U}$ consists of $n-k+1$ elements. 

\begin{proof}
Let $\T=\langle F\tr x\rangle$ be a type of $\mathrm{U}$. Let $\mathrm{F}$ be the hypergraph with $F$ as set of elements. The hypergraph $\mathrm{F}$ has hyperedges of size $k$ and of size $k-1$. A set $S\subseteq F$ of size $k$ is a hyperedge of $\mathrm{F}$ if and only if it is a hyperedge of $\mathrm{U}$. A set $S\subseteq F$ of size $k-1$ is a hyperedge of $\mathrm{F}$ if and only if the set $S\cup \{x\}$ is a hyperedge of $\mathrm{U}$. For $\mathrm{A}$  a finite $k$-uniform hypergraph with $A\cap F=\emptyset$ let $\mathrm{A}^{(\T)}$ be the hypergraph with set of elements $A^{(\T)}=A\cup F$ and so that a $k$-element subset $S$ of $A^{(\T)}$ is a hyperedge of $\mathrm{A}^{(\T)}$ if and only if $S$ is a hyperedge of $\mathrm{A}$ or $S$ is a hyperedge of $\mathrm{F}$ or $|S\cap A|=1$ and $S\cap F$ is a $k-1$ element hyperedge of $\mathrm{F}$. According to  Fact \ref{fact:contemb} the structure $\mathrm{A}^{(\T)}$ has an embedding $f$ into $\mathrm{U}$ with $f(v)=v$ for all $v\in F$ if it is an element of $\mathfrak{A}$. Note that if $f$ is such an embedding then $f[A] \subseteq \sigma(\T)$. Hence if for a finite $k$-uniform hypergraph $\mathrm{A}$ with $A\cap F=\emptyset$ the corresponding structure $\mathrm{A}^{(\T)}\in \mathfrak{A}$ then $\mathrm{A}\in \rro(\T)$. 

Let $A$ be a finite subset of $\sigma(\T)$. Then the identity map is an embedding of $\mathrm{U}_{\downarrow A\cup F}$ into $\mathrm{U}$ and $\mathrm{A}:=\mathrm{U}_{\downarrow A}\in \rro(\T)$. Note that the hypergraph $\mathrm{A}^{(\T)}$ is obtained from the hypergraph $\mathrm{U}_{\downarrow A\cup F}$ by removing all of its  hyperedges $S$ with $S\cap F\not=\emptyset$ and $|S\cap A|\geq 2$.  If  one of the hypergraphs $\mathrm{B}\in \mathfrak{B}$  has an embedding $h$ into $\mathrm{A}^{(\T)}$ then $\mathrm{U}_{\downarrow h[B]}$  has the same or more hyperedges than $\mathrm{B}$ and hence would be isomorphic to one of the hypergraphs in $\mathfrak{B}$. Every structure $\mathrm{A}\in \rro(\T)$ has an embedding into $\sigma(\T)$. It follows that if $\mathrm{A}\in \rro(\T)$ then $\mathrm{A}^{(\T)}\in \mathfrak{A}$. Hence a finite $k$-uniform hypergraph $\mathrm{A}$ is an element of $\rro(\T)$ if and only if $\mathrm{A}^{(\T)}$ does not embed an element of the boundary $\mathrm{B}$. 

For a finite relational structure $\X$ let $\omega(\X)$ be the largest number for which $\X$ has an induced and irreducible  substructure with $\omega(\X)$ elements. (Generalizing the clique number of a graph.) Let $F'$ be the union of the hyperedges of $\mathrm{F}$ which have $k-1$ elements. Let $\mathrm{F}'=\mathrm{U}_{\downarrow F'}$ and let $\omega'(\T)$ be the number $\omega(\mathrm{F}')$. 
Let $\mathrm{A}$ be a finite $k$-uniform hypergraph with $A\cap F=\emptyset$. For a set $B\subseteq A\cup F$ is the hypergraph  $\mathrm{A}^{(\T)}_{\downarrow B}$ irreducible if and only if both structures $\mathrm{F}_{\downarrow B\cap F'}$ and $\mathrm{A}_{\downarrow B\cap A}$ are irreducible. That is $\omega(\mathrm{A}^{(\T)})=\omega(\mathrm{F}')+\omega(\mathrm{A})$.  Hence $\mathrm{A}\in \rro(\T)$ if and only if $\omega(\mathrm{A})<n-\omega(\mathrm{F}')$. Because  a finite $k$-uniform hypergraph $\mathrm{A}$ is an element of $\rro(\T)$ if and only if $\mathrm{A}^{(\T)}$ does not embed an element of the boundary $\mathrm{B}$ we obtained:

\vskip 3pt
\noindent
(*): $\mathrm{A}\in \rro(\T)$ if and only if $\omega(\mathrm{A})<n-\omega'(\T)$.
\vskip 3pt

It follows that $\U$ is rank linear. Note that if $F'=\emptyset$ then $\omega(\mathrm{F}')=0$ and hence $\rro(\T)=\mathfrak{A}$. If $F'\not=\emptyset$ then $\omega(\mathrm{F}')\geq k-1$. If $\omega(\mathrm{F}')=k-1$ and $n\leq 2k-1$ then $\sigma(\T)$ does not contain any hyperedges. The number $\omega(\mathrm{F}')\leq n-2$ because if $\omega(\mathrm{F}')\geq  n-1$ then $x$ together with any maximal irreducible subset of $F'$ would form an irreducible subset of size $\geq n$. Note that for every natural number $l$ with $k-1\leq l\leq  n-2$ there exists a type $\T(l)=\langle F\tr x\rangle$ with $|F|=l$ and $F'=F$. Let $(\mathrm{R};\subseteq)$ be the linear order of ranks of types of $\mathrm{U}$. If $n=k+1$ then $|\mathfrak{R}|=2$. The rank of $\U$ together with the rank of $\T(k-1)$ which consists of all finite 3-uniform hypergraphs which do not contain any hyperedges. In general then $|\mathfrak{R}|=n-k+1$. 

\end{proof}

\subsection{Proof of Theorem \ref{thm:singlrant}}\label{subsect:onerank}
In the special case in which $|\mathfrak{R}(\mathrm{U})|=1$ there is a simple proof of a stronger version of Theorem \ref{thm:main}. We will use the following result of Peter Cameron, see  \cite{Cameronage} Theorem  3.2. 

\begin{lem}\label{lem:Cameronage}[Cameron]
Let S and R be two countable relational structures for which every structure in the age of $S$ is an element of the age of $\R$. If the structure R is oligomorphic then there exists an embedding of the structure S into the structure R.
\end{lem}

\noindent
Theorem \ref{thm:singlrant}:\\
\noindent
Let $\mathrm{U}$ be a countable, oligomorphic,   homogeneous relational  structure. If $|\mathfrak{R}(\mathrm{U})|=1$ then $\mathrm{U}$ is indivisible.
\begin{proof}
Let $\S$ be a type of $\mathrm{U}$ and $\mathrm{H}$ be the group of automorphisms of the structure $\mathrm{U}_{\downarrow \sigma(\T)}$. Let $\mathrm{G}$ be the automorphism group of the homogeneous structure $\mathrm{U}$. The group $\mathrm{G}_{\iota(\T)}$ acting on $\sigma(\T)$ is oligomorphic and is a subgroup of the group $\mathrm{H}$. Hence the group $\mathrm{H}$ is oligomorphic. 

Let $(S,P)$ be a partition of $U$. If $P$ does not induce a copy of $\mathrm{U}$ then there exists according to Lemma \ref{lem:charercop} a type $\T$ of $\mathrm{U}$ with $\iota(\T)\subseteq P$ and with $\sigma(\P)\subseteq S$. Because $\rro(\T)=\rro(\U)$ which is equal to the age of $\mathrm{U}$ that there exists an embedding of $\mathrm{U}$ into $\sigma(\S)$. 

\end{proof}

Let $\mathrm{U}$ be the homogeneous free amalgamation relational structure with one unary relation $E$ and empty boundary. Let $\boldsymbol{L}$ be this language  of one unary relation.   Let $e\in U$ with $E(e)$. Let $u\in U$ with $\neg E(u)$. Then $\T=\langle \emptyset \tr e\rangle$ is a type with $\sigma(\T)=\{x\in U\mid E(x)\}$. $\S=\langle \emptyset\tr u\rangle$ is a type with $\sigma(\S)=\{x\in U\mid \neg E( x)\}$. The age of $\T$ is the class of all finite $\boldsymbol{L}$ structures in which $E(x)$ for all of its elements. The age of $\S$ is the class of all finite $\boldsymbol{L}$ structures in which $\neg E(x)$ for all of its elements. Hence $\mathrm{U}$ does not satisfy the assumptions of Theorem \ref{thm:singlrant}. It is also clearly not indivisible.

The Rado graph and generalizations of it to graphs with several types of edges are obvious examples of homogeneous structures $\mathrm{U}$  with $|\mathfrak{R}(\mathrm{U})|=1$. The following Lemma  shows that there are plenty of examples of structures $\mathrm{U}$ with $|\mathfrak{R}(\mathrm{U})|=1$. First the following definition. 

\begin{defin}\label{defin:onfyp}
Let $\mathrm{M}$ be a structure. The structure $\mathrm{M}$ is {\em 3-irreducible} if for every 3-element subset $\{x,y,z\}$ of $M$ there exists a tuple $\vec{a}$ with entries in $M$ and a relation symbol $R$ in the language of $\mathrm{M}$ with $\{x,y,z\}\subseteq \vec{a}$ and $R(\vec{a})$.   That is if the 3-Gaifman hypergraph is complete.

A set  $A\subseteq M$ is  {\em $\mathrm{M}$-type conform} if $|A|\geq 2$, $M\setminus A\not=\emptyset$  and   for all $\{x,y\}\subseteq A$ there exists an isomorphism $f: (M\setminus A)\cup \{x\}\to (M\setminus A)\cup \{y\}$  with $f(x)=y$ and  $f(v)=v$ for all $v\in M\setminus A$. 
\end{defin}

Let $\T=\langle F\tr x\rangle$ be a type of $\mathrm{U}$. Let $M$ be a finite subset of $F\cup \sigma(\T)$ with $M\cap F\not=\emptyset$ and with $M\cap \sigma(F)\geq 2$. Let $\mathrm{M}=\mathrm{U}_{\downarrow M}$. Then $M\cap \sigma(\T)$ is $\mathrm{M}$-type conform. It follows that if a structure  $\mathrm{A}\in \mathfrak{A}(\mathrm{U})$ is not a structure in $\rro(\T)$ then there exists a boundary structure $\mathrm{B}$ for which there exists an induced substructure $\mathrm{A}'$ of $\mathrm{A}$ and an embedding $h$ of $\mathrm{A}'$ into $\mathrm{B}$ so that: $|A'|\geq 2$,   the set $h[A']$ is $\mathrm{B}$-conform, for $a\in A'$ exists an embedding $k$ of the substructure of $\mathrm{B}$ induced by $\{a\}\cup B\setminus A'$ into $\mathrm{U}$ with $k(a)\in \sigma(\T)$ and $k[B\setminus A']\subseteq F$. Conversely assume that the boundary structure $\mathrm{B}$ has conformal subset $A$. Let $a\in A$.  Then the substructure $\mathrm{C}$ of $\mathrm{B}$ induced by $B\setminus A\cup \{a\}$ is in the age of $\mathrm{U}$. There exists an embedding $k$ of $\mathrm{C}$ into $\mathrm{U}$. Then the structure $\mathrm{A}'$ is not in the rank of  the type $\langle k[B\setminus A\tr k(a)\rangle$.  We conclude that the ranks of the typesets of a homogeneous structure $\mathrm{U}$ depend on the partitions of the boundary structures into conformal subsets and their complements. 
Actually, if $n\geq 2$ is the smallest arity of the relations in the language of $\mathrm{U}$ we can disregard the conformal subsets $A$ of a boundary structure $\mathrm{B}$ with $|B\setminus A|<n-1$.

\begin{lem}\label{lem:singlrant}
Let $\mathrm{U}$ be a countable and oligomorphic relational homogeneous structure, then:
\begin{enumerate}
\item If every structure in the boundary of\/  $\mathrm{U}$ consists of two elements or is 3-irreducible then $|\mathfrak{R}(\mathrm{U})|=1$ and hence  $\mathrm{U}$ is indivisible.
\item If no structure $\mathrm{M}$  in the boundary of\/  $\mathrm{U}$ contains an $\mathrm{M}$-type conform set $A\subseteq M$  then $|\mathfrak{R}(\mathrm{U})|=1$ and hence $\mathrm{U}$ is indivisible.
\item If for every structure $\mathrm{M}$, with $|M|\geq 3$,   in the boundary of\/  $\mathrm{U}$:   For all $\mathrm{M}$-type conform sets $A\subseteq M$ there exists an edge $\{a,b,c\}$ in the 3-Gaifman hypergraph of $\mathrm{M}$ with $\{a,b\}\subseteq  A$ and $c\in M\setminus A$.  Then $|\mathfrak{R}(\mathrm{U})|=1$ and hence $\mathrm{U}$ is indivisible.

\item If for every structure $\mathrm{M}$, with $|M|\geq 3$,   in the boundary of\/  $\mathrm{U}$:  For all $\mathrm{M}$-type conform sets $A\subseteq M$ there exists a three element set $\{a,b,c\}$ which is not in the 3-Gaifman hypergraph of $\mathrm{M}$ with $\{a,b\}\subseteq  A$ and $c\in M\setminus A$.  Then $|\mathfrak{R}(\mathrm{U})|=1$ and hence $\mathrm{U}$ is indivisible. 
\end{enumerate}
\end{lem}
\begin{proof}
Note that Item (3) implies Items (1) and (2). For both Items (3) and (4) we have to prove that given a  type  $\T=\langle F\tr x\rangle$ and a structure   $\mathrm{A}\in \mathfrak{A}(\mathrm{U})$ there exists an embedding of $\mathrm{A}$ into $\sigma(\T)$. We may assume that $A\cap U=\emptyset$.  Let $\mathrm{M}$ be the structure with $M=F\cup A$ for which $\mathrm{U}_{\downarrow F}=\mathrm{M}_{\downarrow F}$.  For every $a\in A$ is the function $f: F\cup \{x\}\to F\cup \{a\}$ with $f(x)=a$ and $f(v)=v$ for all $v\in F$ an isomorphism of $\mathrm{U}_{\downarrow F\cup \{x\}}$ to   $\mathrm{M}_{\downarrow F \cup \{a\}}$. Then:

\vskip 3pt
\noindent
In the case of Item (3): In addition,  the structure $\mathrm{M}$ has the property that  if $R_\mathrm{M}(\vec{x})$ for a relation symbol $R$ and a tuple $\vec{x}$ with entries in $M$,   then all entries of $\vec{x}$ are in $A$  or all entries of $\vec{x}$  except possibly one are in $F$.  That is there is no edge of the 3-Gaifman graph having two elements in $A$ and one element in $F$. Let $B\subseteq M$.

The structure  $\mathrm{M}_{\downarrow B}$ can not be isomorphic to a structure in the boundary, because:  If $|B\cap A|\geq 2$ and $B\cap F\not=\emptyset$ then $B\cap A$ would be $\mathrm{M}$-type conform and not contain an edge of the 3-Gaifman graph having two elements in $B\cap A$ and one element in $F$. If $B\subseteq A$ then $A$ would not be in the age of $\mathrm{U}$. If $B\cap A=\{a\}$ for some element $a\in A$ then $\mathrm{B}$ would have an embedding into $\mathrm{U}$. If $|B|=2$ and $B$ not in $A$ or not in $F$ then $|B\cap A|=1=|B\cap F|$.   It follows that $\mathrm{M}$ is in the age of $\mathrm{U}$. There exists an embedding $f$ of $\mathrm{M}$ with $f(v)=v$ for all $v\in F$. Implying $f[A]\subseteq \sigma(\T)$.   Let $g$ be an extension of $f$ to an automorphism of $\mathrm{U}$.

\vskip 3pt
\noindent
In the case of Item (4): In addition,  the structure $\mathrm{M}$ has the property that for every two elements $a$ and $b$ in $A$ and every element $v\in F$ the set $\{a,b,v\}$ is an edge of the 3-Gaifman hypergraph of $\mathrm{M}$. Let $B\subseteq M$.

The structure  $\mathrm{M}_{\downarrow B}$ can not be isomorphic to a structure in the boundary, because:  If $|B\cap A|\geq 2$ and $B\cap F\not=\emptyset$ then $B\cap A$ would be $\mathrm{M}$-type conform and  contain an edge of the 3-Gaifman graph having two elements in $B\cap A$ and one element in $F$. If $B\subseteq A$ then $A$ would not be in the age of $\mathrm{U}$. If $B\cap A=\{a\}$ for some element $a\in A$ then $\mathrm{B}$ would have an embedding into $\mathrm{U}$. If $|B|=2$ and $B$ not in $A$ or not in $F$ then $|B\cap A|=1=|B\cap F|$. It follows that $\mathrm{M}$ is in the age of $\mathrm{U}$. There exists an embedding $f$ of $\mathrm{M}$ with $f(v)=v$ for all $v\in F$. Implying $f[A]\subseteq \sigma(\T)$.   Let $g$ be an extension of $f$ to an automorphism of $\mathrm{U}$.

\end{proof}

If  there exists  for every structure $\mathrm{B}$ in the boundary of $\mathrm{U}$ a relation $R$ in the language of $\mathrm{U}$ for which $R_\mathrm{B}(\vec{x})$ so that every element of $B$ is an entry of $\vec{x}$, then $\mathrm{U}$ satisfies the conditions of Item (1) in Lemma \ref{lem:singlrant}. Providing examples generalizing the situation of graphs with several types of edges.

\subsection{Weak indivisibility}\label{subsec:weakindiv2}

In the case that $\mathrm{U}$ is binary Section \ref{sect:aingind} is not needed except for the fact that the typesets of types of $\mathrm{U}$ are age indivisible. In the binary case this also follows from Corollary \ref{cor:aindibinn}  below. The proof of Theorem \ref{thm:wiekstru5r} may help to follow the proof of Theorem \ref{thm:hennsage}. The essential ingredients of the proof of Theorem \ref{thm:wiekstru5r} appeared in \cite{EZS3}.

Theorem~\ref{thm:wiekstru5r} deals just with group actions. Hence the age function $\rro$ has to be understood as $\rro_\mathrm{G}$.  Then $\rro(U)$ is the set of finite subsets of $U$.  A subgroup $\mathrm{G}$ of the symmetric group $\mathfrak{S}(U)$ is a {\em free amalgamation group} if the closure $\overline{\mathrm{G}}$ of $\mathrm{G}$ is the automorphism group of a free amalgamation homogeneous structure with domain $U$. It follows from the discussion in Subsection \ref{subsect:gropversio}  that Lemma  \ref{lem:freecoppr} and Corollary \ref{cor:commeldstrm} apply for the types of any free amalgamation group $\mathrm{G}$.

\begin{thm}\label{thm:wiekstru5r}  Let $\mathrm{G}$ be a vertex transitive free amalgamation group acting on a countable infinite set $U$.   Let  $\mathfrak{c}: U\to 2$ be a two colouring of $U$. Then, if $\rro(\mathfrak{c}^{-1}(i))\not=\rro(U)$ the set   $\mathfrak{c}^{-1}(i+1)$   embeds a copy of $\mathrm{G}$. 

\end{thm}

\begin{proof}
 Consider the proposition $P(n)$:  If $\rro(\mathfrak{c}^{-1}(0))$ does not contain some set $A\in \rro(U)$ with $|A|=n$, then $\mathfrak{c}^{-1}(1)$  embeds a copy of $\mathrm{G}$.  The statement $P(n)$ implies the statement $P^\ast(n)$: If there exists a copy $\mathrm{D}$ of  $\mathrm{G}$ for which $\rro(\mathfrak{c}^{-1}(0)\cap D)$ does not contain some set  $A\in \rro(U)$ with $|A|=n$, then $\mathfrak{c}^{-1}(1)\cap D$  embeds a copy of $\mathrm{G}$.  If $P(1)$ does not hold then $U= \mathfrak{c}^{-1}(1)$ because $\mathrm{G}$ acts transitively. We will prove that $P(n)$ implies $P(n+1)$. 
 
 Assume that the set  $\mathfrak{c}^{-1}(1)$ does not embed a copy of $\mathrm{G}$. Then 
$\mathfrak{c}^{-1}(1)$ is not a copy of $\mathrm{G}$.  There exists, according to  Lemma~\ref{lem:charercop}, a type $\T$ of  $\mathrm{G}$ with $\iota(\T)\subseteq \mathfrak{c}^{-1}(1)$ and with $\sigma(\T)\subseteq \mathfrak{c}^{-1}(0)$. Let $\X$ be the free type with $\iota(\X)=\iota(\T)$. It follows from Lemma \ref{lem:freecoppr} that $\sigma(\X)$ induces a copy of  $\mathrm{G}$. Let $A\in \rro(U)$ with $|A|=n+1$. Let $a\in A$ and $\mathrm{B}$ be the restriction of $\mathrm{A}$ to  $A\setminus \{a\}$.  $P^\ast(n)$ implies that there exists a function  $h\in \mathrm{G}$ with   $h[B]\subseteq \mathfrak{c}^{-1}(0)\cap \sigma(\X)$.  Let $\S=\langle h[B]\tr h(a)\rangle$.  It follows from Corollary \ref{cor:commeldstrm} that there is an element $z\in \sigma(\T)\cap \sigma(\S)$. Hence  $z\in \mathfrak{c}^{-1}(0)$ and $z\in \sigma(\S)$.  Implying that $A\in \rro(\mathfrak{c}^{-1}(0))$. 
\end{proof}

Translating to homogeneous structures, we obtain Theorem \ref{thm:wiekstruncor}:

\vskip 3pt
\noindent
Let $\mathrm{U}$ be a free amalgamation homogeneous relational structure with a transitive automorphism group $\mathrm{G}$. Let $\mathfrak{A}$ be the age of $\mathrm{U}$.  Let  $\mathfrak{c}: U\to 2$ be a two colouring of $U$. Then, if  $\rro(\mathrm{U}_{\downarrow (\mathfrak{c}^{-1}(i))})\not=\mathfrak{A}$ the induced substructure    $\mathrm{U}_{\downarrow \mathfrak{c}^{-1}(i+1)}$   embeds a copy of $\mathrm{U}$. (Hence $\mathrm{U}$ is age indivisible.)

\begin{cor}\label{cor:aindibinn}
Let $\mathrm{U}$ be a binary, free amalgamation homogeneous relational structure. Then the typesets of types of $\mathrm{U}$ are age indivisible.
\end{cor}
\begin{proof}
Let $\T$ be a type of $\mathrm{U}$. It follows from Lemma \ref{cor:freambinmeld} that the structure $\mathrm{U}_{\downarrow \sigma(\T)}$ is a free amalgamation homogeneous structure with  $\mathrm{G}_{\iota(\T)}$ as group of automorphisms. It follows from Theorem \ref{thm:wiekstruncor} that $\sigma(\T)$ is age indivisible. 
\end{proof}

\section{Age indivisibility of the groups $\mathrm{G}_F$}\label{sect:aingind}

For this section let $\mathrm{U}$ be a countable free amalgamation homogeneous structure and let $\mathrm{G}$ be the group of automorphisms of $\mathrm{U}$. See Subsection \ref{subsect:gropversio} for the notion of $\mathrm{G}$-age, that is $\rro_\mathrm{G}(S)$,  of a set $S\subseteq U$. 

\begin{defin}\label{defin:copacco9m}
Let $\mathcal{B}=\{\B_j\mid j\in n\in\omega \}$ be a bundle of types of $\mathrm{U}$.  

A function $f: \iota(\mathcal{B})\cup \sigma(\mathcal{B})\to \iota(\mathcal{B})\cup \sigma(\mathcal{B})$ is an {\em embedding of $\mathcal{B}$ into $\mathcal{B}$ if for every finite $A\subseteq \iota(\mathcal{B})\cup \sigma(\mathcal{B})$ there exists a function $g\in \mathrm{G}_{\iota(\mathcal{B})}$ whose restriction to $A$ is equal to the restriction of $f$ to $A$. The image of an embedding of $\mathcal{B}$ into $\mathcal{B}$ is a {\em copy of $\mathcal{B}$}. (In particular $\iota(\mathcal{B})\cup \sigma(\mathcal{B})$ is a copy of $\mathcal{B}$ and if $f$ is an embedding of $\mathcal{B}$  then $f(v)=v$ for all $v\in \iota(\mathcal{B})$.)  If  $\iota(\mathcal{B})\cup \sigma(\mathcal{B})=U$} then $f$ is an embedding of the homogeneous structure $\mathrm{U}$.

Let $C$ be a copy of $\mathcal{B}$ and $j\in n$. A type $\T=\langle \iota(\mathcal{B})\cup F\tr x\rangle$ is an {\em $F$-successor of $\B_j$ within $C$} if $\T$ is an $F$-successor of $\B_j$ and if $F\subseteq C\setminus\iota(\mathcal{B})$.  (Note that then $x\in \sigma(\B_j)$.) Let $\T=\langle \iota(\mathcal{B})\cup F\tr x\rangle$ be an $F$-successor of $\B_j$ within $C$. A copy $D$ of $\mathcal{B}$ is {\em neutral to $\T$ within $C$} if $D\subseteq C$ and if $\sigma(\T)\cap \sigma(\S)\cap C\not=\emptyset$ for every type $\S=\langle \iota(\mathcal{B})\cup E\tr y\rangle$   which is an $E$-successor of $\B_j$ within $D$.

For a set  $S\subseteq  \sigma(\mathcal{B})$ let $\rrro(S)=$
\[
\{A\subseteq \sigma(\mathcal{B})\mid \text{$A$ is finite and there exists $g\in\mathrm{G}_{\iota(\mathcal{B})}$ with $g[A]\subseteq S$}\}.
\]
The   set $S$ is {\em $\mathcal{B}$-age complete} if $\rrro(S)=\rrro(\sigma(\mathcal{B}))$, that is if for every finite set  $A\subseteq \sigma(\mathcal{B})$ there exists a function $g\in\mathrm{G}_{\iota(\mathcal{B})}$ with $g[A]\subseteq S$. (In particular every copy of $\mathcal{B}$ is $\mathcal{B}$-age complete.) 
\end{defin}

Observe: The group $\mathrm{G}_{\iota(\mathcal{B})}$ acts on $\sigma(\mathcal{B})$. It is possible to extend the relational structure $\mathrm{U}_{\downarrow \sigma(\mathcal{B})}$ by adding further relations to obtain a free amalgamation homogeneous structure $\mathrm{H}$ with $\mathrm{G}_{\iota(\mathcal{B})}$ as group of automorphisms.  Then $\rrro(S)=\rro_{\mathrm{G}_{\iota(\mathcal{B})}}(S)$ for $S\subseteq \sigma(\mathcal{B})$.

\begin{lem}\label{lem:tsetwicopy}
Let $\mathcal{B}=\{\B_j\mid j\in n \}$ be a bundle and $C$ be   a copy of $\mathcal{B}$.   Let $\T=\langle \iota(\mathcal{B})\cup F\tr x\rangle$ be an $F$-successor of $\B_j$ within $C$ for some index $j\in n$. Then $\sigma(\T)\cap \sigma(\B_j)\cap C\not=\emptyset$. (Note $\sigma(\T)\cap \sigma(\B_j)\cap C=\sigma(\T)\cap C$ because $\sigma(\T)\subseteq \sigma(\B_j)$.)
\end{lem}
\begin{proof}
Let $f$ be an embedding of $\mathcal{B}$ into $\mathcal{B}$ whose image is $C$ and let $g\in \mathrm{G}_{\iota(\mathcal{B})}$ which agrees with $f$ on $\iota(\mathcal{B})\cup f^{-1}[F]$. Let  $g^{-1}[\T]=\langle \iota(\mathcal{B})\cup g^{-1}[F]\tr g^{-1}(x)\rangle$. Clearly $f(g^{-1})(x)\in C$.  The function $f$ is an embedding of $\mathcal{B}$ into $\mathcal{B}$  and hence there exists a function $h\in \mathrm{G}_{\iota(\mathcal{B})}$ which agrees with $f$ on $F\cup g^{-1}(x)$.  The function $h\circ g^{-1}\in \mathrm{G}_{\iota{\mathcal{B}}\cup F}=\mathrm{G}_{\iota(\T)}$ and $x\in \sigma(\T)$.  Hence $f\circ g^{-1}(x)=h\circ g^{-1}(x)\in \sigma(\T)$.  
\end{proof}

\begin{lem}\label{lem:neutcop7}
Let $\mathcal{B}=\{\B_j\mid j\in n \}$ be a bundle and let $\T=\langle \iota(\mathcal{B})\cup F\tr x\rangle$ be an $F$-successor of $\B_j$ within the copy $\iota(\mathcal{B})\cup \sigma(\mathcal{B})$ of $\mathcal{B}$ for some index $j\in n$. Then there exists a copy $D$ of  $\mathcal{B}$ which is neutral to $\T$ within the copy $\iota(\mathcal{B})\cup \sigma(\mathcal{B})$. 
\end{lem} 
\begin{proof}
Let $\{u_i\mid i\in \omega\}$ be an $\omega$-enumeration of $\sigma(\mathcal{B})$. The set $D\setminus \iota(\mathcal{B})=\{v_i\mid i\in \omega\}$ will be constructed recursively so that for every $i\in n$ there exists a function $f_n\in \mathrm{G}_{\iota(\mathcal{B})}$ with $f_n(u_i)=v_i$ and so that the set $\{v_i\mid i\in n\}$ is in free position to the set $\{u_i\mid i\in n\}$. If $\{v_i\mid i\in n\}$ has been constructed let $\S=\langle \{v_i\mid i\in n\}\tr f(u_n)\rangle$. Using to Lemma \ref{lem:freeduch}  pick $v_n\in \sigma(\R)$ for $\R$ being the free $\{u_i\mid  i\in n\}$-successor of the type $\S$ and pick $h\in \mathrm{G}_{\iota(\mathcal{B})\cup \{v_i\mid i\in n\}}$ with $h(f_n(u_n))=v_n$. Let $f_{n+1}=h\circ f_n$. The function $f=\bigcup_{n\in \omega}f_n$ has then the property that it agrees on every finite subset with a function in $\mathrm{G}_{\iota(\mathcal{B})}$ implying that $f[\iota(\mathcal{B})\cup \sigma(\mathcal{B})]:=D$ is a copy of $\mathcal{B}$.   

Let $\S=\langle \iota(\mathcal{B})\cup E\tr y\rangle$ be an $E$-successor of $\B_j$ within the copy $D$.  The types $\T$ and $\S$ are compatible and in free position. Hence it follows from Corollary \ref{cor:commeldstrm} that there exists an element $z\in \sigma(\T)\cap \sigma(\S)$.   
\end{proof}

\begin{lem}\label{lem:neutcop33}
Let $\mathcal{B}=\{\B_j\mid j\in n \}$ be a bundle and $C$ a copy of $\mathcal{B}$  and let $\T=\langle \iota(\mathcal{B})\cup F\tr x\rangle$ be an $F$-successor of $\B_j$ within $C$ for some index $j\in n$. Then there exists a copy $D$ of  $\mathcal{B}$ which is neutral to $\T$ within $C$. 
\end{lem} 
\begin{proof}
Let $f$ be an embedding whose image is $C$ and let $g\in \mathrm{G}_{\iota(\mathcal{B})}$ which agrees with $f$ on $f^{-1}[\iota(\mathcal{B})\cup F]$.   The type $g^{-1}[\T]= \langle f^{-1}[\iota(\mathcal{B})\cup F]\tr g^{-1}(x)\rangle$ is an $f^{-1}[F]$ successor of $\B_j$ within the copy $\iota(\mathcal{B})\cup \sigma(\mathcal{B})$ of $\mathcal{B}$. According to Lemma \ref{lem:neutcop7} there exists a copy $N$ of  $\mathcal{B}$ which is neutral to $g^{-1}[\T]$ within the copy $\iota(\mathcal{B})\cup \sigma(\mathcal{B})$. Let $h$ be an embedding of $\mathcal{B}$ whose image is $N$. Let $D$ be the copy which is the image of the embedding $f\circ h$. Then $D\subseteq C$. Note that if $l\in \mathrm{G}_{\iota(\mathcal{B})}$ which agrees with $f$ on $f^{-1}[\iota(\mathcal{B})\cup F]$ then $l^{-1}[\T]=g^{-1}[\T]$. 

Let $\S=\langle \iota(\mathcal{B})\cup E\tr y\rangle$  be is an $E$-successor of $\B_j$ within $D$ and let $l\in \mathrm{G}_{\iota(\mathcal{B})}$ which agrees with $f$ on $f^{-1}[\iota(\mathcal{B})\cup F \cup E]$. Then the type $l^{-1}[\S]=\langle \iota(\mathcal{B})\cup f^{-1}[E]\tr l^{-1}(y)\rangle$ is an $f^{-1}[E]$-successor of $\B_j$ within the copy $N$. Implying that there exists an element $z\in \sigma(l^{-1}[\T])\cap \sigma(l^{-1}[\S])=\sigma(g^{-1}[\T])\cap \sigma(l^{-1}[\S])$.  Then $f(z)\in \sigma(\T)\cap \sigma(\S)\cap C$.  
\end{proof}

\begin{defin}\label{defin:Stransdef}
Let $\mathcal{B}=\{\B_j\mid j\in n\}$ be a bundle and $S\subseteq \sigma(\mathcal{B})$.  

Then $P=(P_{i,j};\text{$i\in m_j\in \omega$ and $j\in n$})$    is a {\em colouring  of $S$} if $P_{i,j}\cap P_{i',j'}=\emptyset$ unless $i=i'$ and $j=j'$ and if\/  $\bigcup_{i\in m_j}P_{i,j}=S\cap \sigma(\B_j)$  for every $j\in n$.  The colouring $P$ is an $m$-colouring if $\max\{m_j\mid j\in n\}=m$. The colouring $P$ is a {\em colouring of $\mathcal{B}$} if it is a colouring of $\sigma(\mathcal{B})$. 

For a given colouring $P=(P_{i,j};\text{$i\in m_j\in \omega)$ and $j\in n$}\}$   let $\mathcal{E}$, or $\mathcal{E}(P)$ if the distinction is necessary, be the set of functions $\epsilon$ with $\epsilon(j)\in m_j$ for all $j\in n$. For a function $\epsilon\in \mathcal{E}$ let $P_\epsilon:=\bigcup_{j\in n}P_{\epsilon(j),j}$.  
\end{defin}

\begin{defin}\label{defin:ageindn}
Let $\mathcal{B}=\{\B_j\mid j\in n\}$ be a bundle and $S\subseteq \sigma(\mathcal{B})$.  The set $S$ is {\em $\mathcal{B}$-age indivisible} if for every colouring $P=(P_{i,j})$ of $S$ there exists a function $\epsilon\in \mathcal{E}$ for which  $\rrro(P_\epsilon)=\rrro(S)$. The bundle $\mathcal{B}$ is {\em $\mathcal{B}$-age indivisible} if the set $\sigma(\mathcal{B})$ is  $\mathcal{B}$-age indivisible. A type $\T$ is {\em age indivisible} if for every partition $(S_0,S_1)$ of $\sigma(\T)$ there exists an $i\in 2$ so that for all finite subsets $A$ of $\sigma(\T)$ there exists a function $f\in \mathrm{G}_{\iota(\T)}$ with $f[A]\subseteq S_i$. 

Let $P=(P_{i,j})$ be a colouring of $\mathcal{B}$  and let $\epsilon'\in \mathcal{E}$. For a set $J\subseteq n$ of indices and a function $\epsilon'\in \mathcal{E}$ let $\mathcal{E}_{J,\epsilon'}$ be the set of functions in $\mathcal{E}$ which agree with $\epsilon'$ on $J$.  A copy $C$ of $\mathcal{B}$ is {\em $(J,\epsilon')$-monochromatic} if  $C\cap \sigma(\B_j)\subseteq P_{\epsilon'(j),j}$ for every index $j\in J$. 

The bundle $\mathcal{B}=\{\B_j\mid j\in n\}$    is {\em weakly indivisible} if for every $2$-colouring $P=(P_{i,j};i\in 2, j\in n)$ of $\mathcal{B}$ exists a set $J\subseteq n$ of indices and a function $\epsilon'\in \mathcal{E}$  and a copy $C$ of $\mathcal{B}$ which is $(J,\epsilon')$-monochromatic and so that $\rrro(P_\epsilon\cap C)=\rrro(\sigma(\mathcal{B}))$ for every function $\epsilon\in \mathcal{E}_{J,\epsilon'}$.  
\end{defin}

\begin{thm}\label{thm:hennsage}
Let $\mathrm{U}$ be a countable free amalgamation homogeneous structure. Then every bundle  $\mathcal{B}=\{\B_j\mid j\in n\in \omega\}$ of types of $\mathrm{U}$ is weakly indivisible. 
\end{thm}
\begin{proof}
Let $\mathcal{B}=\{\B_j\mid j\in n\in \omega\}$ be a bundle  of types of $\mathrm{U}$ and let $P=(P_{i,j};i\in 2, j\in n)$ be a $2$-colouring of $\mathcal{B}$.

Let $J$ be a largest subset of $n$ for which there exists a function $\epsilon'\in \mathcal{E}$ and a copy $X$ of $\mathcal{B}$ which is $(J,\epsilon')$-monochromatic. The statement: $A\in \rrro(P_\epsilon\cap Y)$ for every number $k\in \omega$ and for every copy $Y\subseteq X$ of $\mathcal{B}$ and every $\epsilon\in \mathcal{E}_{J,\epsilon'}$ and every set $A\in \rrro(\sigma(\mathcal{B}))$ with $|A|=k$, implies the Theorem. If this statement does not hold, then  there exists  a largest number $k\in \omega$ so that $A\in \rrro(P_{\overline{\epsilon}}\cap Y)$ for  every $\overline{\epsilon}\in \mathcal{E}_{J,\epsilon'}$ and  for every copy $Y\subseteq X$ of $\mathcal{B}$  and every set $A\in \rrro(\sigma(\mathcal{B}))$ with $|A|=k$. Let $\epsilon\in \mathcal{E}_{J,\epsilon'}$ and $C\subseteq X$ be a copy of $\mathcal{B}$ and $A\in \rrro(\sigma(\mathcal{B}))$ with $|A|=k+1$ for which $A\not\in \rrro(P_\epsilon\cap C)$.

The set $A$ is an element of $\rrro(C)$ because $C$ is a copy of $\mathcal{B}$.   If $A\cap \sigma(\B_j)=\emptyset$ for all $j\in n\setminus J$ then $A\in \rrro(P_{\epsilon'}\cap C)$, because $C$ is $(J,\epsilon')$-monochromatic.   Hence  $A\in \rrro(P_{\epsilon}\cap C)$ because $\epsilon$ and $\epsilon'$ agree on $J$. Let $a\in A\cap \sigma(\B_j)$  with $j\in n\setminus J$ and assume without loss that $\epsilon(j)=1$. 

Enumerate $\sigma(\mathcal{B})$ into an $\omega$-sequence $(u_i;i\in \omega)$ and begin to construct recursively a sequence $(v_i\in C;i\in \omega)$ which has the property that the function $f: \iota(\mathcal{B})\cup \sigma(\mathcal{B})\to \iota(\mathcal{B})\cup \{v_i\mid i\in \omega\}$ with $f(u_i)=v_i\in C$ for all $i\in \omega$ and with $f(x)=x$ for all $x\in\iota(\mathcal{B})$ is an embedding and so that if  $v_i\in\sigma(\B_j)$ then  $v_i\in P_{0,j}$. If successful we would  obtain a copy $X\subseteq D$ of $\mathcal{B}$ which violates the maximality of the set $J$ of indices. Hence  there is an $m\in \omega$ and a type $\T=\langle \iota(\mathcal{B})\cup \{v_i\mid i\in m\}|x\rangle$ with $\sigma(\T)\subseteq \sigma(\B_j)$ and with  $\sigma(\T)\cap C\subseteq P_{1,j}$. Let $F=\{v_i\mid i\in m\}$. Then $\T$ is an  $F$-successor of $\B_j$ within $C$. 

It follows from  Lemma \ref{lem:neutcop33} that there exists a copy $D$ of  $\mathcal{B}$ which is neutral to $\T$ within $C$. Then $E :=A\setminus \{a\}\in \rrro(P_\epsilon\cap D)$. Hence there exists a function $h\in \mathrm{G}_{\iota(\mathcal{B})}$ with $h[E]\subseteq P_\epsilon\cap D\subseteq P_\epsilon\cap C$. The type  $\S=\langle \iota(\mathcal{B})\cup h[E]\tr h(a)\rangle$ is an $h[E]$-successor of $\B_j$ within $D$. Because $D$ is neutral to $\T$ within $C$ there exists and element $z\in \sigma(\T)\cap \sigma(\S)\cap C$. Then $z\in P_{1,j}\subseteq P_\epsilon$ because $z\in \sigma(\T)\cap C$. There exists a function $g\in \mathrm{G}_{\iota(\mathcal{B})\cup h[E]}$ with $z=g\circ h(a)$. Implying that the function $g\circ h\in \mathrm{G}_{\iota(\mathcal{B})}$ maps $A$ into $P_\epsilon\cap C$ in contradiction to $A\not\in \rrro(P_\epsilon\cap C)$.
\end{proof}

\begin{cor}\label{cor:hennsage}
Let $\mathcal{B}=\{\B_j\mid j\in n\}$ be a bundle.  Let $P=(P_{i,j};i\in 2, j\in n)$ be a $2$-colouring of $\mathcal{B}$. Then there exists a function $\epsilon\in \mathcal{E}(P)$ with $\rrro(P_\epsilon)=\rrro(\sigma(\mathcal{B}))$. 
\end{cor}

If $B$ is a transitivity  class of the group of automorphisms $\mathrm{G}$ of $\mathrm{U}$ then $\B=\langle \emptyset\tr x\rangle$,  with $x\in B$,  is a type of $\mathrm{U}$ with $\sigma(\B)=B$. Hence the following Theorem \ref{thm:wieklyindivun}, stating that every oligomorphic free amalgamation countable homogeneous structure is {\em weakly indivisible},   is a special case of Theorem \ref{thm:hennsage}.    

\begin{thm}\label{thm:wieklyindivun}
Let $\mathrm{U}$ be a countable free amalgamation homogeneous structure for which the group of automorphisms $\mathrm{G}$ of $\mathrm{U}$ has only finitely many transitivity classes. Let $\mathcal{B}=\{B_j\mid j\in n\in \omega\}$ be the partition of $U$ into the transitivity classes of $\mathrm{G}$. Let $\mathfrak{c}: U\to 2$ be a two colouring of $U$ and let $\mathcal{E}$ be the set of all functions $\epsilon: n\to 2$. 

Then there exists  a copy $C$ of\/ $\mathrm{U}$ and a set $J\subseteq n$ and a function $\epsilon'\in \mathcal{E}$ so that $C\cap B_j\subseteq \mathfrak{c}^{-1}(\epsilon'(j))$ for every $j\in J$ and so that the class of finite structures which have an embedding into the  restriction of\/ $\mathrm{U}$ to the set $\bigcup_{j\in n}\mathfrak{c}^{-1}\epsilon(j)$ is the age of $\mathrm{U}$ for every function $\epsilon\in \mathcal{E}$ which agrees with $\epsilon'$ on $J$. 
\end{thm}

Let $\mathrm{U}$ be a free amalgamation homogeneous relational structure with a transitive automorphism group $\mathrm{G}$. Let $\mathfrak{A}$ be the age of $\mathrm{U}$.  Let  $\mathfrak{c}: U\to 2$ be a two colouring of $U$. Then, if  $\rro(\mathrm{U}_{\downarrow (\mathfrak{c}^{-1}(i))})\not=\mathfrak{A}$ the induced substructure    $\mathrm{U}_{\downarrow \mathfrak{c}^{-1}(i+1)}$  is a copy of $\mathrm{U}$. (Hence $\mathrm{U}$ is age indivisible.)

\begin{thm}\label{thm:agindfreambund}
Let $\mathrm{U}$ be a countable homogeneous structure whose age has free amalgamation. Let $\mathcal{B}=\{\B_j\mid j\in n\}$ be a bundle of types of $\mathrm{U}$ and  $S\subseteq \sigma(\mathcal{B})$ with $\rrro(S)=\rrro(\sigma(\mathcal{B}))$. Then there exists for every $m\in \omega$ and every $m$-colouring $P$ of $S$ a function $\epsilon\in \mathcal{E}(P)$ with $\rrro(P_\epsilon\cap S)=\rrro(\sigma(\mathcal{B}))$.

Every set $S\subseteq \sigma(\mathcal{B})$ with $\rrro(S)=\rrro(\sigma(\mathcal{B}))$ is $\mathcal{B}$-age indivisible. 

Every bundle $\mathcal{B}=\{\B_j\mid j\in n\}$ of types of $\mathrm{U}$ is $\mathcal{B}$-age indivisible.

Every type $\T$ of $\mathrm{U}$ is age indivisible.
\end{thm}
\begin{proof}
It follows  from Corollary \ref{cor:hennsage} via a standard compactness argument that for ever finite subset $A$ of $\sigma(\mathcal{B})$ there exists a finite subset $R(A)$ of $\sigma(\mathcal{B})$ so that for every 2-colouring of $\sigma(\mathcal{B})$ there exists  function $\epsilon\in \mathcal{E}(P)$  for which $A\in \rrro(P_\epsilon\cap R(A))$.  Implying, via repeated Ramseying, that for every finite set $A\subseteq \sigma(\mathcal{B})$ and every $m\in \omega$ there exists a finite set $R(A)\subseteq \sigma(\mathcal{B})$ so that for every   $m$-colouring $P$ of $\sigma(\mathcal{B})$ there exists a function $\epsilon\in \mathcal{E}(P)$  for which $A\in \rrro(P_\epsilon\cap R(A))$. Which implies the Theorem. Let $S\subseteq \sigma(\mathcal{B})$ with $\rrro(S)=\rrro(\sigma(\mathcal{B}))$ and $m\in \omega$ and $P$ an $m$-colouring of $S$. Assume for a contradiction that for every function $\epsilon\in \mathcal{E}(P)$ there exists a finite $A_\epsilon\in \rrro(\sigma(\mathcal{B}))$ with $A_\epsilon\subsetneq P_\epsilon$. The set $A=\bigcup_{\epsilon\in \mathcal{E}}A_\epsilon$ is finite because $\mathcal{E}$ is finite, hence $A\in \rrro(S)$. This  contradicts that there exists an $\epsilon\in \mathcal{E}$ for which $A\in  \rrro(P_\epsilon\cap R(A))$. 
\end{proof}

\begin{cor}\label{cor:agindfreambund}
Let $\mathrm{U}$ be a countable homogeneous structure whose age has free amalgamation. Let $\mathcal{B}'=\{\B'_j\mid j\in n\}$ be a bundle of types of $\mathrm{U}$ and let $z\in \sigma(\mathcal{B}')$. Let the bundle  $\mathcal{B}=\{\B_j\mid j\in n\}$ be a refinement of the bundle $\mathcal{B}'=\{\B'_j\mid j\in n\}$ with refinement bijection $'$ so that  for every finite $A\subseteq \sigma((\mathcal{B}')_{\uparrow z})$ there exists a function $h\in \mathrm{G}_{\iota(\mathcal{B}')\cup \{z\}}$ with $h[A]\subseteq \sigma(\mathcal{B}_{\uparrow z})$.  

Then there exists a bundle $\overline{\mathcal{Y}}\subseteq \mathcal{B}_{\uparrow z}$ which is a refinement of the bundle $(\mathcal{B}')_{\uparrow z}$. 
\end{cor}
\begin{proof} 
Then $S:=\sigma(\mathcal{B}_{\uparrow z})\subseteq \sigma((\mathcal{B}')_{\uparrow z})$  and $\rrro(S)= \rrro(\sigma((\mathcal{B}')_{\uparrow z}))$.  For every type $\Y\in (\mathcal{B}')_{\uparrow z}$ let $\{\Y_i\mid i\in m_\Y\in \omega\}$ be the bundle of types in $\mathcal{B}_{\uparrow z}$ with $\sigma(\Y_i)\subseteq \sigma(\Y)$.  For every $i\in m_\Y$ let $P_{i,\Y} =S\cap \sigma(\Y_i)$. Then $P=(P_{i,\Y}; i\in m_\Y, \Y\in (\mathcal{B}')_{\uparrow z})$ is a colouring of $S$. According to Theorem \ref{thm:agindfreambund} there exists a function $\epsilon\in \mathcal{E}(P)$ with $\rrro(P_\epsilon\cap S)=\rrro(\sigma((\mathcal{B}')_{\uparrow z}))$. Let $\overline{\mathcal{Y}}=\{\Y_{\epsilon(i),\Y}\mid \Y\in (\mathcal{B}')_{\uparrow z}\}$. Then the function $\underline{\epsilon}$ which maps $\Y\in (\mathcal{B}')_{\uparrow z}$ to the type $\Y_{\epsilon(i),\Y}\in \overline{\mathcal{Y}}$ is a bijection of  $(\mathcal{B}')_{\uparrow z}$ to $\overline{\mathcal{Y}}$. Also $\rrro(P_\epsilon)=\rrro(\sigma(\mathcal{Y}))$. Hence there exists  for every finite $A\subseteq \sigma((\mathcal{B}')_{\uparrow z})$ a function $f\in \mathrm{G}_{\iota(\mathcal{B}')\cup \{z\}}$ with $f[A]\subseteq \sigma(\overline{\mathcal{Y}})$. Implying that the bundle $\overline{\mathcal{Y}}$ is a refinement of the bundle $(\mathcal{B}')_{\uparrow z}$ with refinement bijection $\underline{\epsilon}^{-1}$. 
\end{proof}

\section{Necessary conditions} \label{sect:necess}

\noindent
Let $\mathrm{G}$ be a subgroup of the symmetric group of a countable infinite set $U$.  Consult Subsection~\ref{subsect:gropversio} for the notion of age indivisible and the notions of $\rro(S)$ and $\rro_\mathrm{G}$.  Note that if a set $S$ is age indivisible then $S$ is infinite unless $|S|=1$.    The following example shows that it is possible that $\mathrm{G}$ is indivisible and has a type $\T$ with $|\sigma(\T)|=1$.   (I believe this example has been orally  communicated  to me by Peter Cameron a long, long time ago.)

\begin{example}\label{ex:redbluetr}
\normalfont{Let $U$ be the set of two-element subsets of $\omega$. Let $\mathrm{G}$ be the action of the symmetric group of $\omega$ on $U$. Let $a$ and $b$ be two two-element subsets of $\omega$ with ${a\cap b}=1$. Let $c$ be the two element subset of $\omega$ with $c=\{a\setminus b, b\setminus a\}$.  Let $\T=\langle \{a,b\}\tr c\rangle$. Then $\sigma(\T)=\{c\}$. It follows from the standard Ramsey theorem on two element subsets  of $\omega$ that the group $\mathrm{G}$ is indivisible.
}\end{example}

The following Lemmata  are well known within the standard interpretation of age via the function $\rro$. That is the process of ``repeated Ramseying" and the connection of the ``finite versions of Ramsey theorems" and the corresponding ``infinite versions" via compactness as in Lemma \ref{lem:immind} are well established results in Ramsey theory.   A version of Corollary \ref{cor:Frageind} appeared already in \cite{Fra}.

\begin{defin}\label{defin:nageind}
Let $n\in \omega$. A set $S\subseteq U$ is {\em $n$-age indivisible} if for every partition $(S_i; i\in n)$ of $S$ there exists an index $i\in n$ with $\rro(S_i)=\rro(S)$.  (Of course as there is no structure $\mathrm{U}$ but only a group $\mathrm{G}$ acting on $U$ the interpretation of $\rro$ has to be as $\rro_\mathrm{G}$.)
\end{defin}

\begin{lem}\label{lem:immindp}
Let $n\in \omega$. A set $S\subseteq U$ is $n$-age indivisible if and only if for every $A\in \rro(S)$ and every partition $(S_i; i\in n)$ of $S$ there exists an index $i\in n$ and  a function  $f\in \mathrm{G}$ with $f[A]\subseteq S_i$. That is $S$ is age indivisible if and only if $\rro(S)=\bigcup_{i\in n}\rro(S_i)$ for all partitions $(S_i;i\in n)$ of $S$.
\end{lem}
\begin{proof}
If not then  there exists for every $i\in n$ a finite set $A_i\subseteq U$ and  functions  $f_i, g_i\in \mathrm{G}$ with $f_i[A]\subseteq S$ and $g_i[A]\subseteq \bigcup_{j\in n\setminus\{i\}}S_j$  but $h[A_i]$ is not a subset  of $S_i$ for every function $h\in \mathrm{G}$. Let $A=\bigcup_{i\in n}g_i[A_i]$.  The set $A$ is a finite subset of $S\subseteq U$ and hence in $\rro(S)$ because the identity function maps it into $S$. According to the assumption of the Lemma there exists an index $i\in n$ and a  function $h\in \mathrm{G}$ with $h[A]\subseteq S_i$. Leading  to the contradiction $h[A_i]\subseteq S_i$. 
\end{proof}

\begin{lem}\label{lem:immind}
Let $n\in \omega$. A set $S\subseteq U$ is $n$-age indivisible if and only if for every set $A\in \rro(S)$ there exists a set $B\in \rro(S)$ so that for every partition $(B_i; i\in n)$ of $B$ there exists an index $i\in n$ for which  $A\in \rro(B_i)$ .  
\end{lem}
\begin{proof}
If for every $A\in \rro(S)$ there exists a $B^{(A)}\in \rro(S)$ so that for every partition $(B^{(A)}_i; i\in n)$ of $B^{(A)}$ there exists an index $i\in n$ for which  $A\in \rro(B_i)$ let $(S_i; i\in n)$ be a partition of $S$.  We may assume without loss that $B^{(A)}\subseteq S$. Let $A\subseteq S$. Then the partition $(S_i; i\in n)$ of $S$ induces a partition of  $B^{(A)}$ implying that $A\subseteq S_i$ for some index $i\in n$. Hence $S$ is $n$-age indivisible according to Lemma  \ref{lem:immindp} 

Let a finite $A\subseteq S$ be such that for every  finite $B\subseteq S$ there exists a partition  $(B_i;i\in n)$ of $B$ with $A\not\in \rro(B_i)$ for any $i\in n$. Let $\{s_j\mid j\in \omega\}$ be an enumeration of $S$.  Then every initial sequence $\{s_j\mid j\in n\in \omega\}$  has a partition $(P_i;i\in n)$ with $A\not\in \bigcup_{i\in n} \rro(P_i)$.  Using K\"{o}nigs Lemma there exists a partition of $S$ into $n$ parts for which $A$ is not a subset of any part of this partition. 
\end{proof}
Hence:
\begin{cor}\label{cor:Frageind}
Let $S$ and $T$ be two subsets of $U$ with $\rro(S)=\rro(T)$ and $n\in \omega$. Then $S$ is $n$-age indivisible if and only if $T$ is $n$-age indivisible. (That is, being $n$-age indivisible is a property of the age.)
\end{cor}

\begin{lem}\label{lem:immindpp3}
Let  $S\subseteq U$ be age indivisible.  Let $(S_i; i\in n\in \omega)$ be a partition of $S$ into $n$ parts. Then there exists an index $i\in n$ with $\rro(S_i)=\rro(S)$. 
\end{lem}
\begin{proof}
Induction on $n$. For  $n=1$ the Lemma holds. From $n$ to $n+1$: Let $(S_i; i\in n+1)$ be a partition of $S$ into $n+1$ parts. Let $R_0=\bigcup_{i\in n}S_i$ and $R_1=S_n$. If $\rro(R_1)=\rro(S)$ we are done. Otherwise $\rro(R_0)=\rro(S)$. Hence, using Corollary \ref{cor:Frageind} there exists an $i\in n$ with $\rro(S_i)=\rro(S)$. 
\end{proof}

\subsection{Proof of Theorem \ref{them:necesssary}} We are now ready to prove: 

\vskip 3pt
\noindent
Let $\mathrm{G}$ be a subgroup of the symmetric group of a countable infinite set $U$. If $\mathrm{G}$ contains two age indivisible types $\T$ and $\S$ having infinite typesets such that $\rro(\T)\setminus \rro(\S)\not=\emptyset$ and $\rro(\S)\setminus \rro(\T)\not=\emptyset$ then the group $\mathrm{G}$ is divisible.

\begin{proof}
Assume that $U=\omega$. For two finite subsets $X$ and $Y$ of $\omega$ let $X\prec Y$ if  $\max\{(X\cup Y)\setminus(X\cap Y)\}\in Y$. This lexicographic order of the finite subsets of $\omega$ is a total order.

Let $\T=\langle F;x\rangle$ and $\S=\langle E;y\rangle$ and let $A\in \rro(\T)\setminus\rro(\S)$ and $B\in \rro(\S)\setminus\rro(\T)$. Note that the sets  $\sigma(\T)$ and $\sigma(\S)$ are infinite.  Then let:
\begin{enumerate}
\item $P$ be the set of elements $n\in \omega$ for which there exists a function $f\in \mathrm{G}$ with $f[F]\prec \{n\}$ and $f(x)=n$ and so that: If there exists a function $g\in \mathrm{G}$ with $g[E]\prec \{n\}$ and $g(y)=n$ then $f[F]\prec g[E]$. 
\item $Q$  be the set of elements $n\in \omega$ for which there exists a function $f\in \mathrm{G}$ with $f[E]\prec \{n\}$ and $f(y)=n$ and so that: If there exists a function $g\in \mathrm{G}$ with $g[F]\prec \{n\}$ and $g(x)=n$ then $f[E]\prec g[F]$. 
\item $R$ be the set of elements $n\in\omega$ with $n\not\in P\cup Q$. 
\end{enumerate}
Then $(P,Q,R)$ is a partition of $\omega$.  For if $n\in P\cap Q$ there is a function $h\in \mathrm{G}$ with $h[E]=F$ and $h(y)=x$. Implying that $\rro(\T)=\rro(\S)$. It is not possible that there exists an embedding $k$ of $\mathrm{G}$ mapping $\omega$ into $R$. Because $\sigma(\T)$ is infinite and hence there is a $z\in \sigma(\T)$ with $k[F]\prec \{k(z)\}$ and $F\prec \{z\}$. Implying, because there is a function $g\in\mathrm{G}$ which agrees with $k$ on the set $\{s\in \omega\mid s\leq z\}$, that $g[F]\prec \{g(z)\}$ and that  $g(z)=k(z)\in P\cup Q$.  

Assume that there is an embedding $k$ of $\mathrm{G}$ mapping  $\omega$ into $Q$. Note that if there exists a $z\in \sigma(k[\T])$ with $k[F]\prec \{z\}$ then, because there exists a function $f\in \mathrm{G}$ with $f[F]=k[F]\prec \{f(x)\}=\{z\}$,    there exists a function $g\in \mathrm{G}$ with $z\in \sigma\langle g[E]\tr g(y)\rangle$ and with $g[E]\prec k[F]$.  Otherwise $z\in P$ contradicting $z\in \sigma(k[\T])\subseteq Q$. Let $\mathcal{L}$ be the set of types $\L=\langle L\tr z\rangle$ with $L\prec k[F]$ and for which there exists a function $g\in \mathrm{G}$ with $g\langle E\tr y\rangle=\langle L\tr z \rangle$. Because $L\prec k[F]$ and because the set of functions $L\to L$ is finite,   the set of types in $\mathcal{L}$ is finite. For every type $\L\in \mathcal{L}$ let $g_{\L}\in \mathrm{G}$ be a function with $g_\L(\S)=\L$.  Let $S=\{z\in \sigma\langle k[F]\tr k(x)\rangle\mid \{z\}\prec k[F]\}$. The set $S$ is finite. 

For every $z\in (\sigma\langle k[F]\tr k(x)\rangle)\setminus S$ there exists a function $g_\L\in \mathrm{G}$ with $z\in \sigma(g_\L\langle E\tr y\rangle)$.  Hence, for every $v\in \sigma(\T)$ either $k(v)$ is an element of the finite set $S$ or there is a type $\L\in \mathcal{L}$ so that $k(v)\in \sigma(g_\L\langle E\tr y\rangle)$. Because $\sigma(\T)$ is age indivisible it follows from Lemma~\ref{lem:immindpp3} that there exist a function $h\in \mathrm{G}$ and a type $\L\in \mathcal{L}$ for which $k\circ h[A]\subseteq \sigma(g_\L\langle E\tr y\rangle)=\sigma(g_\L(\S))$. Hence $g_\L^{-1}\circ k\circ h[A]\subseteq \sigma(\S)$ in contradiction to $A\not\in \rro(\S)$. 

The argument then  for no embedding $f$ of $\mathrm{G}$ mapping $\omega$ into $P$ is analogous.

\end{proof}

Theorem  \ref{them:necesssary} implies Theorem \ref{thm:maincecess}, which is:

\vskip 3pt
\noindent
Let the age $\mathfrak{A}$ of the homogeneous structure $\mathrm{U}$ be a  free amalgamation class. If $\mathrm{U}$ is not rank linear then $\mathrm{U}$ is divisible.

\begin{proof}
It follows from Theorem \ref{thm:agindfreambund} that every type $\T$ of $\mathrm{U}$ is age indivisible and from Corollary \ref{cor:tyinfA} that it has an infinite typeset. Hence, if the partial order $(\mathfrak{R},\subseteq)$ of ranks of types of $\mathrm{U}$ is not linearly ordered then there are two age indivisible types $\T$ and $\S$ of $\mathrm{U}$ and hence of the automorphism group  $\mathrm{G}$ of $\mathrm{U}$ with $\rro(\T)\setminus \rro(\S)\not=\emptyset$ and $\rro(\S)\setminus \rro(\T)\not=\emptyset$. It follows from Theorem \ref{them:necesssary} and Lemma \ref{lem:ageintypp} that $\mathrm{G}$ and hence $\mathrm{U}$ is divisible.

\end{proof}

\section{Examples}\label{sect:examples}

\begin{example}\label{ex:1} \normalfont{Let $\boldsymbol{L}\not=\emptyset$ be a relational language with only finitely many relation symbols of any given arity. Implying that for any subset $\boldsymbol{L}'\subseteq \boldsymbol{L}$ a countable  homogeneous structure in language $\boldsymbol{L}'$ is oligomorphic.  The class of all finite $\boldsymbol{L}$-structures has free amalgamation. But the homogeneous structure whose age is the class of all finite $\boldsymbol{L}$-structures is divisible because the interpretation of  relations in $\boldsymbol{L}$ might be reflexive. Hence the partial order  $(\mathfrak{R}, \subseteq)$ of ranks of the types of the homogeneous structure is not linear. Let $\mathfrak{L}$ be the class of all finite $\boldsymbol{L}$-structures for which  in every interpretation and for every relation symbol $R\in \boldsymbol{L}$: $R(x_0,x_1,\dots,x_{n-1})$ implies that $x_i\not=x_j$ for all $\{i,j\}\subseteq n$ with $i\not=j$. As stipulated in Section \ref{sect:prelim}.  The class $\mathfrak{L}$ is a free amalgamation class. Let $\mathrm{U}$ be the countable homogeneous structure with age $\mathfrak{L}$.  The boundary consists of structures    $\mathrm{B}$ for which there exists  a relation $R$ in the language $\boldsymbol{L}$  for which $R_\mathrm{B}(\vec{x})$ and  so that every element of $B$ is an entry of $\vec{x}$. Implying that $\mathrm{U}$ satisfies the conditions of Item (1) in Lemma \ref{lem:singlrant}. Hence $\mathrm{U}$ is indivisible.
}\end{example}

The next example is a small but not trivial example of a free amalgamation homogeneous structure for which the partial order of ranks of types is not a linear order.

\begin{example}\label{ex:twtriange}
\normalfont{Let the relational language $\boldsymbol{L}$ have two binary relation symbols $R$ and $B$. Let   $\mathfrak{A}$ be the class of all finite $\boldsymbol{L}$-structures $\mathrm{A}$ for which $R_\mathrm{A}$ forms the edges of a triangle free simple graph and for which $B_{\mathrm{A}}$ forms the edges of a triangle free simple graph and for which $R_\mathrm{A}(x,y)$ implies $\neg B_\mathrm{A}(x,y)$. Let $\mathrm{U}$ be the homogeneous graph whose age is $\mathfrak{A}$. Let $u\in U$ and let $\T=\langle \{u\}\tr x\rangle$ with $R_\mathrm{U}(u,x)$ and let $\S=\langle u\tr y\rangle$ with $B_\mathrm{U}(u,x)$. Then $\sigma(\T)$ does not contain an edge of the form $R_\mathrm{A}$ and $\sigma(\S)$ does not contain an edge of the form $B_\mathrm{A}$. Hence $\mathrm{U}$ is not rank linear and therefore divisible. On the other hand it follows from  \cite{SaCan} that for every colouring function $\mathfrak{c}: U\to n\in \omega$ there exists a copy $C$ of $\mathrm{U}$ for which $|\mathfrak{c}[C]|\leq 2$. 
}\end{example}

Theorem \ref{thm:main} can effectively be used to determine for a given finite  boundary of irreducible structures whether the corresponding homogeneous structure is indivisible. But it could require a fair bit of work.  Which the next example  for just 3-uniform hypergraphs on at most five verices shows.

\begin{example}\label{ex:nelinenine5}
\normalfont{ Let $\mathfrak{B}$ be a boundary which consists of irreducible 3-uniform hypergraphs on at most five vertices. Let then  $\mathfrak{A}$ be the free amalgamation  age  of finite 3-uniform hypergraphs whose boundary is  $\mathfrak{B}$ and let $\mathrm{U}$ be the countable homogeneous 3-uniform hypergraph with age $\mathfrak{A}$. Let $\mathrm{K}$ be the irreducible  3-uniform hypergraph on five vertices and four hyperedges.  This structure is unique up to isomorphisms.  

If $\mathfrak{B}$ contains a structure with not more than two verices the structure $\mathrm{U}$ does not exist. If $\mathcal{B}$ contains the structure on three vertices and one hypredge then $\mathrm{U}$ is the countable infinite  three uniform hypergraph having no edges. If $\mathcal{B}$ contains the structure on three vertices and no edge then $\mathrm{U}$ is the countable infinite  complete three uniform hypergraph. 

According to the discussion after Definition \ref{defin:onfyp} the set $\mathcal{B}$ contains one or both of the two  irreducible 3-uniform hypergraphs having four vertices then the age of $\mathrm{U}$ does not contain a hypergraph embedding those structures but otherwise the ranks of the typesets are not influenced. The only way to have a boundary element $\mathrm{B}$ with a conformal subset $A$ and $|B\setminus A|\geq 2$ is if $\mathrm{B}$ contains a monomorphic copy of the hypergraph $\mathrm{K}$. Let $\mathfrak{C}$ be the set of all five element three uniform hypergraphs which contain a monomorphic copy of $\mathrm{K}$. If $\mathfrak{B}\supseteq \mathfrak{C}$ let $\{x,y,z\}$ be an edge in $\mathrm{U}$ and $\T=\langle \{z,y\}\tr x\rangle$. Then $\sigma(\T)$ does not contain a hyperedge. If $\mathfrak{C}\cap \mathfrak{B}=\emptyset$ then $\rro(\T)$ is equal to the age of $\mathrm{U}$ for every type $\T$. 

Otherwise $\mathrm{K}\in \mathfrak{B}$ and  $\mathfrak{B}\setminus \mathfrak{C}\not=\emptyset$. Let  $\T=\langle T=\langle F\tr x\rangle$ be a type of $\mathrm{U}$. If $T$ does not contain a pair $\{a,b\}$ for which $\{a,b,x\}$ is a hyperedge of $\mathrm{U}$ then $\rro(\T)$ is equal to the age of $\mathrm{U}$. Let $\mathrm{M}$ be the structure with $M=F\cup \{u,v,w\}$ for which $\{u,v,w\}$ forms a hyperedge of $\mathrm{M}$ and $F\cap \{u,v,w\}=\emptyset$. $\mathrm{M}_{\downarrow F}=\mathrm{U}_{\downarrow F}$. If it is possible to add hyperedges with one vertex in $F$ and the other two in $\{u,v,w\}$, obtaining the structure $\mathrm{M}'$,   in such a way that for every pair $\{a,b\}\in F$ for which $\{a,b,x\}$ is a hyperedge of $\mathrm{U}$ the structure $\mathrm{M}'_{\downarrow\{a,b,u,v,w\}}$ is in $\mathfrak{B}\setminus \mathfrak{C}$ then $\mathrm{M}'$ is in the age of $\mathrm{U}$. Then $\rro(\T)$ is equal to the age of $\U$. Otherwise $\sigma(\T)$ does not contain an edge. To determine the cases in which such a structure $\mathrm{M}'$ can actually be found will require quite a bit more work. In either case the rank of $\T$ is equal to the age of $\mathrm{U}$ or $\sigma(\T)$ does not contain a triangle. 

We conclude that the rank of a type $\T$ of $\mathrm{U}$ is either equal to the age of $\mathrm{U}$ or the typeset of $\T$ does not contain a hyperedge. Hence $|\mathfrak{R}|=1$ and or  $|\mathfrak{R}|=2$. That is rank linear in both cases. We obtained:

\begin{lem}\label{lem:3hyper6}
If\/  $\mathrm{U}$ is a countable infinite  homogeneous  3-uniform hypergraph whose boundary consist of hypergraphs on at most five vertices, then $\mathrm{U}$ is indivisible. 

\end{lem}

}\end{example}

For the next example see Ramark \ref{remark:meld} and the Lemma preceding it.

\begin{example}\label{ex:strmelmeld}\normalfont{
Let $\mathfrak{L}$ be the class of relational structures with one binary relation and two ternary relations. For any $\mathfrak{L}$-structure the binary relation forms a simple graph whose edges will be called {\em graph edges}. Both ternary relations form 3-uniform hypergraphs, the hyperedges of one will be called {\em blue hyperedges} and the hyperedges of the other {\em red hyperedges}.  In addition, if $\mathrm{L}\in \mathfrak{L}$ then no graph edge of $\mathrm{L}$ is a subset of any blue hypergraph edge but may be a subset of a red hypergraph edge. No blue hypergraph edge occupies the same set of vertices as any of the red hypergraph edges.    Let $\mathrm{K}\in \mathfrak{L}$ be the irreducible  3-uniform hypergraph on five vertices and four blue hyperedges and no graph edges and no red hypergraph edges. This structure is unique up to isomorphisms.   Let $\mathrm{M}$ be the $\mathfrak{L}$-structure with $M=\{x_0,x_1,x_2,x_3\}$.  The structure $\mathrm{M}$ has the two blue hyperedges $\{x_0,x_1, x_2\}$ and $\{x_0,x_1,x_3\}$ together with the graph edge $\{x_2,x_3\}$. Then $\mathrm{M}$ is irreducible. The structure $\mathrm{N}$ is obtained from the structure $\mathrm{M}$ by augmenting it with two additional red hypergraph edges which have the graph edge of $\mathrm{M}$ as a subset.  Let $\mathfrak{A}\subseteq \mathfrak{L}$ be the class of structures which do not contain a monomorphic copy of $\mathrm{M}$ except for the structure $\mathrm{N}$  and which do not contain a monomorphic copy of $\mathrm{K}$.   The class $\mathfrak{A}$ is a free amalgamation age.  Let $\mathrm{U}$ be the countable homogeneous structure whose age is $\mathfrak{A}$. Let $\mathrm{G}$ be the group of automorphisms of $\mathrm{U}$. 

 Let  $\T=\langle T\tr x\rangle$ be a type of $\mathrm{U}$. If $T$ does not contain a pair $\{a,b\}$ for which $\{a,b,x\}$ is a blue hyperedge  then $\rro(\T)$ is equal to the age of $\mathrm{U}$. Let $\{a,b\}\subseteq F$ for which $\{a,b,x\}$ is a blue hyperedge. Then $\sigma(\T)$ does not contain a blue hyperedge, but does contain a graph edge, say $\{u,v\}$. The red hyperedges $\{a,u,v\}$ and $\{b,u,v\}$ can be added to obtain the structure $\mathrm{N}$ not in the boundary.    Actually then $\mathrm{U}_{\downarrow \sigma(\T)}$ is isomorphic to the Rado graph.
 
 Let $F=\{a,b\}\subseteq U$ with $a\, \, \, \, \nedge{S}\, \, \, \, b$. For  $b\, \, \, \, \nedge{S}\, \, \, \, x $ let $\B$ be the type $\langle \{b\}\tr x\rangle$. Then $\rro(\B)=\rro(\U)$. Let $\C$ be the type $\langle \{a,b\}\tr y\rangle$ for which $\{a,b,y\}$ forms a blue hyperedge. Then $\mathrm{U}_{\downarrow \sigma(\T)}$ is isomorphic to the Rado graph. Let $\{u,v\}$ be a graph edge in $\sigma(\B)$ for which $u\, \, \, \, \nedge{S}\, \, \, \, b\, \, \, \, \nedge{S}\, \, \, \, v$ with $S=\{b,u,v\}$. There is no function $g\in \mathrm{G}_{\{b\}}$ because then the structure $\mathrm{U}_{\downarrow \{a,b, g(u),g(v)\}}$ would be a monomorphic copy of $\mathrm{M}$ to which a red hyperedge with set of  vertices $\{b,g(u),g(v)\}$ could not be added.

}\end{example}

Irreducible structures in a given language do in general not have the property of complete structures that the ones on a smaller set of elements can be embedded into the ones on a larger set of elements. Even in the case in which for every structure in the boundary all of the monomorphic copies of it are also in the boundary. This necessitates in Theorem \ref{thm:Knfrreegen}  to have an infinite boundary.

\begin{example}\label{ex:nelinenine}
\normalfont{Let $\mathfrak{B}$ be the class of all irreducible 3-uniform hypergraphs on exactly nine vertices. 
Let $\mathfrak{A}$ be the age of all finite 3-uniform hypergraphs which do not embed any one of the hypergraphs in $\mathfrak{B}$. The age $\mathfrak{A}$ has free amalgamation. Let $\mathrm{U}$ be the countable homogeneous structure whose age is $\mathfrak{A}$.  Let $\mathrm{A}$ be an irreducible 3-uniform hypergraph with five vertices and a set $E$ of  five hyperedges, such that the set of their complements forms a (graph) pentagon,  denoted $\overline{\mathrm{A}}$.  Note that  the structure obtained from $\mathrm{A}$ by removing a vertex is not irreducible.  Because $\mathrm{A}\in \mathfrak{A}$ we may assume that  the set of vertices $A$ of $\mathrm{A}$ is a subset of the set $U$ of vertices of $\mathrm{U}$.  Let $\T=\langle A\tr x\rangle$ be the type of $\mathrm{U}$ for which $\mathrm{U}_{\downarrow A}$ is  $\mathrm{A}$ and for which for $\{a,b\}\subseteq A$ the set $\{x,a,b\}$ is a hyperedge of $\mathrm{U}$ if and only if $\{a,b\}$ forms an edge of the pentagon $\overline{\mathrm{A}}$. Then $\rro(\T)$ contains a copy of $\mathrm{A}$ but does not contain a copy of the 3-uniform hypergraph $\mathrm{B}$ on four vertices and having four hyperedges. Let $\S=\langle S\tr y\rangle $ be the type of $\mathrm{U}$ for which $\mathrm{U}_{\downarrow S}$ is isomorphic to $\mathrm{B}$ and for which $\{y,c,d\}$ is a hyperedge of $\mathrm{U}$ for all $\{c,d\}\subseteq S$ with $c\not=d$.  Then $\rro(\S)$ does not contain a copy of $\mathrm{A}$ but does contain a copy of $\mathrm{B}$. It follows from Theorem \ref{thm:maincecess} that the homogeneous 3-uniform hypergraph $\mathrm{U}$ is divisible. 
}\end{example}

The partial order of the ranks of the typesets of a homogeneous structure has a maximum, the age of the structure. Example \ref{ex:strmelmeld} shows that the linear order of the ranks of a free amalgamation homogeneous structure can be the order of the rationals with a maximum added. It should be clear that the example can be adapted to show that any countable linear order with a maximum can be the linear order of the ranks of some homogeneous structure.  To adapt the example for the non binary case is possible but requires a bit more work.

\begin{example}\label{ex:strmelmeld}\normalfont{

Let $\boldsymbol{L}$ be the relational language with two binary, symmetric simple graph relations $E$ and $F$. As well as with a binary  oriented graph relation $O$. For $5\leq n\in\omega$ let $\mathrm{A}(n)$ be the $\boldsymbol{L}$-structure on $n$ vertices for which the $E$-graph forms a cycle and all of the other two element subsets are $F$-edges. $\mathrm{B}(n,m)$ is an $\boldsymbol{L}$-structure with $n+m$ vertices  which consists of an $\mathrm{A}(n)$-structure and an $\mathrm{A}(m)$-structure.  There exists a exactly one vertex $v$ in the $\mathrm{A}(m)$-structure of $\mathrm{B}(n,m)$ so that: 
\begin{enumerate}
\item For every vertex $x\in A(n)$ and every vertex $v\not=y\in A(m)$ exists an oriented $O$-edge from $x\to y$. 
\item For every vertex $x\in A(n)$ exists an oriented $O$-edge from $v(\to x)$. 
\end{enumerate}
Note that each of the structures $\mathrm{B}(n,m)$ is irreducible. The set $A(n)$ is a conformal subset of $\mathrm{B}(n,m)$ and it is the only conformal subset of $\mathrm{B}(n,m)$. (See Definition \ref{defin:onfyp}.)
Let $\mathrm{B}'(m)$ be an $\boldsymbol{L}$-structure with $m+1$ vertices containing an $\mathrm{A}(m)$-structure and one additional vertex $x$. There exists a vertex $v\in A(m)$ so that: For every vertex $v\not=y\in A(m)$ exists an oriented $O$-edge from $x\to y$. There exists an oriented $O$-edge from $v$ to $x$. 

Let $\alpha$ be a bijection from $\{n\mid 5\leq n\in \omega\}$ to the set of rational numbers. A structure $\mathrm{B}(n,m)$ is an element of the boundary if $\alpha(n)\geq \alpha(m)$. Let $\mathrm{U}$ be the homogeneous structure in language $\boldsymbol{L}$ having this boundary. 

Let $\T=\langle T\tr x\rangle$ be a type of $\mathrm{U}$. Let $r$ be the smallest rational for which there exists a structure $\mathrm{A}(m)$ embedded into $T$ with $\alpha(m)=r$ and for which $\mathrm{A}(m)$ together with $x$ induces a $\mathrm{B}'(m)$ structure in $\mathrm{U}$. The age of $\mathrm{U}$ is then the class of finite $\boldsymbol{L}$-structures which do not embed an $\mathrm{A}(n)$ with $\alpha(n)\geq r$.  If such a structure $\mathrm{A}(m)$ does not exists then the age of $\sigma(\T)$ is equal to the age of $\mathrm{U}$.  Then $\mathrm{U}$ is a free amalgamation homegeneous relational structure for which the $\subseteq$-order of the set of ranks of its types is order isomorphic to the rationals together with a maximum. 
}\end{example}


\begin{thebibliography}{LANVT}





\bibitem{Fra} R. \Fra, {\em Theory of Relations}, Studies in Logic and the Foundations of Mathematics, {\bf 145}, North-Holland Publishing Co., Amsterdam (2000).
\bibitem{Fraisse}
R.Fra\"\i ss\'e,  {\em Theory of relations, studies in logic and foundations of mathematics}.  118 (1986), Elsevier Science Publishing Co., Inc., U.S.A.
\bibitem{Hodges}
W. Hodges, {\em Model Theory}. Cambridge University Press, June 2008.  


\bibitem{Henson}
W. Henson, {\em Countable homogeneous relational structures and $\aleph_0$-categorical theories}, Journal of Symbolic Logic, \textbf(37), Issue 3 (1972),     494-500 


\bibitem{KomRo} P. Komj\'ath, V. R\"{o}dl, V.  {\em Coloring of universal graphs}, Graphs and  Combinatorics {\textbf 2}, Issue 1 (1986) 55-60. 

\bibitem{EZS1}
 M. El-Zahar, N. Sauer,  {\em The Indivisibility of the Homogeneous  $K_n$-free graphs},   Journal of Combinatorial Theory, Series
$B$, {\bf 47} (1989), no. 2, 162-170.

\bibitem{EZS2}
 M. El-Zahar, N. Sauer,  {\em On the divisibility of homogeneous directed graphs}, Candian J. Math. (\textbf{45}) (2), 1993 pp. 284-294. 
 
 \bibitem{EZS3}
  M. El-Zahar, N. Sauer,  {\em Ramsey-type properties of relational structures}, Discrete Mathematics, \textbf{94} (1991)  1-10. North-Holland.  

\bibitem{SaCan}
N.Sauer {\em Canonical vertex partitions}, Combinatorics Probability 
and Computing, \textbf{12} (2003), issue 5-6, pp 671-704.


\bibitem{ZahSauer1}M. EL-Zahar, N.W. Sauer, {\em On the Divisibility of Homogeneous Hypergraphs.}  Combinatorica {\bf 14} (2) (1994) 1-7. 

\bibitem{Dobrinen} Natasha Dobrinen, {\em The Ramsey theory of Henson graphs}, arXiv: 1901.06660v2 [math.CO] 27 Apr 2019. 

\bibitem{Honza} Jan Hubi\v{c}ka, Jaroslav Ne\v{s}et\v{r}il, {\em All those Ramsey classes}, arXiv: 1606.07979v4 [math.CO] 25 Jun 2016. 

\bibitem{Zucker} Andy Zucker, {\em Big Ramsey degrees and topological dynamics}, Groups, Geometry and Dynamics (2018), 235-276.  


\bibitem{Cameronage}
P. J. Cameron, {\em The age of a relational structur}, Directions in Infinite
Graph Theory and Combinatorics (ed. R. Diestel), Topic in Discrete Math.  \textbf{3},
49-67, North-Holland, Amsterdam, 1992. 

\bibitem{Cameronoligo}
P. J. Cameron, {\em Oligomorphic Permutation Groups}, London Math. Soc. Lecture Notes \textbf{152}, Cambridge Univ. Press, 1990.


\bibitem{Cameronrand}
P.J. Cameron, {\em The random graph},  Algorithms and Combinatorics, Springer, New York, 1997, \textbf{14}, 333-351. 

\bibitem{Cameronpigeon}
P.J. Cameron, {\em Generalized Pigeonhole Properties of Graphs and Oriented Graphs}, Europ. J. Combinatorics (2002) \textbf{23}, 257-274. 


\bibitem{Pos-Cop}
C. Laflamme, M. Pouzet, N. Sauer, R. Woodrow, {\em The poset of copies for automorphism groups of countable relational structures}, submitted to the Rosenberg Special Issue. It is available on arXiv: arXiv:2002.04771v1 [math.CO]. 

\bibitem{Siblings}
C. Laflamme, M. Pouzet, N. Sauer, R. Woodrow, {\em Siblings of an $\aleph_0$-categorical relational structure}, Submitted to Banff proceedings, editor: Lionel Nguyen Van Th\'{e}. It is available on arXiv: arXiv:1811.04185v2 [math.LO]. 





\end{thebibliography}
\end{document}